\documentclass[DIV=14]{scrartcl}
\usepackage[sb]{libertine} %
\usepackage[T1]{fontenc}
\usepackage{textcomp}
\usepackage{multirow}
\usepackage[varqu,varl]{zi4}%
\usepackage[amsthm,upint]{libertinust1math} %
\usepackage[scr=boondoxo,bb=boondox]{mathalpha} %
\usepackage{bm}
\usepackage{mathtools,amssymb}
\usepackage{todonotes}
\usepackage{booktabs}
\usepackage{enumitem}
\usepackage{placeins}
\usepackage{colortbl}
\usepackage{tikz-3dplot}
\usepackage{bm}
\usepackage[style=numeric, giveninits]{biblatex}
\AtEveryBibitem{%
  \clearfield{urlyear}%
  \clearfield{urldate}%
}%
\AtEveryCitekey{%
  \clearfield{urlyear}%
  \clearfield{urldate}%
}%
\usepackage{xcolor}
\usepackage{float}
\RequirePackage{siunitx}
\usepackage[noend]{algorithmic}
\RequirePackage[algoruled,norelsize, linesnumbered,
lined,commentsnumbered]{algorithm2e}
\usepackage[format=plain,labelfont=bf]{caption}
\usepackage{subcaption}
\definecolor{linkblue}{RGB}{0 83 139}
\definecolor{citered}{RGB}{114 16 69}
\usepackage[pdftex,colorlinks=true,linkcolor={linkblue},citecolor={citered},urlcolor=blue]{hyperref}
\RequirePackage{cleveref}
\usepackage{mycommands}
\usepackage{pgfplots}
\newcommand{\mc}[1]{\multicolumn{1}{c}{#1}}

\usepgfplotslibrary{groupplots,colorbrewer,fillbetween}
\pgfplotsset{compat=newest}%
\pgfplotsset{ colormap/Set1-4, cycle multiindex* list={ mark
    list*\nextlist Set1-4\nextlist }, every axis/.append style = {thick},%
}%
\pgfkeys{/pgf/number format/.cd,1000 sep={\,}}
\usepgfplotslibrary{colormaps}
\pgfplotsset{ every mark/.append style={4pt},tick style = {thick,black}}%

\newcommand{\symbf}[1]{\boldsymbol{#1}}

\newcommand\restrict[1]{\raisebox{-.5ex}{$|$}_{#1}}

\definecolor{myblue}{RGB}{35 71 189}
\definecolor{myred}{RGB}{142 5 0}
\definecolor{mygreen}{RGB}{41 175 52}
\theoremstyle{plain} %
\newtheorem{theorem}{Theorem}[section]

\newtheorem{lemma}[theorem]{Lemma}
\newtheorem{corollary}[theorem]{Corollary}
\newtheorem{remark}[theorem]{Remark}
\newtheorem{definition}[theorem]{Definition}

\newtheorem{problem}[theorem]{Problem}

\theoremstyle{definition} %

\theoremstyle{remark} %

\newcommand{\A}[1]{{N({#1})}}
\newcommand{\AZ}[2]{{N_{#1}^{({#2})}}}

\newcommand{\Riso}[1]{\mathrm{R}^{#1}_{h}}
\newcommand{\Rdir}[2]{\mathrm{R}^{#1}_{h,#2}}
\newcommand{\innerpart}{b_{\boldsymbol{\alpha}}}
\allowdisplaybreaks
\setlist[description]{%
  font={\normalfont}, %
}

\author{
  Nils Margenberg
  \thanks{
    University of Magdeburg,
    Institute for Analysis and Numerics,
    Universit\"atsplatz 2,
    39104 Magdeburg,
    Germany,
    \texttt{nils.margenberg@ovgu.de}, (Corresponding Author)
  }
  \and
  Marius Paul Bruchhäuser
  \thanks{Helmut Schmidt University,
    Faculty of Mechanical and Civil Engineering,
    Holstenhofweg 85,
    22043 Hamburg,
    Germany,
    {\texttt{bruchhaeuser@hsu-hh.de}}}
  \and
  Bernhard Endtmayer\thanks{Leibniz University Hannover, Institute for Applied Mathematics, Welfengarten 1, 30167 Hannover, Germany, \texttt{endtmayer@ifam.uni-hannover.de}}
  \thanks{Cluster of Excellence PhoenixD (Photonics, Optics, and
    Engineering -- Innovation Across Disciplines), Leibniz University Hannover, Germany}}
\date{}
\title{Anisotropic $hp$ space-time adaptivity and goal-oriented error control
  for convection-dominated problems}

\begin{document}
\maketitle
\vspace{-8ex}%
\begin{abstract}
  We present an anisotropic goal-oriented error estimator based on the Dual
  Weighted Residual (DWR) method for time-dependent convection-dominated problems.
  Using elementwise $p$-anisotropic finite element spaces, the estimator is
  elementwise separated with respect to the single directions in space and time.
  This naturally leads to adaptive, anisotropic $hp$-refinement  ($h$-anisotropic refinement and elementwise anisotropic $p$-enrichment).
  We employ discontinuous elements in space and time, which are well suited for
  problems with high P\'{e}clet numbers.
  Efficiency and robustness of the underlying algorithm are demonstrated for
  different goal functionals.
  The directional error indicators quantify anisotropy of the solution with
  respect to the goal, and produce $hp$-refinements that efficiently capture sharp layers.
  Numerical examples in up to three spatial dimensions demonstrate the superior performance of the proposed method compared to isotropic $h$ and $hp$ adaptive refinement using
  established benchmarks for convection-dominated transport.
  

\end{abstract}
\section{Introduction}

Goal-oriented adaptivity and anisotropic mesh refinement are established tools in the numerical solution of partial differential equations (PDEs).
While each of these techniques has been studied extensively, their combination in the context of convection-dominated, time-dependent problems remains an area of active research.
Convection-dominated transport problems are characterized by sharp layers and fronts that benefit greatly from anisotropic \textit{hp}-adaptive, goal-oriented refinement.
In particular, the development of directional error indicators offers new opportunities in space-time $hp$-adaptivity for further improvement with regard to accurate and efficient simulations.
This work contributes to this development by building a framework that performs $hp$-anisotropic space-time refinement driven by goal-oriented error control.

The concept of \textit{hp}-adaptive finite elements was pioneered by Babuška et al.~\cite{babuska_error_1981,gui-babuska-hp-1986,guo_h-p_1986,guo_h-p_1986-1,babuska_p_1994}.
The theoretical foundation was laid in the 1980s, where it was demonstrated that combining mesh refinement ($h$-refinement) and increasing polynomial order ($p$-enrichment) can achieve exponential convergence rates on so-called geometrically graded meshes~\cite{gui-babuska-hp-1986,guo_h-p_1986,babuska_approximation_1996}.
First versions of a \textit{hp}-adaptive code were introduced by Devloo~\cite{devloo-hp-2000} as well as Ainsworth and Senior~\cite{ainsworth_aspects_1997} in the late 1990s and later on by Demkowicz et al.\ introducing a fully automatic \textit{hp}-adaptive strategy for elliptic problems~\cite{demkowicz_fully_2002}.
For a review of different methods to decide between $h$ and $p$ refinement, we refer to~\cite{houston_note_2005}.
For a short review to the development of implementing \textit{hp} elements we refer to the monograph~\cite{demkowicz_computing_2006}. Moreover, the monograph of Schwab provides a treatment of \textit{p}/\textit{hp}-FEM theory~\cite{Schwab1998_hpFEM}.

For anisotropic refinement, Apel's analysis of anisotropic interpolation error estimates~\cite{apel_anisotropic_1999} laid the groundwork for stretched meshes aligned with layers.
Picasso was an early contributor to apply an adaptive algorithm with anisotropic spatial and temporal refinement for time-dependent problems~\cite{picasso_adaptive_1998}.
Since then, several approaches for the directional refinement decisions have been proposed and are well established, cf., e.g., \cite{formaggia_anisotropic_2001,venditti_anisotropic_2003,formaggia_anisotropic_2004,leicht_error_2010,alauzet_anisotropic_2012,carpio2013}.
For a short overview of anisotropic refinement methods in finite elements, we refer to the review~\cite{schneider-aniso-2013}.

At the turn of the last millennium, Becker and Rannacher introduced the Dual Weighted Residual (DWR) method~\cite{BeckerRannacher1995,BeRa01},
which uses the solution of an adjoint problem to weight local residuals and estimates errors in a specific quantity of interest.
At the same time, Oden and Prudhomme formulated a similar framework~\cite{oden_goal-oriented_2001}.
These works established the principle of goal-oriented adaptivity for elliptic PDEs, whereby
Becker and Rannacher's work~\cite{BeRa01} became the foundational reference for DWR-based adaptivity.
Šolín and Demkowicz extended DWR to drive both $h$ and $p$ refinement decisions in elliptic problems~\cite{solin_goal-oriented_2004}.
In the last two decades, the DWR method was applied to many mathematical models based on PDEs, including fluid mechanics, wave propagation, structual mechanics, eigenvalue problems and many others, cf., e.g., \cite{rannacher-elasto-1999,BruchhaeuserSV08,bangerth-wave-2010,richter-fsi-2012,BruchhaeuserBR12,wick-fsi-2021,endtmayer_goal-oriented_2024}.
For a review including detailed lists of references regarding different model problems, we refer to~\cite{BaRa03,endtmayer_chapter_2024}.

With regard to (convection-dominated) transport problems, the three approaches described above have been successfully applied severally as well as in varying combinations among each other.
The analysis of $hp$-adaptive algorithms for transport-dominated problems was pioneered by Houston, Schwab, and
Süli~\cite{houston_stabilized_2000,houston_discontinuous_2002}, who demonstrated exponential convergence of stabilized and discontinuous Galerkin methods.
For convection-dominated problems, Melenk provided rigorous analysis to prove exponential convergence of $hp$-FEM for singularly perturbed problems using anisotropic mesh grading and $p$-enrichment~\cite{melenk_hp-finite_2002}.
Bittl and Kuzmin proposed an $hp$-adaptive flux-corrected transport algorithm
for unsteady linear convection equations in~\cite{bittl-hp-fct-2013}.
For the development of DG methods on $hp$-anisotropic meshes for steady-state convection-diffusion-reaction equations we refer
to~\cite{georgoulis_discontinuous_2007,georgoulis_discontinuous_2008}, and using additionally goal-oriented error control we refer to~\cite{georgoulis_discontinuous_2009}.
Regarding non-stationary convection-diffusion problems, an $hp$-adaptive discontinuous Galerkin method using an a posteriori error estimator for the error in the $L^2(H^1)+L^\infty(L^2)$-type norm was developed in~\cite{cangiani-hp-cdr-2019}.
Anisotropic mesh refinement applied to convection-diffusion-reaction equations can be found in~\cite{apel_anisotropic_1996,formaggia_anisotropic_2001,formaggia_anisotropic_2004,knobloch2025adaptive-aniso}.
For time-dependent convection-dominated problems, the DWR method combined with stabilization has been investigated by some of the authors for isotropic~\cite{bruchh-dwr-2020,BruchhaeuserB22,bruchh-dwr-2022,bruchh-dwr-2025} as well as anisotropic~\cite{bause2025anisotropicspacetimegoalorientederror,bbemtw-multi-2025} adaptive refinement.
For high-order methods, Ceze and Fidkowski developed an anisotropic
\textit{hp}-adaptation strategy driven by adjoint error
estimates~\cite{ceze_anisotropic_2013}.
Dolejší et al.\ developed space-time DWR-based refinement strategies for convection-diffusion-reaction (CDR)
equations~\cite{dolejsi_anisotropic_2023}, adjusting both mesh anisotropy and polynomial degrees in space and time.
For general reviews of convection-dominated problems and further references, we refer to~\cite{roos-cdr-2008,john-cdr-2018}.

In this work, we develop a tensor-product space-time DG discretization with fully anisotropic \(hp\)-adaptivity, allowing independent \(h\)- and \(p\)-updates per spatial direction and in time.
Within a DWR framework, we derive directional error indicators \(\eta_{h,\,i}, \eta_{\tau}, i=1,\dots,d\), and use them to
develop a principled refinement policy that naturally yields refinement
decisions on anisotropic polynomial enrichment and geometric refinement. The decision between these two options is made based on a local saturation indicator, cf., e.\,g.,~\cite{houston_note_2005}. The algorithms use efficient transfer and integration procedures on anisotropically refined meshes, implemented in the finite element library \texttt{deal.II}~\cite{africa_dealii_2024} with shared-memory parallelism.
To improve the time-to-solution, we empirically assess an \(h\)- and \(p\)-robust preconditioning strategy for the resulting space-time systems that remains effective under strong anisotropy.
We target well-known challenges in the numerical simulation of convection-dominated problems:
\begin{enumerate}\itemsep1pt \parskip0pt \parsep0pt
\item Resolving thin boundary/interior layers and geometric singularities (e.g.,
  the Fichera corner) without spurious oscillations;
\item Reducing computational cost by concentrating resolution where a goal functional, chosen by the user, is most sensitive;
\item Deciding between \(h\) and \(p\) \emph{per direction} in accordance with
  local anisotropy and regularity;
\item Ensuring solver robustness for highly anisotropic meshes and high
  polynomial degrees;
\end{enumerate}
We propose a goal-oriented, anisotropic $hp$ space-time adaptive algorithm based on the DWR method.
Its realization builds on Richter's work on splitting spatial error
contributions~\cite{richter_posteriori_2010}, on subsequent extensions in the
space-time setting~\cite{bause2025anisotropicspacetimegoalorientederror},
and on anisotropic $p$-adaptivity by Houston et al.~\cite{georgoulis_discontinuous_2008,georgoulis_discontinuous_2009}.
We integrate these techniques into a unified framework suitable for challenging, time-dependent, convection-dominated problems.
The directional error decomposition and associated directional refinement policy are key to efficient \(hp\)-adaptation and distinguish our approach from methods that select an action via a ``competition'' among candidate refinements (see, for instance,~\cite{chakraborty_anisotropic_2024}).
By selecting \(h\)- or \(p\)-updates directly from the split estimator, we avoid candidate testing while aligning refinement with the local anisotropy.

The remainder of this work is structured as follows:
In Sec.~\ref{sec:notation}, we introduce the mathematical notation, the CDR equation, and its variational formulation.
The discrete anisotropic finite element spaces are introduced in Sec.~\ref{sec:discrete_spaces}.
The space-time discretization is presented in Sec.~\ref{sec:st-discretization}.
In Sec.~\ref{sec:error_rep}, we review an a posteriori error representation with respect to a single goal.
In Sec.~\ref{sec:anisotropic-ee}, we detail the anisotropic error estimation technique.
Sec.~\ref{sec:algorithm} describes the algorithm for goal-oriented, anisotropic $hp$ mesh refinement.
The solution of the arising algebraic systems is discussed in
Sec.~\ref{sec:solvers}.
Numerical examples are presented in Sec.~\ref{sec:numerical_examples}, including a fully time-dependent 3D test case.
Finally, Sec.~\ref{ref:conclusion} summarizes the results and outlines
directions for future research.

\section{Mathematical Problem and Notation}
\label{sec:notation}
Let \(\Omega\subset\mathbb R^d\) (\(d\in\N\)) be a bounded Lipschitz domain, and let the boundary be partitioned as
\( \partial\Omega = \Gamma_D \cup \Gamma_N,\ \Gamma_D\neq\emptyset,\ \Gamma_D\cap\Gamma_N=\emptyset.
\)
Let \(I=(0,T]\) denote the time interval, with final time \(T>0\).

We denote by \(H^1(\Omega)\) the Sobolev space of \(L^2(\Omega)\) functions
whose first derivatives are in \(L^2(\Omega)\). Define
\(H \coloneq L^2(\Omega)\), \(V \coloneq H^1(\Omega)\), and
\(V_0 \coloneq H^1_0(\Omega)\) as the space of \(H^1\)-functions with vanishing
trace on the Dirichlet boundary \(\Gamma_D\). The corresponding adjoint spaces
are denoted by \(V^\prime\) and \(V^\prime_0\), respectively.
The \(L^2\)-inner product is
\(( \cdot, \cdot )\), with the norm
\(\|\cdot\| \coloneq \|\cdot\|_{L^2(\Omega)}\).
We introduce the continuous space
\begin{equation}
  \label{eq:st-cont-l2}
  \mathcal{X}\coloneq\brc{w\in L^2(I;\,V_0)\,:\, \partial_t w \in L^{2}(I; V^\prime_0)}\,,
\end{equation}
where $L^2(I;\,V_0)$ is the Bochner space of \(V_0\)-valued functions as defined
in~\cite{lions}.

\paragraph{Convection-Diffusion-Reaction Equation.}
Convection-diffusion-reaction (CDR) equations are fundamental in modeling a wide
range of physical phenomena, including fluid dynamics, chemical reactions, and
heat transfer. These equations are characterized by the interaction of
convective transport, diffusion, and reactive sources or sinks. In this work we
study a time-dependent CDR equation given by
\begin{subequations}\label{eq:CDR_system}
    \begin{align}
      \partial_t u -\nabla \cdot (\boldsymbol\varepsilon \nabla
      u) + \boldsymbol{b}\cdot\nabla u + \alpha u &= f(\symbf{x},\,t), && \text{in } Q = \Omega \times I\,, \label{eq:CDR_a} \\
        u(\symbf{x},\,t) &= g(\symbf{x},\,t), && \text{on } \Gamma_D \times I\,, \label{eq:CDR_b} \\
      \boldsymbol\varepsilon \nabla u \cdot \boldsymbol{n} &= u_N(\symbf{x},\,t), && \text{on } \Gamma_N \times I\,, \label{eq:CDR_c} \\
      u(\symbf{x},\,0) &= u_0(\symbf{x}), && \text{in } \Omega\,, \label{eq:CDR_d}
    \end{align}
\end{subequations}
where $\Gamma_D$ and $\Gamma_N$ denote the Dirichlet and Neumann parts of the boundary $\partial\Omega$, respectively, with $\Gamma_D \cup \Gamma_N =
\partial\Omega$ and $\Gamma_D \neq \emptyset$.
The unit outward normal vector to the boundary is denoted by $\boldsymbol{n}$.

In equation~\eqref{eq:CDR_system}, the parameter $\boldsymbol\varepsilon\in
  L^{\infty}(I; W^{1,\,\infty}(\Omega)^{d\times d})$ represents a symmetric
  positive definite
  diffusion coefficient. The convection field $\boldsymbol{b} \in L^{\infty}(I; W^{1,\,\infty}(\Omega)^d)$ governs the advective transport, while $\alpha \in L^{\infty}(I; L^{\infty}(\Omega))$ is a reaction coefficient satisfying $\alpha \geq 0$. The source term $f \in L^{2}(I; H^{-1}(\Omega))$, initial condition $u_0 \in L^2(\Omega)$, Dirichlet boundary data $g \in L^{2}(I; H^{\frac{1}{2}}(\Gamma_D))$, and Neumann boundary data $u_N \in L^{2}(I; H^{-\frac{1}{2}}(\Gamma_N))$ are given.
To ensure the well-posedness of problem~\eqref{eq:CDR_system}, we assume either
(i) $\nabla \cdot \boldsymbol{b}(\boldsymbol{x},t) = 0$ and $\alpha(\boldsymbol{x},t) \ge 0$, or
(ii) that there exists a constant $c_0 > 0$ such that
\begin{equation}
  \alpha(\boldsymbol{x},t) - \tfrac{1}{2}\,\mathrm{div}\,\boldsymbol{b}(\boldsymbol{x},t)
  \;\ge\; c_0
  \qquad \forall\,(\boldsymbol{x},t)\in\overline{\Omega}\times\overline{I}.
\end{equation}
These conditions are standard in the analysis of convection-diffusion-reaction
problems of type~\eqref{eq:CDR_system} and ensure coercivity and stability of the
associated variational formulation; see, e.\,g.,~\cite{BruchhaeuserAJ15,roos-cdr-2008}.

\paragraph{Weak Formulation of the CDR Equation.} Under the above conditions,
problem~\eqref{eq:CDR_system} admits a unique weak solution $u\in \mathcal{X}$ satisfying the variational formulation
\begin{equation}
    \label{eq:weak_formulation}
    A(u)(w) = F(w) \quad \forall w \in \mathcal{X},
\end{equation}
where the bilinear form $A: \mathcal{X} \times \mathcal{X} \to \mathbb{R}$ and the linear form $F: \mathcal{X} \to \mathbb{R}$ are defined by
\begin{equation}
    \label{eq:A_def}
    A(u)(w) \coloneq \int_{I} \left\{ (\partial_t u, w) + a(u)(w) \right\} \, \drv t + (u(0), w(0)),\qquad
    F(w) \coloneq \int_{I} (f, w) \, \drv t + (u_0, w(0)),
  \end{equation}
with $(v,w)\coloneq \int_{\Omega}v w\, \textrm{d}\boldsymbol x$  and the bilinear form $a: V \times V \to \mathbb{R}$ given by
\begin{equation}
    \label{eq:a_def}
    a(u)(w) \coloneq (\boldsymbol\varepsilon \nabla u, \nabla w) + (\boldsymbol{b} \cdot \nabla u, w) + (\alpha u, w).
  \end{equation}

\section{Discrete Spaces}
\label{sec:discrete_spaces}
We now introduce the discrete space-time meshes and finite element spaces underlying the DG discretization.
Letting $t_0=0$ and
$t_{N_I}=T$, we split the time
interval \(I\) into a sequence of \(N_I\) disjoint subintervals
\(I_n=(t_{n-1},\,t_n]\), \(n=1,\dots,\,N_I\) with the time step
$\tau_n\coloneq t_n-t_{n-1}$. By $\mathcal T_{\tau}=\{0\}\cup I_1 \cup \cdots \cup I_{N_I}$ we refer to the time mesh.

For the space discretization, let \(\mathcal T_h\) be a conforming partition of
\(\Omega\) into open quadrilaterals ($d=2$) or hexahedrals ($d=3$).
Each element \( K \in \mathcal{T}_h \) is mapped from the reference cell
\(\hat{K} = (-1,1)^d\) by an isoparametric transformation \(\boldsymbol{T}_K\)
with \(\det(\boldsymbol{T}_K)(\hat{x}) > 0\) for all \(\hat{x} \in \hat{K}\).
We decompose
\begin{equation}
\label{eq:TK}
\boldsymbol T_K \coloneq \boldsymbol R_K \circ\boldsymbol S_{c,K} \circ\boldsymbol S_{h,K} \circ\boldsymbol P_K,
\end{equation}
where \(\boldsymbol R_K \) is a rotation and translation, \(\boldsymbol S_{c,K}
\) is an anisotropic scaling, \(\boldsymbol S_{h,K} \) is a shearing, and
\(\boldsymbol P_K \) is a nonlinear component.
To account for anisotropic elements, we relax the standard shape-regularity
assumptions and require only uniform boundedness of \(\boldsymbol S_{h,K} \) and
\(\boldsymbol P_K \) for all \( K \in \mathcal{T}_h \).
Let \( h_{K,i} \), \(i = 1,\dots,d\), denote the element size in coordinate direction \(i\). Define the local anisotropy ratio by
\begin{equation}\label{eq:loc-aniso-ratio}
  \rho_{K} \coloneq \max_{\boldsymbol{x} \in Q_K} \frac{\lambda_{\max}(\boldsymbol{x})}{\lambda_{\min}(\boldsymbol{x})},
\end{equation}
where \(Q_K\) is the set of quadrature points in \(K\), and \(\lambda_{\min}\),
\(\lambda_{\max}\) are the minimal and maximal eigenvalues of \(\nabla
\boldsymbol{T}_K(\boldsymbol{x})\), respectively. An element $K$ is called
\emph{anisotropic} if \(\rho_{K} {>} 1\).%

For each element \(K\), let \(\mathcal F(K)\) be its set of \((d-1)\)-dimensional faces, and let
\(\mathcal F_h = \bigcup_{K\in\mathcal T_h} \mathcal F(K)\), where \(\mathcal F_h = \mathcal F_h^0 \cup \mathcal F_h^b\) is partitioned into the sets of interior faces and boundary faces, respectively. For any interior face \(F\), there are exactly two elements \(K^+\) and \(K^-\) sharing \(F\), with outward unit normals \(\boldsymbol n_F^+\) (w.r.t.\ \(K^+\)) and \(\boldsymbol n_F^-\) (w.r.t.\ \(K^-\)).  If \(F\in\mathcal F_h^b\), then only one element \(K\) touches \(F\), and we denote by \(\boldsymbol n_F\) its outward unit normal.  The \((d-1)\)-measure of \(F\) is \(|F|\).
We denote the space-time tensor product mesh by
$\mathcal T_{\tau h}=\mathcal T_h\times \mathcal T_{\tau}$.

For a $k\in \N_0\coloneq \N\cup \{0\}$ and Banach space $B$, we let $\mathbb P_{k}(I_n;\,B)$ denote the set of all polynomials of degree less than or equal to $k$ on $I_n$ with values in $B$.
As we allow $hp$ anisotropic refinements in time, we define \(k_n\) as the polynomial
degree on interval \(I_n\), which potentially vary significantly.
For a multi index \(\boldsymbol k =(k_1,\dots,\,k_{N_I})\in \N_0^{N_I}\) we
define the space of semi-discrete \(L^2\) in time functions as
\begin{equation}
  \label{eq:st-disc-l2}
  \mathcal{X}_{\tau}^{\boldsymbol k}(V_0)\coloneq\brc{w_{\tau}\in L^2(I;\,V_0) \suchthat w_{\tau}\restrict{I_n} \in \mathbb{P}_{k_n}(I_{n};\,V_0)\,, w_{\tau}(0)\in H\,, \forall n\in \{1,\dots,\,N_I\}}\,.
\end{equation}

\subsection{Anisotropic Polynomial Degrees in Space}
We define the anisotropic polynomial degree vector on each element \( K \in \mathcal{T}_h \) by
\[
  p_K = (p_{K,1},\dots,p_{K,d}) \quad \text{with} \quad p_{K,i} \in \mathbb{N}_0.
\]
\begin{definition}[$p$-Anisotropic Finite Element]
  \label{def:anisoFEM}
  Let \((p_1,\dots,p_d)\in\mathbb{N}_0^d\) be a multi-index. Then, the anisotropic polynomial space on the reference element \(\hat{K} = (-1,1)^d\) is defined as
  \[
    \hat{\mathbb{Q}}_{p_1,\dots,p_d}(\hat{K}) \coloneq \bigotimes_{i=1}^{d} \mathbb{P}_{p_i}([-1,1])\,,
  \]
  where \(\mathbb{P}_{p}([-1,1])\) denotes the space of univariate polynomials of
  degree at most \(p\).
\end{definition}
The degrees of freedom (DoFs) are associated with Gauss-Lobatto nodes on each reference element \(\hat{K}\), and the total number of DoFs per element is
\[
  N_K^{p_1,\dots,p_d} = \prod_{i=1}^{d} (p_{K,i} + 1).
\]
This construction generalizes the isotropic case: for \( p_1 = \dots = p_d
\coloneq p \), we recover the standard space \(\hat{\mathbb Q}_p(\hat{K})\).
The corresponding local finite element space on \(K\) consists of tensor-product polynomials:
\[
  \mathbb{Q}_{p_K}(K) \coloneq \bigotimes_{i=1}^{d}
  \mathbb{P}_{p_{K,i}}([-1,1])\circ \boldsymbol{T}_K^{-1}\,,
\]
where $\boldsymbol{T}_K$ is the isoparametric mapping of
degree $\max_{i\in \{1,\dots,d\}} p_{K,i}$.

\begin{definition}[Global $p$-Anisotropic Finite Element Spaces]
  We denote the global finite element space with $p$-anisotropic elements as defined in Def.~\ref{def:anisoFEM} by
  \begin{equation}\label{eq:p-aniso-gl}
    \mathcal{V}_h^{\boldsymbol{p}} \coloneq \{ v_h\in L^2(\bar\Omega)\;:\;
    v_h\restrict{K}\in \mathbb{Q}_{p_K}(K)\;\forall K\in\mathcal T_h\},\quad p_K=(p_{K,1},\dots,p_{K,d})\,,
  \end{equation}
  where $\boldsymbol{p} \coloneq (p_K)_{K\in\mathcal{T}_h}$, i.\,e.\ the vector of elementwise anisotropic polynomial degrees.
  Further, for a fixed $\boldsymbol p$ in~\eqref{eq:p-aniso-gl} and $k\in \N$ we introduce the higher-order space
  \begin{equation}\label{eq:p-aniso-gl+1}
    \mathcal{V}_h^{\boldsymbol{p}+\boldsymbol k} \coloneq \{ v_h\in
    L^2(\bar\Omega)\;:\; v_h\restrict{K}\in
    \mathbb{Q}_{p_K+k}(K)\;\forall
    K\in\mathcal T_h\},\quad p_K+k\coloneq (p_{K,1}+k,\dots,p_{K,d}+k),\quad
  \end{equation}
  For the directional error estimation, we also need spaces for restrictions in
  the $i$-th spatial direction. For the same $\boldsymbol p$ considered above, we define these spaces as
  \begin{equation}\label{eq:p-aniso-gl-ei}
    \mathcal{V}_h^{\boldsymbol{p},i} \coloneq \{ v_h\in L^2(\bar\Omega)\;:\;
    v_h\restrict{K}\in \mathbb{Q}_{p_K-e_i}(K)\;\forall K\in\mathcal T_h\}
  \end{equation}
  where $e_i$ is the $i$-th unit vector in $\mathbb N_0^d$. We further need an
  analogous variant of $\mathcal{V}_h^{\boldsymbol{p},i}$ for the higher-order
  space $\mathcal{V}_h^{\boldsymbol{p}+\boldsymbol{k}}$, which we denote by
  \begin{equation}
    \label{eq:p-aniso-gl+1-ei}
    \mathcal{V}_h^{\boldsymbol{p}+\boldsymbol{k},i}\coloneq \{ v_h\in L^2(\bar\Omega)\;:\;
    v_h\restrict{K}\in \mathbb{Q}_{p_K+k-ke_i}(K),\,\;\forall K\in\mathcal
    T_h\},\quad p_K+k-ke_i=(p_{K,1}+k,\dots,p_{K,d}+k)-ke_i\,.
  \end{equation}
\end{definition}

\begin{definition}[$p$-Anisotropic Discrete Space-Time Function Spaces.]
  We define the discrete space-time function space for~\eqref{eq:p-aniso-gl} by the tensor-products
  \begin{align}
    \label{eq:disc-spaces-tp-y}
    \mathcal{X}_{\tau h}^{\boldsymbol{k},\,\boldsymbol{p}}\coloneq\mathcal{X}_{\tau}^{\boldsymbol{k}}(\R)\otimes\mathcal{V}_h^{\boldsymbol{p}}
    &=\operatorname{span} \{
      \xi\otimes\phi\suchthat\xi \in \mathcal{X}_{\tau}^{\boldsymbol{k}}(\R),\: \phi\in \mathcal{V}_h^{\boldsymbol{p}}\}\,,
  \end{align}
  where $\xi\otimes\phi\,(\symbf{x},\,t) = \phi(\symbf{x})\xi(t)$. The
  space-time function spaces based
  on the spaces~\eqref{eq:p-aniso-gl+1}\,--\,\eqref{eq:p-aniso-gl+1-ei}
  are defined analogously.
\end{definition}

\begin{remark}[Tensor product function spaces]\label{rem:tensor-product-spaces}
  For the
  construction of tensor products of Hilbert spaces we refer to~\cite[Sec.\ 1.2.3]{picard_partial_2011}. We remark that the space
  $\mathcal{X}_{\tau h}^{\boldsymbol{k},\,\boldsymbol{p}}$ is isometric to the
  Hilbert space \[\mathcal{X}_{\tau}^{\boldsymbol k}(\mathcal V_h^{\boldsymbol p})=\{ w_{\tau h}\in
  L^2(I;\mathcal{V}_h^{\boldsymbol{p}})\;:\; w_{\tau h}|_{I_n}\in\mathbb
  P_{k_n}(I_n;\mathcal{V}_h^{\boldsymbol{p}})\,, w_{\tau h}(0)\in
  \mathcal{V}_h^{\boldsymbol{p}}\,, \forall n\in \{1,\dots,N_I\}\}\,,\]
  see~\cite[Proposition 1.2.28]{picard_partial_2011}.
  Thus, we identify these pairs of spaces with each other throughout this work.
  Similarly, we identify $\mathbb{P}_{k_n}(I_{n};\,\mathcal{V}_{h}^{\boldsymbol{p}})$ with
  $\mathbb{P}_{k_n}(I_{n};\,\R)\otimes \mathcal{V}_{h}^{\boldsymbol{p}}$.
\end{remark}

For \((x,t)\in \Omega\times\overline{I_n}\), we write
\[
  w_{\tau h}(x,t)\bigm|_{I_n}
  = \sum_{i=1}^{k_n+1} \Bigl(\sum_{j=1}^{N_{\symbf x}} w^i_{n,j}\,\phi_j(x)\Bigr)\;\xi_{n,i}(t)\,,
\]
where \(\{\phi_j\}_{j=1}^{N_{\symbf x}}\subset\mathcal V_h^{\boldsymbol p}\) is the nodal basis,
\(\{\xi_{n,i}\}_{i=1}^{k_n+1}\subset\mathbb P_{k_n}(I_n)\) is the \(k_n+1\)-point Lagrange basis on \(I_n\),
and \(w^i_{n,j}\in\mathbb R\) are the corresponding coefficients.

Given an elementwise defined scalar function \(v\), its \emph{jump} across an
interior face \(F\in\mathcal F_h^0\) shared by elements \(K^+\) and \(K^-\) is
\[
  [ v ]_{F} = v^+ - v^-,
\quad
\text{where}
\quad
  v^+ = v(x)\restrict{F,K^+},
  \qquad
  v^- = v(x)\restrict{F,K^-}.
\]
Similarly, the average of an elementwise defined scalar function \(v\) across an interior face \(F\in\mathcal F_h^0\) shared by elements
\(K^+\) and \(K^-\) is defined as
\[
  \{v\}=\frac{1}{2}(v^++v^-)\,.
\]

Here \(v\restrict{F,K^+}\) means ``the trace of \(v\) on \(F\) as viewed from the
element \(K^+\)'', and similarly for \(v\restrict{F,K^-}\). On a boundary face
\(F\in\mathcal F_h^b\), one typically sets \([v]_F =\{v\}_F = v\restrict{F}\),
since there is only a single trace on that boundary face.

Throughout this paper, the symbol $C$ denotes a generic positive constant. Its value may change from one occurrence to another and is independent of the
mesh size $h$ and other discretization parameters.

\subsection{Anisotropic Polynomial Degrees in Time}
For $hp$ anisotropic refinements in time we define the temporal basis functions
and corresponding temporal matrices.
We discretize in time using a discontinuous Galerkin method DG($k$), $k \ge 0$,
with trial and test spaces $\mathcal{X}_{\tau}^{\boldsymbol k}(V_0)$
(cf.~\eqref{eq:st-disc-l2}). This allows discontinuities across subinterval
boundaries and decouples the system at interval endpoints. This enables a
piecewise solution with temporal basis functions supported on each
subinterval~$I_n$.

Let \(\{\hat \xi_i\}_{i=1}^{k_n+1}\subset \mathbb{P}_{k_n}(\hat I,\,\R)\),
\(\hat I \coloneq [-1,\,1]\) denote the Lagrangian basis
\(\mathbb{P}_{k_n}(\hat I,\,\R)\) w.\,r.\,t.\ the integration points of the
\(k_n+1\) point Gauss-Radau quadrature including the right endpoint. Further,
denote by \(\mathbb{P}_{k_n}(I_n,\,\R)\) the mapped polynomial space from
$\hat I$ to $I_n$ with the mapped basis $\{\xi_i\}_{i=1}^{k_n+1}$.
\paragraph{Time Discretization Matrices}
The matrices for the time discretization are
given through the weights \(\symbf M_{\tau}^k\in \R^{k+1\times k+1}\),
\(\symbf A_{\tau}^k\in \R^{k+1\times k+1}\) and \(\symbf m_{\tau}^k\in \R^{k+1}\), with
\begin{equation}
  \label{eq:dg-time-weights}
  {(\symbf M_{\tau}^k)}_{i,\,j} \coloneq
  \int_{\hat I}{\hat\xi_j{(\hat t)}
    \hat \xi_i{(\hat t)} \drv \hat t}
  \,,\quad
  {(\symbf A_{\tau}^k)}_{i,\,j} \coloneq \int_{\hat I} \hat \xi_j'{(\hat t)\hat \xi_i{(\hat t)} \drv \hat t}+ \hat\xi_j{(0)} \hat \xi_i{(0)}
  \,,\quad
  {(\symbf m_{\tau}^k)}_{i} \coloneq  \hat \xi_i{(0)}\,,
\end{equation}
for $i,\,j=1,\dots,\,k+1$.

\section{Space-Time Finite Element Discretization}\label{sec:st-discretization}
We now introduce the space-time finite element discretization of the weak
formulation~\eqref{eq:weak_formulation}, based on the discrete spaces from the
previous section, which will be used within the DWR framework. In time, we
employ DG($k$) with slabwise polynomial degrees $k_n$; in space, we use
$p$-anisotropic tensor-product elements on quadrilateral/hexahedral meshes. We
define the discrete variational form, and the resulting local and global
algebraic systems. This prepares the discrete adjoint (by transposition) and the
anisotropic error estimator.

\paragraph{Discrete Space-Time Variational Formulation.}

Given the $hp$-anisotropic space-time finite element space
\(\mathcal{X}_{\tau h}^{\boldsymbol k,\boldsymbol{p}}\), we formulate the discrete space-time
variational problem for the CDR equation.
First, we define the spatial DG bilinear form \(a_s^{\gamma}(\cdot,\cdot)\) by
\begin{equation}
  \label{eq:spatial_dg_bilinear_form}
  \begin{aligned}
    a_s^{\gamma}(u,w)
    &\coloneq \sum_{K\in\mathcal T_h}\Bigl(
      (\boldsymbol\varepsilon\nabla u, \nabla w)_K
      - (u,\boldsymbol b\cdot\nabla w)_K
      + (\alpha u,w)_K
    \Bigr)
    \\
    &\quad + \sum_{F\in\mathcal F_h^0}\Bigl(
      -(\boldsymbol\varepsilon\{\nabla u\}\cdot\boldsymbol n_F,[w])_F
      -([u],\boldsymbol\varepsilon\{\nabla w\}\cdot\boldsymbol n_F)_F
      \\
    &\qquad\qquad\qquad\qquad
      +\left(\gamma_F\bigl\|\boldsymbol\varepsilon\bigr\|_{L^\infty(F)}
             +\frac{|\boldsymbol b\cdot\boldsymbol n_F|}{2}\right)([u],[w])_F
      +(\boldsymbol b\cdot\boldsymbol n_F\{u\},[w])_F
    \Bigr)
    \\
    &\quad + \sum_{F\in\mathcal F_h^b}\Bigl(
      -(\boldsymbol\varepsilon\nabla u\cdot\boldsymbol n_F,w)_F
      -(u,\boldsymbol\varepsilon\nabla w\cdot\boldsymbol n_F)_F
      +\left(\gamma_F\bigl\|\boldsymbol\varepsilon\bigr\|_{L^\infty(F)}
             +(\boldsymbol b\cdot\boldsymbol n_F)^+\right)(u,w)_F
    \Bigr).
  \end{aligned}
\end{equation}
Here, $\gamma_F>0$ is the penalty parameter we define later. We use the standard
notation that $({\cdot},\,{\cdot})_K$ denotes the $L^2$-inner product on an
element $K$, with $({\cdot},\,{\cdot})_F$ defined analogously on a face $F$.
Moreover,
$(\boldsymbol{b}\cdot\boldsymbol{n}_F)^+ =
\max(\boldsymbol{b}\cdot\boldsymbol{n}_F,\,0)$ denotes the positive part of the
convective normal flux. Now, we can introduce the space-time discrete
variational formulation.
\begin{problem}[Space-Time Discrete variational form of anisotropic \(hp\)-DG for CDR]
\label{prob:discrete_variational}
Find \(u_{\tau h}\in\mathcal{X}_{\tau h}^{\boldsymbol k,\boldsymbol{p}}\) such that for all
\(w_{\tau h}\in\mathcal{X}_{\tau h}^{\boldsymbol k,\boldsymbol{p}}\)
\begin{equation}\label{eq:discrete_variational_abstract}
  A_{\tau h}(u_{\tau h}, w_{\tau h}) = F_{\tau h}(w_{\tau h}),
\end{equation}
where the bilinear form splits into temporal and spatial parts as
\begin{equation}
  \label{eq:discrete_variational_form}
  \begin{aligned}
    A_{\tau h}(u_{\tau h}, w_{\tau h})
    &\coloneq
    \sum_{n=1}^{N_I}\int_{I_n}\Bigl(
      (\partial_t u_{\tau h}, w_{\tau h})
      + a_s^{\gamma}(u_{\tau h}, w_{\tau h})
    \Bigr)\,\drv t
    \\
    &\quad
    + \sum_{n=1}^{N_I-1} \bigl(([u_{\tau h}]_{n}, w_{\tau h}(t_n^+))\bigr)
    + \bigl((u_{\tau h}(0^+), w_{\tau h}(0^+))\bigr),
  \end{aligned}
\end{equation}
and the linear form is given by
\begin{equation}
  \label{eq:linear_form}
  \begin{aligned}
    F_{\tau h}(w_{\tau h})
    &\coloneq
    \sum_{n=1}^{N_I}\int_{I_n} (f,w_{\tau h})\,\drv t
    + (u_0,w_{\tau h}(0^+))
    \\
    &\quad
    + \sum_{n=1}^{N_I}\int_{I_n}
      \sum_{F\in\mathcal F_h^b}\Bigl(
        \left(\frac{\gamma\boldsymbol\varepsilon}{h_F}
               +(\boldsymbol b\cdot\boldsymbol n_F)^+\right)(g,w_{\tau h})_F
        -(u_N,w_{\tau h})_F
      \Bigr)\,\drv t\,.
  \end{aligned}
\end{equation}
\end{problem}
All integrals are taken over space and time, and the jumps \([u_{\tau h}]_n\) and traces \(u_{\tau h}(0^+)\) are understood with respect to the DG-in-time discretization.
We derive the local and global algebraic systems resulting from the anisotropic
\(hp\)-DG discretization of the CDR
equations~\eqref{eq:discrete_variational_form}. To this end, we define the
spatial mass matrix \(\symbf M_h\) and stiffness matrix \(\symbf A_h\)
associated with~\eqref{eq:spatial_dg_bilinear_form} by
\[
  \symbf M_h \coloneq ((\phi_j,\phi_i))_{i,j=1}^{N_{\symbf x}}, \qquad
  \symbf A_h \coloneq (a_s^{\gamma}(\phi_j,\phi_i))_{i,j=1}^{N_{\symbf x}},
\]
where \(\{\phi_i\}_{i=1}^{N_{\symbf x}}\) and \(\{\phi_j\}_{j=1}^{N_{\symbf x}}\subset \mathcal V_h^{\boldsymbol{p}}\) denotes
a nodal Lagrangian basis of \(\mathcal V_h^{\boldsymbol{p}}\).

\paragraph{Penalty parameter scaling.}
For robustness on anisotropic $hp$-meshes we choose the symmetric
interior-penalty parameter on a face $F$ shared by $K^+$ and $K^{-}$ as
\[
  \gamma_F = C_{\text{pen}}\frac{1}{2}\Biggl( \frac{p_{K^+,i}(p_{K^+,i}+1)}{h_{F}^{K^+}}+ \frac{p_{K^-,i}(p_{K^-,i}+1)}{h_{F}^{K^-}} \Biggr),
\]
with $\gamma_F$ halved on boundary faces (only one neighbor). Here, $h_{F}^{K}$
denotes the directional mesh size of $K$ normal to $F$ and,
for \(K\in\{K^+,K^-\}\), $p_F^{K}$ is a representative polynomial degree
\[
  p_F^{K}
  \coloneq
  \max\Bigl\{1,\,
      \frac{1}{d-1}\sum_{\substack{i=1 \\ i\neq i_F}}^{d} p_{K,i}
      \Bigr\},
\]
where \(i_F\in\{1,\dots,d\}\) is the index of the coordinate direction normal to
\(F\). This scaling is similar to the ones proposed
in~\cite{georgoulis_discontinuous_2007,georgoulis_discontinuous_2008} and aligns
with anisotropic SIP theory and ensures stability uniformly in $h$, $p$, and
moderate anisotropy.

\paragraph{Local Algebraic Problem.}
Upon numerical integration of the discrete variational formulation in
Prob.~\ref{prob:discrete_variational}, we obtain local algebraic systems on
each interval \(I_n\). We define the right-hand side vector \(\mat F_n\in\mathbb{R}^{(k_n+1)N_{\symbf x}}\) as
\[
  \mat F_n=(\vct f_n^1,\dots,\vct f_n^{k_n+1})^\top,
\]
with components given by
\[
(\vct f_n^i)_j=\int_{I_n}\left((f, \phi_j\xi_{n,i}) - (\boldsymbol\varepsilon\nabla (\phi_j\xi_{n,i})\cdot\boldsymbol n,g)_{\Gamma_D} - \frac{\gamma}{h_F}(g,\phi_j\xi_{n,i})_{\Gamma_D}\right)\drv t.
\]

\begin{problem}[Local Algebraic System Anisotropic \(hp\)-DG for CDR]\label{prob:primal-local}
Find \(\vct u_n=(\vct u_n^1,\dots,\vct u_n^{k_n+1})^\top\), with \(\vct u_n^i\in\mathbb{R}^{N_{\symbf x}}\), such that
\[
  (\tau_n\symbf M_\tau^{k_n}\otimes \symbf A_h + \symbf A_\tau^{k_n}\otimes
  \symbf M_h)\vct u_n=\mat F_n+\symbf m_{\tau}^{k_n}\otimes \symbf M_h \vct u_{n-1}^{k_{n-1}+1},
\]
with given initial vector \(\vct u_0^{k_0+1}\).
\end{problem}

\paragraph{Global Algebraic Problem}
For a straightforward derivation of the adjoint problem, we construct the global algebraic system for the full time interval \(I\).
Define the coupling matrix \(\symbf B_n\coloneq \symbf e_{k_n+1}^\top\otimes
\symbf m_{\tau}^{k_n}\otimes \symbf M_h\), which corresponds to the temporal jumps. The global algebraic system is then represented as follows.
\begin{problem}[Global Algebraic System Anisotropic \(hp\)-DG for CDR]\label{prob:primal-global}
Find \((\vct u_1,\dots,\vct u_{N_I})^\top\) solving
\[
\begin{pmatrix}
\symbf S_1 & & &\\
-\symbf B_2 & \symbf S_2 & &\\
 & \ddots & \ddots &\\
 & & -\symbf B_{N_I} & \symbf S_{N_I}
\end{pmatrix}
\begin{pmatrix}
\vct u_1\\
\vct u_2\\
\vdots\\
\vct u_{N_I}
\end{pmatrix}=
\begin{pmatrix}
\mat F_1+\symbf m_{\tau}^{k_1}\otimes \symbf M_h \vct u_0^{k_0+1}\\
\mat F_2\\
\vdots\\
\mat F_{N_I}
\end{pmatrix},
\]
with block matrices defined as \(\symbf S_n=\tau_n\symbf M_\tau^{k_n}\otimes \symbf A_h + \symbf A_\tau^{k_n}\otimes \symbf M_h\).
\end{problem}

We interpret the time marching scheme as forward substitution in the global
system given in Prob.~\ref{prob:primal-global}. This facilitates a straightforward derivation of the global adjoint
problem.
\paragraph{Adjoint Algebraic System}
We now derive the corresponding adjoint system from the global formulation by
transposing the global system in Prob.~\ref{prob:primal-global}. We describe
this on the discrete level: Based on the global slabwise DG bilinear
form~\eqref{eq:discrete_variational_form} in
Prob.~\ref{prob:discrete_variational}, we define the \emph{adjoint variational
  problem} by the formal Hilbert adjoint
\begin{equation}\label{eq:discrete_adjoint_variational_abstract}
  A^\ast(z_{\tau h},\,v_{\tau h})\coloneq A(v_{\tau h},\,z_{\tau h}),\qquad v_{\tau h},\,z_{\tau h}\in\mathcal{X}_{\tau h}^{\boldsymbol k,\boldsymbol{p+1}}.
\end{equation}
In particular, we do \emph{not} integrate by parts in space or time. Therefore,
the convection direction and temporal orientation are not reversed. Then, the
adjoint discretization is, by design, the transpose of the global primal
algebraic system in Prob.~\ref{prob:primal-global}. This yields an exact
discrete-adjoint pairing for DWR.\@

For the implementation of the DWR approach, we define the discrete representation
of the target functional \(J(u)\) on each interval \(I_n\) analogously to the
right-hand side as
\[
  (\vct J_n^i)_k = \int_{I_n}\left( j, \phi_k\xi_{n,i} \right)\drv t,
\]
leading to the global vector
\(\mat J_n=(\vct J_n^1,\dots,\vct J_n^{k_n+1})^\top\). In the implementation we
also support $J(u)$ to be space–time point evaluations. For brevity, we only
consider goal functionals that can be evaluated as space–time integrals.

\begin{problem}[Global Adjoint System Anisotropic \(hp\)-DG for CDR]\label{prob:adjoint-variational}
Let $\symbf S_n$, $\symbf B_n$, and the right-hand side blocks be as in the primal global system of Prob.~\ref{prob:primal-global}. The discrete adjoint unknowns are the vectors $\vct z_n\in\mathbb R^{(k_n+1)N_{\symbf x}}$, and the global adjoint system is defined by transposing the primal system:
\[
\begin{pmatrix}
\symbf S_1^\top & -\symbf B_2^\top & & \\
 & \symbf S_2^\top & -\symbf B_3^\top & \\
 & & \ddots & \ddots \\
 & & & \symbf S_{N_I}^\top
\end{pmatrix}
\begin{pmatrix}
\vct z_1\\
\vct z_2\\
\vdots\\
\vct z_{N_I}
\end{pmatrix} =
\begin{pmatrix}
  \mat J_1\\
  \mat J_2\\
  \vdots\\
  \mat J_{N_I}
\end{pmatrix},
\]
with $\mat J_n$ the goal load on $I_n$.
\end{problem}
The transpose of the global primal system in
Prob.~\ref{prob:primal-global} yields a lower-triangular system.
Consequently, the adjoint is solved by backward substitution. The per‑slab
adjoint system stated below is the corresponding block row of the transposed
global system. At each step of the backward substitution, we solve a local slab
problem. This mirrors the primal solve, but it is not a time‑marching scheme
since the time derivative acts on the test function.
\begin{problem}[Local Adjoint System Anisotropic \(hp\)-DG for CDR]\label{prob:adjoint-variational-local}
On interval \(I_n\), solve
\[
  (\tau_n(\symbf M_\tau^{k_n})^\top\otimes \symbf A_h + (\symbf A_\tau^{k_n})^\top\otimes \symbf M_h)\vct z_n = (\symbf M_\tau^{k_n})\otimes \symbf M_h \vct J_n + \symbf e_{k_n+1}\otimes((\symbf m_{\tau}^{k_n})^\top\otimes \symbf M_h)\vct z_{n+1},
\]
with \(\vct z_{N_I+1}\coloneq 0\).
\end{problem}

\section{A Posteriori Error Representation}\label{sec:error_rep}
In this section, we present a goal-oriented a posteriori error representation
based on the Dual Weighted Residual (DWR) method for the space-time DG
discretizations considered here. The error is measured with respect to a
user-chosen goal functional $J \in \mathcal C^3(\mathcal X,\,\mathbb R)$ of the
form
\begin{equation}\label{eq:def:goal-functional}
  J(u) = \int_0^T J_t(u(t))\,\mathrm dt + J_T(u(T))\,,
\end{equation}
where $J_t$ and $J_T$ are three-times continuously differentiable functionals;
see, e.g.,~\cite[Ch.~4]{BruchhaeuserB22}. We state the main error representation
adapted to our setting and refer to~\cite[Ch.~4]{BruchhaeuserB22} for a detailed
derivation for CDR equations.

For the DWR framework we introduce Lagrangians on the continuous, semi-discrete
and fully discrete levels. With $A(\cdot)(\cdot)$ and $F(\cdot)$ as in the weak
formulation~\eqref{eq:A_def}, \eqref{eq:linear_form}, and with $A_\tau(\cdot,\cdot)$ and
$F_\tau(\cdot)$ denoting their time-semi-discrete counterparts, and
$A_{\tau h}(\cdot)(\cdot)$ the fully discrete bilinear form of
\eqref{eq:discrete_variational_form}, we define
\begin{subequations}\label{eq:def:lagrangians}
\begin{align}
  \mathcal L(u, z) &\coloneq J(u) + F(z) - A(u)(z)\,,\\
  \mathcal L_\tau(u_\tau, z_\tau) &\coloneq J(u_\tau) + F_\tau(z_\tau) - A_\tau(u_\tau)(z_\tau)\,,\\
  \mathcal L_{\tau h}(u_{\tau h}, z_{\tau h}) &\coloneq J(u_{\tau h}) + F_{\tau h}(z_{\tau h}) - A_{\tau h}(u_{\tau h})(z_{\tau h})\,.
\end{align}
\end{subequations}
Here, $(u, z)$, $(u_\tau, z_\tau)$, and $(u_{\tau h}, z_{\tau h})$ are primal/adjoint
pairs at the continuous, time-semi-discrete, and fully discrete levels,
respectively. Directional (Gâteaux) derivatives of the Lagrangians with respect
to the second argument yield the primal problems on each level, while
derivatives with respect to the first argument yield the corresponding adjoint
problems; see also~\cite{BruchhaeuserBR12}. This representation motivates the
discretization choices in Sec.~\ref{sec:st-discretization} and, in particular, the
discrete adjoint obtained by transposition of the global primal system
(Prob.~\ref{prob:primal-global}), see Sec.~\ref{prob:adjoint-variational}.

We further define the time-semi-discrete primal and adjoint residuals by
\begin{subequations}\label{eq:def:tau-residuals}
\begin{align}
  \rho_\tau(u)(\varphi) &\coloneq \mathcal L'_{\tau, z}(u, z)(\varphi) = F_\tau(\varphi) - A_\tau(u)(\varphi)\,,\\
  \rho_\tau^{\ast}(u, z)(\varphi) &\coloneq \mathcal L'_{\tau, u}(u, z)(\varphi) = J'(u)(\varphi) - A_\tau'(u)(\varphi, z)\,.
\end{align}
\end{subequations}
Here
$A_\tau: \mathcal X_\tau^{\boldsymbol k}(V_0)\times \mathcal X_\tau^{\boldsymbol k}(V_0)\to\mathbb
R$ denotes the time–semi‑discrete DG bilinear form obtained by applying DG in
time to the weak form~\eqref{eq:A_def} (including slab-interface jump terms).
$A_{\tau h}: \mathcal X_{\tau h}^{\boldsymbol k,\,\boldsymbol p}\times
\mathcal X_{\tau h}^{\boldsymbol k,\,\boldsymbol p}\to\mathbb R$ and
$F_{\tau h}: \mathcal X_{\tau h}^{\boldsymbol k,\,\boldsymbol p}\to\mathbb R$
are the the fully discrete forms. For the error representation, it suffices that
$(u_{\tau h},z_{\tau h})$ satisfies the discrete stationarity conditions
$\mathcal L'_{\tau h}(u_{\tau h},z_{\tau h})=0$; the concrete assembly of
$A_{\tau h}$ and $F_{\tau h}$ is detailed later in Sec.~\ref{sec:st-discretization}.

\begin{theorem}[Goal-oriented error representation]\label{thm:go-error-rep}
Let $(u, z)$, $(u_\tau, z_\tau)$, and $(u_{\tau h}, z_{\tau h})$ be stationary points
of $\mathcal L$, $\mathcal L_\tau$, and $\mathcal L_{\tau h}$, respectively. Then,
for arbitrary $(\tilde u_\tau, \tilde z_\tau)$ in the time-semi-discrete space
and $(\tilde u_{\tau h}, \tilde z_{\tau h})$ in the fully discrete space, the
goal-error admits the representations
\begin{subequations}\label{eq:go-rep}
\begin{align}
  J(u) - J(u_\tau)
  &= \tfrac12\,\rho_\tau(u_\tau)(z - \tilde z_\tau)
   + \tfrac12\,\rho_\tau^{\ast}(u_\tau, z_\tau)(u - \tilde u_\tau)
   + \mathcal R_\tau\,,\label{eq:go-rep-tau}\\
  J(u_\tau) - J(u_{\tau h})
  &= \tfrac12\,\rho_\tau(u_{\tau h})(z_\tau - \tilde z_{\tau h})
   + \tfrac12\,\rho_\tau^{\ast}(u_{\tau h}, z_{\tau h})(u_\tau - \tilde u_{\tau h})
   + \mathcal R_h\,,\label{eq:go-rep-h}
\end{align}
\end{subequations}
where $\mathcal R_\tau$ and $\mathcal R_h$ are higher-order remainders with respect
to $(u - u_\tau,\, z - z_\tau)$ and $(u_\tau - u_{\tau h},\, z_\tau - z_{\tau h})$,
respectively.
\end{theorem}

The choices of $(\tilde u_\tau, \tilde z_\tau)$ and $(\tilde u_{\tau h}, \tilde z_{\tau h})$
are free and will be instantiated later to obtain temporally and directionally
split estimators. In particular, we will use temporal enrichments via lifting
and anisotropic spatial restrictions to separate the goal error into temporal
and directional spatial parts suitable for anisotropic $hp$ marking. %

\section{Anisotropic Error Estimation}\label{sec:anisotropic-ee}
Building on the error representation (Sec.~\ref{sec:error_rep}) and the
space-time discretization (Sec.~\ref{sec:st-discretization}), we use temporal
enrichment and anisotropic spatial restriction to obtain a directional split of
the goal error. The resulting error estimators drive the anisotropic $hp$
marking strategy.

We proceed as follows. First, we define anisotropic restriction operators in
space and a temporal enrichment by lifting (Sec.~\ref{subsection:AnIsoHighInt}).
Next, we connect enrichment and restriction to the DWR representation to obtain
temporally and directionally split indicators. These indicators control the
anisotropic $hp$ marking strategy described in Sec.~\ref{sec:algorithm}.

\subsection{Anisotropic Restriction Operations}\label{subsection:AnIsoHighInt}
In this section, we use an anisotropic broken $L^2$-projection into a directionally restricted finite element space as restriction operator.

  \begin{definition}[The restriction]\label{def:restriction}
  Let $v_h \in \mathcal{V}_h^{\boldsymbol{p}+\boldsymbol{k}}$ as in~\eqref{eq:p-aniso-gl+1}.
  The local restriction operator $\Riso{p_K,k}:\mathbb{Q}_{p_K+k}(K) \to \mathbb{Q}_{p_K}(K)$ is the local $L^2$-projection, i.e., $\Riso{p_K,k}v_h \in  \mathbb{Q}_{p_K}(K)$ solves
  \begin{equation}
    \int_{K} (\Riso{p_K,k} v_h) \, \phi_h \, \mathrm{d}x = \int_{K} v_h \, \phi_h \, \mathrm{d}x \qquad \forall \phi_h \in  \mathbb{Q}_{p_K}(K),
  \end{equation}
  for a fixed $K \in \mathcal{T}_h$.
  The global restriction operator $\Riso{\boldsymbol
    p,\boldsymbol{k}}:\mathcal{V}_h^{\boldsymbol{p}+\boldsymbol{k}} \to
  \mathcal{V}_h^{\boldsymbol{p}}$ is realized elementwise using the broken $L^2$-projection, i.e., $\Riso{\boldsymbol p,\boldsymbol{k}}v_h \in \mathcal{V}_h^{\boldsymbol{p}}$ solves
  \begin{equation}
    \int_{K} (\Riso{\boldsymbol p,\boldsymbol{k}} v_h) \, \phi_h \, \mathrm{d}x = \int_{K} v_h \, \phi_h \, \mathrm{d}x \qquad \forall \phi_h \in \mathcal{V}_h^{\boldsymbol{p}}, \; \forall K \in \mathcal{T}_h.
  \end{equation}
\end{definition}

  \begin{definition}[The directional restriction]\label{def:directional-restriction}
  For a fixed $k\in\mathbb N$ and $i\in\{1,\dots,d\}$, let $v_h \in
  \mathcal{V}_h^{\boldsymbol{p}+\boldsymbol{k}}$ (cf.~\eqref{eq:p-aniso-gl+1-ei}). The local
  directional restriction operator
  $\Rdir{p_K,k}{i}:\mathbb{Q}_{p_K+k}(K)\to \mathbb{Q}_{p_K+k-ke_i}(K)$
  (cf.~\eqref{eq:p-aniso-gl+1-ei}) is the local $L^2$-projection, i.e.,
  $\Rdir{p_K,k}{i}v_h \in \mathbb{Q}_{p_K+k-ke_i}(K)$ solves
  \[
    \int_{K} (\Rdir{p_K,k}{i} v_h)\,\phi_h\,dx = \int_{K} v_h\,\phi_h\,dx
    \qquad \forall \phi_h \in \mathbb{Q}_{p_K+k-ke_i}(K).
  \]
  The global directional restriction operator $\Rdir{\boldsymbol p,\boldsymbol k}{i}:\mathcal{V}_h^{\boldsymbol{p}+\boldsymbol{k}} \to \mathcal{V}_h^{\boldsymbol{p}+\boldsymbol{k},i}$
  is realized elementwise by the broken $L^2$-projection, i.e., $\Rdir{\boldsymbol p,\boldsymbol k}{i}v_h \in \mathcal{V}_h^{\boldsymbol{p}+\boldsymbol{k},i}$ satisfies
  \[
    \int_{K} (\Rdir{\boldsymbol p,\boldsymbol k}{i} v_h)\,\phi_h\,dx = \int_{K} v_h\,\phi_h\,dx
    \qquad \forall \phi_h \in \mathcal{V}_h^{\boldsymbol{p}+\boldsymbol{k},i},\ \forall K\in\mathcal T_h.
  \]
\end{definition}

Based on Definitions~\ref{def:restriction} and~\ref{def:directional-restriction},
we introduce the isotropic error operator as the remainder obtained by
comparing the isotropic restriction with the family of directional restrictions.
\begin{definition}[Isotropic Error Part]
  Identifying $\mathcal V_h^{\boldsymbol p}$ and
  $\mathcal V_h^{\boldsymbol p+\boldsymbol k,i}$ with subspaces of
  $\mathcal V_h^{\boldsymbol p+\boldsymbol k}$ via the canonical injection, we
  define
  $$
  \mathbb{E}^{\boldsymbol{p}+\boldsymbol{k}}
  :=(d-1)I+\Riso{\boldsymbol p,\boldsymbol k}-\sum_{i=1}^d \Rdir{\boldsymbol p,\boldsymbol
    k}{i}\,.
  $$
  The localized version $\mathbb{E}^{p_K+k}$ is given by
  $$\mathbb{E}^{{p_K}+{k}}:=(d-1)I+\Riso{p_K,k}-\sum_{i=1}^d \Rdir{p_K,k}{i}.$$
\end{definition}

\begin{lemma} \label{lemma: integral_identity}
For $k \in \mathbb{N}\cup \{0\}$, it holds
\begin{equation}
\int_{-1}^1 (1-x^2)^k\, \textrm{d} x = \textrm{B}\left(\frac{1}{2},k+1\right)=\frac{\sqrt{\pi}\Gamma\left(k+1\right)}{\Gamma\left(k+\frac{3}{2}\right)}=\frac{2^{2k}(k!)^2}{(2k+1)!},
\end{equation}
where $\textrm{B}$ is the beta function and $\Gamma$ is the gamma function.
\begin{proof}
    It holds
    \begin{equation}
        \begin{aligned}
            \int_{-1}^1 (1-x^2)^k\, \textrm{d} x =
            2\int_{0}^1 (1-x^2)^k\, \textrm{d} x =
            \int_{0}^1 t^{-\frac{1}{2}}(1-t)^k\, \textrm{d} t = \textrm{B}\left(\frac{1}{2},k+1\right).
        \end{aligned}
  \end{equation}
  The other equations follows from the properties of the beta function $\textrm{B}$, the gamma function $\Gamma$ and $\Gamma\left(\frac{1}{2}\right)=\sqrt{\pi}$.
\end{proof}
\end{lemma}

\begin{theorem} \label{thm: isotropic part}
    Let $K = \bigotimes_ {i=1}^d (0,h_i)$ and let $p_K \coloneq (p_1, \ldots, p_d)$.
    Then, for the  isotropic error $\Vert \mathbb{E}^{\boldsymbol{p}+\boldsymbol{1}}v_h \Vert_{L^2({K})}$, it holds
    \begin{equation}
      \Vert \mathbb{E}^{\boldsymbol{p}+\boldsymbol{1}}v_h \Vert_{L^2({K})} \leq \max_{i,j\in \{1,\dots,d\},\, i\not=j} C_{i,j}\Vert \partial_{x_i}^{p_i+1} \partial_{x_j}^{p_j+1} {v_h}({x})\Vert_{L^2(K)} h_i^{p_i+1}h_j^{p_j+1},
    \end{equation}
    with $v_h \in \mathcal{V}_h^{\boldsymbol{p}+\boldsymbol{1}}$
    and
    $$C_{i,j}^2\coloneq\frac{d^2(d-1)^2(2d-1)}{48(2p_i+2)!(2p_j+2)! 2^{(2p_i+2)+(2p_j+2)}} \textrm{B}\left(\frac{1}{2},p_i+2\right)\textrm{B}\left(\frac{1}{2},p_j+2\right).$$
    \begin{proof}
      First, let
      \[
        \A{p_K}:=\{\alpha\in\mathbb N_0^d:\ \alpha_i\le p_i\ \forall i\},
        \qquad
        \A{p_K+1}:=\{\alpha\in\mathbb N_0^d:\ \alpha_i\le p_i+1\ \forall i\},
      \]
      and for $k\in\{0,\dots,d\}$ define
      \(
        \AZ{p_K+1}{k}
        :=\left\{\alpha\in\A{p_K+1}:\ \big|\{i:\alpha_i=p_i+1\}\big|=k\right\}.
      \)
\begin{figure}[t]
  \centering
  \begin{tikzpicture}[scale=1.25,
    x=0.78cm, y=0.78cm,
    font=\small,
    >=stealth,
    axis/.style   ={line width=0.9pt, -stealth},
    frame/.style  ={line width=1.0pt},
    A0/.style     ={fill=blue!18,   draw=none},
    A1/.style     ={fill=orange!35, draw=none},
    A2/.style     ={fill=red!45,    draw=none},
    sep/.style    ={draw=gray!55, line width=0.9pt},
    lab/.style    ={align=center, inner sep=1.5pt}
  ]
  \def\P{7}
  \def\Q{4}
  \path[A0] (0,0) rectangle (\P-1,\Q-1);
  \path[A1] (\P-1,0) rectangle (\P,\Q-1);
  \path[A1] (0,\Q-1) rectangle (\P-1,\Q);
  \path[A2] (\P-1,\Q-1) rectangle (\P,\Q);
  \draw[frame] (0,0) rectangle (\P,\Q);
  \draw[sep] (\P-1,0) -- (\P-1,\Q);
  \draw[sep] (0,\Q-1) -- (\P,\Q-1);
  \draw[axis] (-0.55,0) -- (\P+0.75,0) node[right] {$\alpha_1$};
  \draw[axis] (0,-0.55) -- (0,\Q+0.75) node[above] {$\alpha_2$};
  \node[lab] at ({(\P-1)/2},{(\Q-1)/2})
    {$\AZ{p_K+1}{0}$};
  \node[lab] at (\P-0.5,{(\Q-1)/2}) {$\AZ{p_K+1}{1}$};
  \node[lab] at ({(\P-1)/2},\Q-0.5) {$\AZ{p_K+1}{1}$};
  \node[lab] at (\P-0.5,\Q-0.5) {$\AZ{p_K+1}{2}$};
  \node[
    anchor=south west,
    draw=gray!35,
    rounded corners=1.5pt,
    fill=white,
    fill opacity=0.94,
    text opacity=1,
    inner sep=3pt
  ] at ({\P+0.25},0.36) {%
  \begin{tabular}{@{}l@{\;}l@{}}
  \multicolumn{2}{@{}c@{}}{$\A{p_K+1}=\AZ{p_K+1}{0}\cup\AZ{p_K+1}{1}\cup\AZ{p_K+1}{2}$}\\[3pt]
  \tikz\fill[A0] (0,0) rectangle (0.25,0.25); &
  $\AZ{p_K+1}{0}=\A{p_K}$\\[3pt]
  \tikz\fill[A1] (0,0) rectangle (0.25,0.25); &
  $\AZ{p_K+1}{1}=\A{p_K+1}\setminus\left(\A{p_K}\cup \AZ{p_K+1}{2}\right)$\\[3pt]
  \tikz\fill[A2] (0,0) rectangle (0.25,0.25); &
  $\AZ{p_K+1}{2}=\{(p_1 + 1,p_2 + 1)\}$
  \end{tabular}
  };
  \end{tikzpicture}
  \caption{Visualization of the sets $\AZ{p_K+1}{k}$ for $d=2$. The union of all sets is $\A{p_K+1}$.\label{fig:multiindices}}
\end{figure}
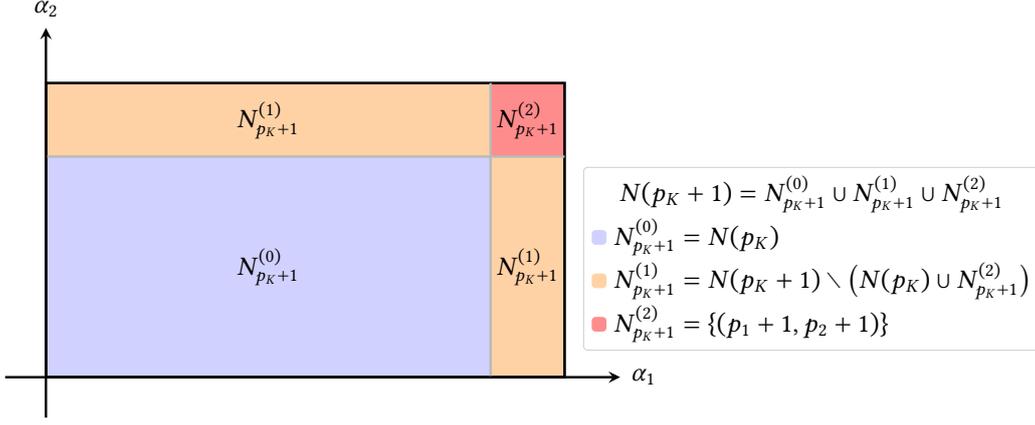
See Fig.~\ref{fig:multiindices} for a visualization of these multiindex sets for $d=2$.
    Furthermore, let $T_K: \hat{K} \to K$ be implicitly defined via
    \begin{equation}
      [T_K^{-1}(x)]_i = \hat x_i \coloneq \frac{2x_i-h_i}{h_i}.
    \end{equation}
    For $v_h \in \mathcal{V}_h^{\boldsymbol{p}+\boldsymbol{1}}$ and an arbitrary but fixed $K \in \mathcal{T}_h$, we define $\hat{v}\coloneq v_{h,|K} \circ T_K$.
    Furthermore, let $L_\alpha:\hat{K} \to \mathbb{R}$ be the
    tensor-product Legendre-Polynomial of degree $\alpha \in \mathbb{N}^d$; see \cite{Schwab1998_hpFEM}. Since $\hat{v}$ is a multivariate polynomial of degree $p_K+1=(p_1+1,\dots,p_d+1)$, the polynomial $\hat{v}$ can be written as
  \begin{equation}
  \hat{v}(\hat{x})=\sum_{\boldsymbol{\alpha}\in \A{p_K+1}} \innerpart L_{\boldsymbol{\alpha}}(\hat{x}) \quad \textrm{ with} \quad \innerpart :=  \int_{\hat{K}} \hat{v}(\hat{y}) L_{\boldsymbol{\alpha}}(\hat{y})\, \textrm{d} \hat{y} \prod_{i=1}^d \frac{2\alpha_i+1}{2},
\end{equation}
where $\boldsymbol{\alpha}\coloneq(\alpha_1, \dots, \alpha_d) \in \mathbb{N}_0^d$
and $L_{\boldsymbol{\alpha}}(\hat{x})\coloneq \prod_{i=1}^{d}
L_{\alpha_i}(\hat{x}_i)$.
With a slight abuse of notation, we denote by $\mathrm{R}^{p_K}_{h}$ and
$\mathrm{R}^{p_K}_{h,i}$ also the corresponding projectors on $\hat K$
obtained by pullback with $T_K$.
Then, one can show that
\begin{equation}
  \mathrm{R}^{p_K}_{h} \hat{v}(\hat{x})\coloneq
  \sum_{\boldsymbol{\alpha}\in \A{p_K}} \innerpart L_{\boldsymbol{\alpha}}(\hat{x}),
\end{equation}
and
\begin{equation}
  \mathrm{R}^{p_K}_{h,i} \hat{v}(\hat{x})\coloneq
  \sum_{{\boldsymbol{\alpha}\in \A{p_K+1-e_i}}} \innerpart L_{\boldsymbol{\alpha}}(\hat{x}).
\end{equation}
From the definition of $\mathbb{E}^{\boldsymbol{p}+\boldsymbol{1}}$ we know that
\begin{equation}
  \begin{aligned}
    \mathbb{E}^{p_K+1}:=(d-1)I+\mathrm{R}^{p_K}_{h}-\sum_{i=1}^d \mathrm{R}^{p_K}_{h,i}.
  \end{aligned}
\end{equation}
Therefore, we get that
\begin{equation}\label{eq:proof:Ep}
  \begin{aligned}
    \mathbb{E}^{\boldsymbol{p}+\boldsymbol{1}}\hat{v}(\hat{x})&\coloneq(d-1)\sum_{\boldsymbol{\alpha} \in \A{p_K+1}} \innerpart L_{\boldsymbol{\alpha}}(\hat{x})
       +\sum_{\boldsymbol{\alpha}\in \A{p_K}} \innerpart L_{\boldsymbol{\alpha}}(\hat{x})-\sum_{i=1}^d \sum_{\boldsymbol{\alpha}\in \A{p_K+1-e_i}} \innerpart L_{\boldsymbol{\alpha}}(\hat{x})
      \\&=d\sum_{\boldsymbol{\alpha} \in \A{p_K+1}} \innerpart L_{\boldsymbol{\alpha}}(\hat{x})
      \\&\qquad-\sum_{\boldsymbol{\alpha}\in (\A{p_K+1}\setminus\A{p_K})} \innerpart L_{\boldsymbol{\alpha}}(\hat{x})-\sum_{i=1}^d \sum_{\boldsymbol{\alpha}\in \A{p_K+1-e_i}} \innerpart L_{\boldsymbol{\alpha}}(\hat{x})
      \\&=\sum_{i=1}^d\sum_{\boldsymbol{\alpha} \in (\A{p_K+1}\setminus \A{p_K+1-e_i})} \innerpart L_{\boldsymbol{\alpha}}(\hat{x})
    -\sum_{\boldsymbol{\alpha}\in (\A{p_K+1}\setminus\A{p_K})} \innerpart L_{\boldsymbol{\alpha}}(\hat{x}).
  \end{aligned}
\end{equation}
We observe that $\A{p_K+1}\setminus \A{p_K}=\{\boldsymbol{\alpha}\in \A{p_K+1}: \exists i \in \{1,\ldots,d\},\, \alpha_i=p_i+1 \}$ and $\A{p_K+1}\setminus \A{p_K+1-e_i}=\{\boldsymbol{\alpha}\in \A{p_K+1}: \alpha_i=p_i+1 \}$, see also~Fig.~\ref{fig:multiindices}.
Then, we see that
\begin{equation}\label{eq:proof:Ep single faces}
  \begin{aligned}
    \sum_{i=1}^d\sum_{\boldsymbol{\alpha} \in (\A{p_K+1}\setminus \A{p_K+1-e_i})} \innerpart L_{\boldsymbol{\alpha}}(\hat{x})=
    \sum_{k=1}^d k \sum_{\boldsymbol{\alpha} \in \AZ{p_K+1}{k}}\innerpart L_{\boldsymbol{\alpha}}(\hat{x}),
  \end{aligned}
\end{equation}
and
\begin{equation}\label{eq:proof:Ep all max faces}
  \begin{aligned}
    \sum_{\boldsymbol{\alpha} \in (\A{p_K+1}\setminus \A{p_K})} \innerpart L_{\boldsymbol{\alpha}}(\hat{x})=
    \sum_{k=1}^d \sum_{\boldsymbol{\alpha} \in \AZ{p_K+1}{k}}\innerpart L_{\boldsymbol{\alpha}}(\hat{x}).
    \end{aligned}
  \end{equation}
Using~\eqref{eq:proof:Ep all max faces} and~\eqref{eq:proof:Ep single faces} we can rewrite~\eqref{eq:proof:Ep} as
\begin{equation} \label{eqn: split of Ev}
  \begin{aligned}
    \mathbb{E}^{\boldsymbol{p}+\boldsymbol{1}}\hat{v}(\hat{x})\coloneq	\sum_{k=1}^d (k-1) \sum_{\boldsymbol{\alpha} \in \AZ{p_K+1}{k}}\innerpart L_{\boldsymbol{\alpha}}(\hat{x})
    = \sum_{k=1}^{d-1} k \sum_{\boldsymbol{\alpha} \in \AZ{p_K+1}{k+1}}\innerpart L_{\boldsymbol{\alpha}}(\hat{x}).
  \end{aligned}
\end{equation}
Let $\hat{x}_i'\coloneq (\hat x_1, \ldots, \hat x_{i-1}, \hat x_{i+1}, \ldots, \hat x_d)$.
Using \cite[Lemma 3.10]{Schwab1998_hpFEM} and the tensor product property of $L_{\boldsymbol{\alpha}}$ and $\hat{K}$
we get that for $\alpha$ with $\alpha_i=p_i+1$
\begin{equation}
     \sum_{\substack{\boldsymbol\alpha \\ \alpha_i=p_i+1}} \innerpart^2 \prod_{\substack{j=1}}^d \frac{2}{2\alpha_j+1}=
     \frac{1}{(2p_i+2)!}\int_{ \bigotimes_{\substack{j=1 \\j \not=i}}^d[-1,1]} \int_{-1}^1 |\partial_{ \hat{x}_i}^{p_i+1}\hat{v}(\hat{x})|^2 (1-\hat{x}_i^2)^k \,\textrm{d} \hat{x}_i \textrm{d} \hat{x}_i'.
 \end{equation}
 In combination with Lemma~\ref{lemma: integral_identity} and that $\partial_{\hat{x}_i}^{p_i+1}\hat{v}(\hat{x})$ is constant with respect to $\hat{x}_i$, it holds
 \begin{equation}
     \sum_{\substack{\boldsymbol\alpha \\ \alpha_i=p_i+1}} \innerpart^2 \prod_{\substack{j=1}}^d \frac{2}{2\alpha_j+1}= \frac{1}{(2p_i+2)!} \textrm{B}\left(\frac{1}{2},p_i+2\right)
     \int_{ \bigotimes_{\substack{j=1 \\j \not=i}}^d[-1,1]}  |\partial_{ \hat{x}_i}^{p_i+1}\hat{v}(\hat{x})|^2 \,\textrm{d} \hat{x}_i'.
 \end{equation}
 Finally, we can deduce that
  \begin{equation}
     \sum_{\substack{\boldsymbol\alpha \\ \alpha_i=p_i+1}} \innerpart^2 \prod_{\substack{j=1}}^d \frac{2}{2\alpha_j+1}= \frac{1}{(2p_i+2)!2} \textrm{B}\left(\frac{1}{2},p_i+2\right)
     \int_{ \hat{K}}  |\partial_{ \hat{x}_i}^{p_i+1}\hat{v}(\hat{x})|^2 \,\textrm{d} \hat{x}.
 \end{equation}
For $\alpha$ where $\alpha_i=p_i+1$ and $\alpha_j=p_j+1$, one can show in a similar fashion
\begin{equation}
     \sum_{\substack{\boldsymbol\alpha \\ \alpha_i=p_i+1\\\alpha_j=p_j+1}} \innerpart^2 \prod_{\substack{k=1}}^d \frac{2}{2\alpha_k+1}= \frac{1}{(2p_i+2)!(2p_j+2)!4} \textrm{B}\left(\frac{1}{2},p_i+2\right)\textrm{B}\left(\frac{1}{2},p_j+2\right)
     \int_{ \hat{K}}  |\partial_{ \hat{x}_i}^{p_i+1}\partial_{ \hat{x}_j}^{p_j+1}\hat{v}(\hat{x})|^2 \,\textrm{d} \hat{x}.
 \end{equation}
 Finally from \cite[(3.3.6)]{Schwab1998_hpFEM} and \eqref{eqn: split of Ev}, we know that
 \begin{equation}
     \Vert \mathbb{E}^{\boldsymbol{p}+\boldsymbol{1}}\hat{v} \Vert_{L^2(\hat{K})}^2=\sum_{k=1}^{d-1} k^2 \sum_{\boldsymbol{\alpha} \in \AZ{p_K+1}{k+1}}\innerpart^2 \prod_{\substack{l=1}}^d \frac{2}{\alpha_l+1}.
 \end{equation}
 This concludes that
 \begin{equation}
     \Vert \mathbb{E}^{\boldsymbol{p}+\boldsymbol{1}}\hat{v} \Vert_{L^2(\hat{K})}^2\leq \sum_{k=1}^{d-1} k^2 \max_{\scriptsize\substack{\begin{aligned} i,j &\in \{1,\dots,d\}\\i &\neq j\end{aligned}}} \tilde{C}_{i,j}
  \Vert\partial_{ \hat{x}_i}^{p_i+1}\partial_{ \hat{x}_j}^{p_j+1}\hat{v}\Vert_{L^2(\hat{K})}^2,
 \end{equation}
 where $$\tilde{C}_{i,j}\coloneq\frac{d(d-1)}{(2p_i+2)!(2p_j+2)!8} \textrm{B}\left(\frac{1}{2},p_i+2\right)\textrm{B}\left(\frac{1}{2},p_j+2\right).$$
 Finally, we have
 \begin{equation}
 \begin{aligned}
      \Vert \mathbb{E}^{\boldsymbol{p}+\boldsymbol{1}}v_h \Vert_{L^2({K})}^2 &=\Vert \mathbb{E}^{\boldsymbol{p}+\boldsymbol{1}}\hat{v} \Vert_{L^2(\hat{K})}^2\prod_{i=1}^d\frac{h_i}{2}\\
      &\leq \frac{d(d-1)(2d-1)}{6}\max_{\scriptsize\substack{\begin{aligned} i,j &\in \{1,\dots,d\}\\i &\neq j\end{aligned}}}\tilde{C}_{i,j}
  \Vert\partial_{ \hat{x}_i}^{p_i+1}\partial_{ \hat{x}_j}^{p_j+1}\hat{v}\Vert_{L^2(\hat{K})}^2\prod_{k=1}^d\frac{h_k}{2}\\
  &=  \frac{d(d-1)(2d-1)}{6}\max_{\scriptsize\substack{\begin{aligned} i,j &\in \{1,\dots,d\}\\i &\neq j\end{aligned}}}\tilde{C}_{i,j}
  \Vert\partial_{{x}_i}^{p_i+1}\partial_{ {x}_j}^{p_j+1}{v_h}\Vert_{L^2({K})}^2 \left(\frac{h_i}{2}\right)^{2(p_i+1)}\left(\frac{h_j}{2}\right)^{2(p_j+1)}
 \end{aligned}
 \end{equation}
 The proof is completed by taking the square root.
    \end{proof}
\end{theorem}
  \begin{corollary} \label{coro: isotropic part 2d}
    Let $d=2$, $K = (0,h_1)\times (0,h_2)$ and $p_K \coloneq (p_1, p_2)$.
    Then, for the  isotropic error $\Vert \mathbb{E}^{\boldsymbol{p}+\boldsymbol{1}}v_h \Vert_{L^2({K})}$, it holds
    \begin{equation}
      \Vert \mathbb{E}^{\boldsymbol{p}+\boldsymbol{1}}v_h \Vert_{L^2({K})} = C_{1,2}\Vert \partial_{x_1}^{p_1+1} \partial_{x_2}^{p_2+1} {v_h}({x})\Vert_{L^2(K)} h_1^{p_1+1}h_2^{p_2+1},
    \end{equation}
    with $v_h \in \mathcal{V}_h^{\boldsymbol{p}+\boldsymbol{1}}$
    and
    $$C_{1,2}^2\coloneq\frac{1}{4(2p_1+2)!(2p_2+2)! 2^{(2p_1+2)+(2p_2+2)}} \textrm{B}\left(\frac{1}{2},p_1+2\right)\textrm{B}\left(\frac{1}{2},p_2+2\right).$$
    \begin{proof}
    Follow the same steps as in Theorem~\ref{thm: isotropic part} and verify that the inequalities become equalities.
    \end{proof}
    \end{corollary}
\begin{remark}
Even though Theorem~\ref{thm: isotropic part} and Corollary~\ref{coro: isotropic part 2d} are proven for
hyperrectangles \(K=\bigotimes_{i=1}^d(0,h_i)\), the same estimate extends to parallelepipeds, i.e.,
affine images of the reference cell \(\hat K\).
More precisely, let \(T_K:\hat K\to K\) be an affine bijection of the form
\(
T_K(\hat x)=B\hat x+b,
\)
with constant Jacobian matrix \(B\in\mathbb R^{d\times d}\).
Then the scaling of derivatives is governed by \(B^{-1}\),
and the constants in the above bounds depend on \(B\). Since the Jacobian is constant, the change-of-variables arguments used in the proof
carry over verbatim.
\end{remark}

\subsection{Spatial and Temporal enrichment}
\paragraph{Temporal enrichment via lifting}
In the $hp$-adaptive setting we use a local-in-time map that turns a piecewise
DG($k_n$) function into a continuous function of one higher local degree while
preserving the DG nodal values at the right Gauss–Radau points on each slab
$I_n=(t_{n-1},t_n]$. Following the Radau lifting
in~\cite[Sec.~3.2, (3.11)–(3.14)]{ern_discontinuous_2016}, we construct on every slab the \emph{unique}
polynomial $\vartheta_n\in\mathbb P_{k_n+1}(I_n)$ that equals $1$ at the left endpoint
$t_{n-1}^+$ and vanishes at the $(k_n+1)$ right-sided Gauss–Radau nodes
$\{t_{n,\mu}\}_{\mu=1}^{k_n+1}\subset\overline{I_n}$ (with $t_{n,k_{n}+1} =t_n$). Subtracting the temporal
jump multiplied by $\vartheta_n$ cancels the discontinuity while leaving all
Radau nodal values intact, and raises the local degree by one; see also the use
of $L_\tau$ in the DG-in-time formulation in~\cite[Lem.~3.2]{ern_discontinuous_2016}.

\begin{definition}[Lifting-based temporal enrichment]
Let $B$ be a Banach space and $\boldsymbol k=(k_n)_{n=1}^N$. For $w_\tau\in\mathcal X_\tau^{\boldsymbol k}(B)$ define $\mathfrak L_\tau w_\tau\in \mathcal X_\tau^{\boldsymbol k+\boldsymbol 1}(B)\cap C^0(\overline I;B)$ by
\[
(\mathfrak L_\tau w_\tau)(t):= w_\tau(t)\;-\;[w_\tau]_{n-1}\,\vartheta_n(t)\quad\text{for }t\in I_n,\qquad
(\mathfrak L_\tau w_\tau)(0):=w_\tau(0),
\]
where $[w_\tau]_{n-1}:=\lim_{t\to t_{n-1}^+}w_\tau(t)-\lim_{t\to t_{n-1}^-}w_\tau(t)$ and $\vartheta_n\in\mathbb P_{k_n+1}(I_n)$ satisfies
\[
\vartheta_n(t_{n-1}^+)=1,\qquad \vartheta_n(t_{n,\mu})=0\,,\quad\mu=1,\dots,k_n+1.
\]
\end{definition}
\begin{remark}[Properties of the temporal enrichment]
  \hfill
  \begin{description}[leftmargin=*,labelsep=1em,itemsep=3pt,parsep=0pt,topsep=3pt]
  \item[Continuity and interpolation.] On each slab $I_n$,
    $(\mathfrak L_\tau w_\tau)(t_{n-1}^+) = w_\tau(t_{n-1}^-)$ and
    $(\mathfrak L_\tau w_\tau)(t_{n,\mu}) = w_\tau(t_{n,\mu})$ for
    $\mu=1,\dots,k_n+1$, hence $\mathfrak L_\tau w_\tau$ is continuous on
    $\overline I$ and preserves all right Radau
    values~\cite[(3.12)–(3.14)]{ern_discontinuous_2016}.
  \item[Degree elevation.]
    $(\mathfrak L_\tau w_\tau)\restrict{I_n}\in\mathbb P_{k_n+1}(I_n;B)$, so the
    local temporal degree increases by
    one~\cite[(3.11)–(3.12)]{ern_discontinuous_2016}.
  \item[Locality and block structure.] $\mathfrak L_\tau$ acts slabwise and
    depends only on the jump $[w_\tau]_{n-1}$; globally it is block-diagonal
    with respect to the slabs~\cite[Sec.~3.2]{ern_discontinuous_2016}.
  \item[$L^2$-stability.] There exists $C_L\ge 1$, depending only on $\max_n k_n$ and the (reference) Radau nodes, such that
    \[
      \|\mathfrak L_\tau v\|_{L^2(I;B)} \le C_L\,\|v\|_{L^2(I;B)}\qquad\forall v\in\mathcal X_\tau^{\boldsymbol k}(B).
    \]
    The constant is independent of $h$ and of the time-step sizes. This follows
    by scaling of~\cite[(3.13)]{ern_discontinuous_2016} and standard DG
    jump–trace bounds (see the discussion
    in~\cite[Sec.~3.2]{ern_discontinuous_2016}).
  \item[Lifting on $\hat I$.] On the interval $\hat I=(-1,1)$ the lifting polynomial can be written explicitly as
    \[
      \hat{\vartheta}(\hat t) = \prod_{\mu=1}^{k_n+1}\frac{\hat t-\hat t_\mu}{-1-\hat t_\mu}\,,
    \]
    where $\{\hat t_\mu\}$ are the right Gauss–Radau nodes on $\hat I$ (cf.~\cite[(3.13)]{ern_discontinuous_2016}). Mapping $I_n\to\hat I$ yields
    $\vartheta_n$. The stability follows by scaling.
  \end{description}
\end{remark}

\paragraph{Spatial enrichment through higher-order solution}
We realize spatial enrichment by \emph{solving} the primal and the adjoint problems
on the higher-order space $\mathcal V_h^{\boldsymbol p+\boldsymbol 1}$. These
enriched solutions are used to build the error estimators.
\begin{definition}[Spatial enrichment]\label{def:spatial-enrichment}
  Let $\mathcal V_h^{\boldsymbol p}$ be the current spatial
  space~\eqref{eq:p-aniso-gl} and $\mathcal V_h^{\boldsymbol p+\boldsymbol 1}$
  the space with all local degrees raised by one~\eqref{eq:p-aniso-gl+1}. For a
  fixed temporal discretization and global bilinear form
  $A_{\tau h}(\cdot)(\cdot)$ with right-hand side $F_{\tau h}$, the \emph{primal}
  enriched solution
  $u_{\tau h}^{\boldsymbol p+\boldsymbol 1}\in\mathcal X_{\tau h}^{k,\boldsymbol p+\boldsymbol 1}$ is defined by
  \[
    A_{\tau h}(u_{\tau h}^{\boldsymbol p+\boldsymbol 1})(w)=F_{\tau h}(w)\qquad\forall\,w\in\mathcal X_{\tau h}^{k,\boldsymbol p+\boldsymbol 1}.
  \]
  For a goal functional $J$, the \emph{adjoint} enriched
  solution
  $z_{\tau h}^{\boldsymbol p+\boldsymbol 1}\in\mathcal X_{\tau h}^{k,\boldsymbol p+\boldsymbol 1}$ solves
  \[
    A_{\tau h}(w)(z_{\tau h}^{\boldsymbol p+\boldsymbol 1})=J'(u_{\tau h}^{\boldsymbol p+\boldsymbol 1})(w)\qquad\forall\,w\in\mathcal X_{\tau h}^{k,\boldsymbol p+\boldsymbol 1}.
  \]
\end{definition}

\paragraph{Connection to anisotropic DWR estimators}
We use the lifting $\mathfrak L_\tau$ as the concrete temporal enrichment in the
time–estimator components; i.e., in the formulas
\eqref{eq: implemented error estimator-tau} and
\eqref{eq: implemented error estimator-tau+1} we take
$\mathrm{E}_{\tau}^{(\boldsymbol k+\boldsymbol 1)}\coloneq\mathfrak L_\tau$.
On the spatial side, all enrichment terms are evaluated with the enriched
solutions $u_{\tau h}^{\boldsymbol p+\boldsymbol 1}$ and
$z_{\tau h}^{\boldsymbol p+\boldsymbol 1}$ from
Def.~\ref{def:spatial-enrichment}. Directional splitting is effected with the
directional restrictions $\mathrm{R}_{h,i}$, so that differences of the form
$v-\mathrm{R}_{h,i}v$ isolate the contribution in direction $i$.
This realization yields the temporal indicators $\eta_{\tau}^{\boldsymbol p}$,
$\eta_{\tau}^{\boldsymbol p+\boldsymbol 1}$ and the directional spatial
indicators $\{\eta_{h,i}^{\boldsymbol p+\boldsymbol 1}\}_{i=1}^d$ and
$\{\eta_{h,i}^{\boldsymbol p}\}_{i=1}^d$ used for anisotropic $hp$ marking.
Here, $\eta_{h,i}^{\boldsymbol p+\boldsymbol 1}$ is evaluated in the higher-order
space of degree $\boldsymbol p+\boldsymbol 1$, and $\eta_{h,i}^{\boldsymbol p}$ denotes
the corresponding indicator evaluated after restricting the discrete solution to
degree $\boldsymbol p$. This pair enables the definition of a local saturation
indicator, which measures how much of the estimated error is already captured on
the current space and how much is reduced by $p$-enrichment.

\subsection{Anisotropic Goal Oriented Error Estimation}

In this section, we briefly describe the theoretical background of the error estimator.
\begin{theorem}\label{thm:identity-space-time}
  Let $u_{\tau h}^{\boldsymbol p}$ be an approximation of the discrete solution and $u$ be the continuous solution. Furthermore let $u_{\tau h}^{\boldsymbol p+\boldsymbol 1}, z_{\tau h}^{\boldsymbol p},z_{\tau h}^{\boldsymbol p+\boldsymbol 1}$ and $z_{\tau h}^{\boldsymbol p}$ be arbitrary.
  Then it holds
   \begin{equation}
    \begin{aligned}
      J\left(u\right)-J\left(u_{\tau h}^{\boldsymbol p}\right)=J\left(u\right)
      -J\left(\mathrm{E}_{\tau}^{\left(\boldsymbol k+\boldsymbol 1\right)}u_{\tau h}^{\boldsymbol p+\boldsymbol 1}\right)+\eta_{h}^{\boldsymbol{p}+\boldsymbol{1}} +\eta_{h,\approx}^{\boldsymbol{p}+\boldsymbol{1}}
      +\eta_{\tau}^{\boldsymbol{p}+\boldsymbol{1}}
      +\eta_{\tau,\approx}^{\boldsymbol{p}+\boldsymbol{1}},
    \end{aligned}
  \end{equation}
  where
    \begin{alignat*}{2}
      \eta_{h}^{\boldsymbol{p}+\boldsymbol{1}} &\coloneq \eta_{\tau h}\left(u_{\tau h}^{\boldsymbol p},u_{\tau h}^{\boldsymbol p+\boldsymbol 1},z_{\tau h}^{\boldsymbol p+\boldsymbol 1},u_{\tau h}^{\boldsymbol p},z_{\tau h}^{\boldsymbol p}\right)\,, \\[3pt]
      \eta_{h,\approx}^{\boldsymbol{p}+\boldsymbol{1}} &\coloneq \eta_{\tau h,\approx}\left(u_{\tau h}^{\boldsymbol p},u_{\tau h}^{\boldsymbol p+\boldsymbol 1},z_{\tau h}^{\boldsymbol p+\boldsymbol 1},u_{\tau h}^{\boldsymbol p},z_{\tau h}^{\boldsymbol p}\right) +\eta_{k}\left(u_{\tau h}^{\boldsymbol p},z_{\tau h}^{\boldsymbol p}\right) +\eta_{\mathcal{R}}\left(u_{\tau h}^{\boldsymbol p+\boldsymbol 1},z_{\tau h}^{\boldsymbol p+\boldsymbol 1},u_{\tau h}^{\boldsymbol p},z_{\tau h}^{\boldsymbol p}\right) \,, \\[3pt]
      \eta_{\tau}^{\boldsymbol p+\boldsymbol 1} &\coloneq \eta_{\tau h}\left(u_{\tau h}^{\boldsymbol p+\boldsymbol 1},\mathrm{E}_{\tau}^{\left(\boldsymbol k+\boldsymbol 1\right)}u_{\tau h}^{\boldsymbol p+\boldsymbol 1},\mathrm{E}_{\tau}^{\left(\boldsymbol k+\boldsymbol 1\right)}z_{\tau h}^{\boldsymbol p+\boldsymbol 1},u_{\tau h}^{\boldsymbol p+\boldsymbol 1},z_{\tau h}^{\boldsymbol p+\boldsymbol 1}\right) \,, \\[3pt]
      \eta_{\tau,\approx}^{\boldsymbol{p}+\boldsymbol{1}} &\coloneq \eta_{\tau h,\approx}\left(u_{\tau h}^{\boldsymbol p+\boldsymbol 1},\mathrm{E}_{\tau}^{\left(\boldsymbol k+\boldsymbol 1\right)}u_{\tau h}^{\boldsymbol p+\boldsymbol 1},\mathrm{E}_{\tau}^{\left(\boldsymbol k+\boldsymbol 1\right)}z_{\tau h}^{\boldsymbol p+\boldsymbol 1},u_{\tau h}^{\boldsymbol p+\boldsymbol 1},z_{\tau h}^{\boldsymbol p+\boldsymbol 1}\right) +\eta_{k}\left(u_{\tau h}^{\boldsymbol p+\boldsymbol 1},z_{\tau h}^{\boldsymbol p+\boldsymbol 1}\right) 
      \\&+\eta_{\mathcal{R}}\left(\mathrm{E}_{\tau}^{\left(\boldsymbol k+\boldsymbol 1\right)}u_{\tau h}^{\boldsymbol p+\boldsymbol 1},\mathrm{E}_{\tau}^{\left(\boldsymbol k+\boldsymbol 1\right)}z_{\tau h}^{\boldsymbol p+\boldsymbol 1},u_{\tau h}^{\boldsymbol p},z_{\tau h}^{\boldsymbol p}\right)\,,
    \end{alignat*}
    and
   \begin{equation*}
    \begin{alignedat}{2}
      \eta_{\tau h}\left(u,v,z,v_{\tau h},z_{\tau_h}\right)&\coloneq&&\frac{1}{2}\rho_{\tau}\left(u\right)\left(z-z_{\tau h}\right)
      +
      \frac{1}{2}\rho_{\tau}^*\left(u,z_{\tau h}\right)\left(v-v_{\tau h}\right)\,,\\
      \eta_{\tau h,\approx}\left(u,v,z,v_{\tau h},z_{\tau_h}\right)&\coloneq{}&&
      \frac{1}{2}
      \rho_{\tau}\left(u\right)\left(z+z_{\tau h}\right)
      +\frac{1}{2}\rho_{\tau}^*\left(u,z\right)\left(v-v_{\tau h}\right)\,,\\
      \eta_{k}\left(u,z\right)&\coloneq&&
      -\rho_{\tau}\left(u\right)\left(z\right)\,,\\ %
      \eta_{\mathcal{R}}\left(u,z,u_{\tau h},z_{\tau h}\right)&\coloneq&&
      \int_{0}^{1} \big[J'''\left(u-s\left(u-u_{\tau h}\right)\right)\left(u-u_{\tau h},u-u_{\tau h},u-u_{\tau h}\right)\\
      &&&-A'''\left(u-s\left(u-u_{\tau h}\right)\right)\left(u-u_{\tau h},u-u_{\tau h},u-u_{\tau h},z-s\left(z-z_{\tau h}\right)\right)\\
      &&&-3A''\left(u-s\left(u-u_{\tau h}\right)\right)\left(u-u_{\tau h},u-u_{\tau h},z-z_{\tau h}\right)\big]\left(s-1\right)s\,\mathrm{d} s\,.
    \end{alignedat}
  \end{equation*}
  \begin{proof}
    Using Thm. 7 of~\cite{endtmayer_chapter_2024} we obtain
     \begin{equation*}
      \begin{alignedat}{2}
        J\left(u\right)-J\left(u_{\tau h}^{\boldsymbol p}\right)&={}&&J\left(u\right)
        -J\left(\mathrm{E}_{\tau}^{\left(\boldsymbol k+\boldsymbol 1\right)}u_{\tau h}^{\boldsymbol p+\boldsymbol 1}\right)
        +J\left(\mathrm{E}_{\tau}^{\left(\boldsymbol k+\boldsymbol 1\right)}u_{\tau h}^{\boldsymbol p+\boldsymbol 1}\right)
        -J\left(u_{\tau h}^{\boldsymbol p+\boldsymbol 1}\right)
        +J\left(u_{\tau h}^{\boldsymbol p+\boldsymbol 1}\right)-J\left(u_{\tau h}^{\boldsymbol p}\right)\\[3pt] %
        &={}&&J\left(u\right)
          -J\left(\mathrm{E}_{\tau}^{\left(\boldsymbol k+\boldsymbol 1\right)}u_{\tau h}^{\boldsymbol p+\boldsymbol 1}\right)\\
        &&& +
        \frac{1}{2}\rho_{\tau}\left(u_{\tau h}^{\boldsymbol p+\boldsymbol 1}\right)\left(\mathrm{E}_{\tau}^{\left(\boldsymbol k+\boldsymbol 1\right)}z_{\tau h}^{\boldsymbol p+\boldsymbol 1}-z_{\tau h}^{\boldsymbol p+\boldsymbol 1}\right)
        +
        \frac{1}{2}\rho_{\tau}^*\left(u_{\tau h}^{\boldsymbol p+\boldsymbol 1},z_{\tau h}^{\boldsymbol p+\boldsymbol 1}\right)\left(\mathrm{E}_{\tau}^{\left(\boldsymbol k+\boldsymbol 1\right)}u_{\tau h}^{\boldsymbol p+\boldsymbol 1}-u_{\tau h}^{\boldsymbol p+\boldsymbol 1}\right)\\
        &&&-	\rho_{\tau}\left(u_{\tau h}^{\boldsymbol p+\boldsymbol 1}\right)\left(z_{\tau h}^{\boldsymbol p+\boldsymbol 1}\right)
        +
        \frac{1}{2}
        \rho_{\tau}\left(\mathrm{E}_{\tau}^{\left(\boldsymbol k+\boldsymbol 1\right)}u_{\tau h}^{\boldsymbol p+\boldsymbol 1}\right)\left(\mathrm{E}_{\tau}^{\left(\boldsymbol k+\boldsymbol 1\right)}z_{\tau h}^{\boldsymbol p+\boldsymbol 1}+z_{\tau h}^{\boldsymbol p+\boldsymbol 1}\right) \\
        &&&+
            \frac{1}{2}\rho_{\tau}^*\left(\mathrm{E}_{\tau}^{\left(\boldsymbol k+\boldsymbol 1\right)}u_{\tau h}^{\boldsymbol p+\boldsymbol 1},\mathrm{E}_{\tau}^{\left(\boldsymbol k+\boldsymbol 1\right)}z_{\tau h}^{\boldsymbol p+\boldsymbol 1}\right)\left(\mathrm{E}_{\tau}^{\left(\boldsymbol k+\boldsymbol 1\right)}u_{\tau h}^{\boldsymbol p+\boldsymbol 1}-u_{\tau h}^{\boldsymbol p+\boldsymbol 1}\right)\\
                                          &&&+\eta_{\mathcal{R}}\left(\mathrm{E}_{\tau}^{\left(\boldsymbol k+\boldsymbol 1\right)}u_{\tau h}^{\boldsymbol p+\boldsymbol 1},\mathrm{E}_{\tau}^{\left(\boldsymbol k+\boldsymbol 1\right)}z_{\tau h}^{\boldsymbol p+\boldsymbol 1},u_{\tau h}^{\boldsymbol p+\boldsymbol 1},z_{\tau h}^{\boldsymbol p+\boldsymbol 1}\right)\\ %
        &&&+
        \frac{1}{2}\rho_{\tau}\left(u_{\tau h}^{\boldsymbol p}\right)\left(z_{\tau h}^{\boldsymbol p+\boldsymbol 1}-z_{\tau h}^{\boldsymbol p}\right)
        +
        \frac{1}{2}\rho_{\tau}^*\left(u_{\tau h}^{\boldsymbol p},z_{\tau h}^{\boldsymbol p}\right)\left(u_{\tau h}^{\boldsymbol p+\boldsymbol 1}-u_{\tau h}^{\boldsymbol p}\right)-	\rho_{\tau}\left(u_{\tau h}^{\boldsymbol p}\right)\left(z_{\tau h}^{\boldsymbol p}\right)\\
        &&&+
        \frac{1}{2}
        \rho_{\tau}\left(u_{\tau h}^{\boldsymbol p+\boldsymbol 1}\right)\left(z_{\tau h}^{\boldsymbol p+\boldsymbol 1}+z_{\tau h}^{\boldsymbol p}\right)
        +
          \frac{1}{2}\rho_{\tau}^*\left(u_{\tau h}^{\boldsymbol p+\boldsymbol 1},z_{\tau h}^{\boldsymbol p+\boldsymbol 1}\right)\left(u_{\tau h}^{\boldsymbol p+\boldsymbol 1}-u_{\tau h}^{\boldsymbol p}\right) +\eta_{\mathcal{R}}\left(u_{\tau h}^{\boldsymbol p+\boldsymbol 1},z_{\tau h}^{\boldsymbol p+\boldsymbol 1},u_{\tau h}^{\boldsymbol p},z_{\tau h}^{\boldsymbol p}\right)\\[3pt] %
        &={}&&J\left(u\right)
        -J\left(\mathrm{E}_{\tau}^{\left(\boldsymbol k+\boldsymbol 1\right)}u_{\tau h}^{\boldsymbol p+\boldsymbol 1}\right)
        +\eta_{\tau h}\left(u_{\tau h}^{\boldsymbol p+\boldsymbol 1},\mathrm{E}_{\tau}^{\left(\boldsymbol k+\boldsymbol 1\right)}u_{\tau h}^{\boldsymbol p+\boldsymbol 1},\mathrm{E}_{\tau}^{\left(\boldsymbol k+\boldsymbol 1\right)}z_{\tau h}^{\boldsymbol p+\boldsymbol 1},u_{\tau h}^{\boldsymbol p+\boldsymbol 1},z_{\tau h}^{\boldsymbol p+\boldsymbol 1}\right) \\
        &&&+\eta_{\tau h,\approx}\left(u_{\tau h}^{\boldsymbol p+\boldsymbol 1},\mathrm{E}_{\tau}^{\left(\boldsymbol k+\boldsymbol 1\right)}u_{\tau h}^{\boldsymbol p+\boldsymbol 1},\mathrm{E}_{\tau}^{\left(\boldsymbol k+\boldsymbol 1\right)}z_{\tau h}^{\boldsymbol p+\boldsymbol 1},u_{\tau h}^{\boldsymbol p+\boldsymbol 1},z_{\tau h}^{\boldsymbol p+\boldsymbol 1}\right) +\eta_{k}\left(u_{\tau h}^{\boldsymbol p+\boldsymbol 1},z_{\tau h}^{\boldsymbol p+\boldsymbol 1}\right)\\
        &&&+\eta_{\mathcal{R}}\left(\mathrm{E}_{\tau}^{\left(\boldsymbol k+\boldsymbol 1\right)}u_{\tau h}^{\boldsymbol p+\boldsymbol 1},\mathrm{E}_{\tau}^{\left(\boldsymbol k+\boldsymbol 1\right)}z_{\tau h}^{\boldsymbol p+\boldsymbol 1},u_{\tau h}^{\boldsymbol p+\boldsymbol 1},z_{\tau h}^{\boldsymbol p+\boldsymbol 1}\right) \\%spatial
        &&&
        +\eta_{\tau h}\left(u_{\tau h}^{\boldsymbol p},u_{\tau h}^{\boldsymbol p+\boldsymbol 1},z_{\tau h}^{\boldsymbol p+\boldsymbol 1},u_{\tau h}^{\boldsymbol p},z_{\tau h}^{\boldsymbol p}\right) +\eta_{\tau h,\approx}\left(u_{\tau h}^{\boldsymbol p},u_{\tau h}^{\boldsymbol p+\boldsymbol 1},z_{\tau h}^{\boldsymbol p+\boldsymbol 1},u_{\tau h}^{\boldsymbol p},z_{\tau h}^{\boldsymbol p}\right) \\
        &&&+\eta_{k}\left(u_{\tau h}^{\boldsymbol p},z_{\tau h}^{\boldsymbol p}\right)
        +\eta_{\mathcal{R}}\left(u_{\tau h}^{\boldsymbol p+\boldsymbol 1},z_{\tau h}^{\boldsymbol p+\boldsymbol 1},u_{\tau h}^{\boldsymbol p},z_{\tau h}^{\boldsymbol p}\right) \\
      \end{alignedat}
    \end{equation*}
  \end{proof}
\end{theorem}
\begin{remark}
  According to Thm.~\ref{thm:identity-space-time}, the error $J\left(u\right) - J\left(u_{\tau h}^{\boldsymbol{p}}\right)$ can be split into the error $J\left(u\right) - J\left(\mathrm{E}_{\tau}^{\left(\boldsymbol{k} + \boldsymbol{1}\right)} u_{\tau h}^{\boldsymbol{p}+\boldsymbol{1}}\right)$ and several error estimator components: $\eta_{h}^{\boldsymbol{p}+\boldsymbol{1}}$, $\eta_{h,\approx}^{\boldsymbol{p}+\boldsymbol{1}}$, $\eta_{\tau}^{\boldsymbol{p}+\boldsymbol{1}}$, and $\eta_{\tau,\approx}^{\boldsymbol{p}+\boldsymbol{1}}$. Here, $\eta_{h}^{\boldsymbol{p}+\boldsymbol{1}}$ and $\eta_{\tau}^{\boldsymbol{p}+\boldsymbol{1}}$ correspond to the discretization error in space and time, respectively. Therefore, $\eta_{h}^{\boldsymbol{p}+\boldsymbol{1}}$ and $\eta_{\tau}^{\boldsymbol{p}+\boldsymbol{1}}$ can be used to drive the adaptivity. The components $\eta_{h,\approx}^{\boldsymbol{p}+\boldsymbol{1}}$ and $\eta_{\tau,\approx}^{\boldsymbol{p}+\boldsymbol{1}}$ capture the inaccuracies related to not solving the problem at higher-orders in space and time, as well as the inaccuracies when solving for $u_{\tau h}^{\boldsymbol{p}}$ and higher-order components. These parts will be neglected for the adaptivity.
\end{remark}

\paragraph{Error estimator in practice}
Let $\mathrm{R}^{\boldsymbol p}_{h} \coloneq \mathrm{R}^{\boldsymbol p,1}_{h}$ and $\mathrm{R}^{\boldsymbol p}_{h,i} \coloneq \mathrm{R}^{\boldsymbol p,1}_{h,i}$.
For the error estimator in practice this, we replace $u_{\tau h}^{\boldsymbol p}$ by $\mathrm{R}^{\boldsymbol p}_{h}u_{\tau h}^{\boldsymbol p+\boldsymbol 1}$ and $z_{\tau h}^{\boldsymbol p}$ by $\mathrm{R}^{\boldsymbol p}_{h}z_{\tau h}^{\boldsymbol p+\boldsymbol 1}$.
The resulting discretization error estimator $\eta_h^{\boldsymbol p + \boldsymbol 1}$ is given by
\begin{equation}\label{eq:iso-error-est-h-iso}
    \eta_{h}^{\boldsymbol p+ \boldsymbol 1}\coloneq{}\frac{1}{2}\rho_{\tau}\left(\mathrm{R}^{\boldsymbol p}_{h}u_{\tau h}^{\boldsymbol p+\boldsymbol 1}\right)\left(z_{\tau h}^{\boldsymbol p+\boldsymbol 1}- \mathrm{R}^{\boldsymbol p}_{h}z_{\tau h}^{\boldsymbol p+\boldsymbol 1}\right)
     + \frac{1}{2}
    \rho_{\tau}^{\ast}\left(\mathrm{R}^{\boldsymbol p }_{h}u_{\tau h}^{\boldsymbol p ~ \boldsymbol 1},\mathrm{R}_{h}^{\boldsymbol p}z_{\tau h}^{\boldsymbol p+\boldsymbol 1}\right)
    \left(u_{\tau h}^{\boldsymbol p+\boldsymbol 1}-\mathrm{R}^{\boldsymbol p}_{h}u_{\tau h}^{\boldsymbol p+\boldsymbol 1}\right)
\end{equation}
\begin{theorem}\label{thm:anisotropic-error-split}
  Let $\eta_{h}^{\boldsymbol p+ \boldsymbol 1}$ be defined as in  \eqref{eq:iso-error-est-h-iso}.
  Then, $\eta_{h}^{\boldsymbol p+ \boldsymbol 1}$ can be decomposed as follows
  \begin{equation} \label{eq: implemented error estimate}
    \begin{aligned}
      \eta_{h}^{\boldsymbol p + \boldsymbol 1}\coloneq&
      \left(\sum_{i=1}^{d}\eta_{h,i}^{\boldsymbol p + \boldsymbol 1}\right)
      -
      \eta_{h,\mathbb{E}}^{\boldsymbol p + \boldsymbol 1}
      \,.
    \end{aligned}
  \end{equation}
  where for $i \in \{1\,,\dots, d\}$, we define
  \begin{equation}
    \begin{alignedat}{2}\label{eq:aniso-error-est-h-aniso-high}
      \eta_{h,i}^{\boldsymbol p + \boldsymbol 1}&\coloneq{}&&\frac{1}{2}\rho_{\tau}\left(\mathrm{R}^{\boldsymbol p}_{h}u_{\tau h}^{\boldsymbol p+\boldsymbol 1}\right)\left(z_{\tau h}^{\boldsymbol p+\boldsymbol 1}- \mathrm{R}^{\boldsymbol p + \boldsymbol {1}}_{h,i}z_{\tau h}^{\boldsymbol p+\boldsymbol 1}\right)
      + \frac{1}{2}
      \rho_{\tau}^{\ast}\left(\mathrm{R}^{\boldsymbol p}_{h}u_{\tau h}^{\boldsymbol p+\boldsymbol 1},\mathrm{R}_{h}^{\boldsymbol p}z_{\tau h}^{\boldsymbol p+\boldsymbol 1}\right)
      \left(u_{\tau h}^{\boldsymbol p+\boldsymbol 1}-\mathrm{R}^{\boldsymbol p + \boldsymbol {1}}_{h,i}u_{\tau h}^{\boldsymbol p+\boldsymbol 1}\right)
      \,,
    \end{alignedat}
  \end{equation}
  and
  \begin{equation}
    \begin{alignedat}{2}
      \eta_{h,\mathbb{E}}^{\boldsymbol p+\boldsymbol 1}&\coloneq{}&&\frac{1}{2}\rho_{\tau}\left(\mathrm{R}^{\boldsymbol p}_{h}u_{\tau h}^{\boldsymbol p+\boldsymbol 1}\right)\left(\mathbb{E}^{\boldsymbol p + \boldsymbol 1}z_{\tau h}^{\boldsymbol p+\boldsymbol 1}\right)
      + \frac{1}{2}
      \rho_{\tau}^{\ast}\left(\mathrm{R}^{\boldsymbol p}_{h}u_{\tau h}^{\boldsymbol p+\boldsymbol 1},\mathrm{R}_{h}^{p}z_{\tau h}^{\boldsymbol p+\boldsymbol 1}\right)
      \left(\mathbb{E}^{\boldsymbol p + \boldsymbol 1}u_{\tau h}^{\boldsymbol p+\boldsymbol 1}\right)
      \,.
    \end{alignedat}
  \end{equation}
\end{theorem}
\begin{proof}
  For $\eta_{h}^{\boldsymbol p}$ we follow the idea of  to~\cite{bause2025anisotropicspacetimegoalorientederror}.
  From the definition of $\mathbb{E}^{{p_K}+{1}}$ we know
  \begin{equation}
      \mathbb{E}^{{p_K}+{1}}:=\left(d-1\right)I+\mathrm{R}^{\boldsymbol p}_h-\sum_{i=1}^d \mathrm{R}^{\boldsymbol p + \boldsymbol {1}}_{h,i}.
  \end{equation}
  Subtracting $\mathbb{E}^{{p_K}+{1}}$ and adding $I-\mathrm{R}^{\boldsymbol p}_h$ yields
  \begin{equation*}
      I-\mathrm{R}^{\boldsymbol p}_h= -\mathbb{E}^{{p_K}+{1}} + dI-\sum_{i=1}^d \mathrm{R}^{\boldsymbol p + \boldsymbol {1}}_{h,i}=-\mathbb{E}^{{p_K}+{1}}+\sum_{i=1}^d I-\mathrm{R}^{\boldsymbol p + \boldsymbol {1}}_{h,i}.
  \end{equation*}
  The rest follows as in \cite{bause2025anisotropicspacetimegoalorientederror} by using the shown identity and the linearity of the parts $u_{\tau h}^{\boldsymbol p+\boldsymbol 1}- \mathrm{R}^{\boldsymbol p}_{h}u_{\tau h}^{\boldsymbol p+\boldsymbol 1}$ and $z_{\tau h}^{\boldsymbol p+\boldsymbol 1}- \mathrm{R}^{\boldsymbol p}_{h}z_{\tau h}^{\boldsymbol p+\boldsymbol 1}$ in \eqref{eq:iso-error-est-h-iso}.
\end{proof}
\begin{theorem}
    Let us assume there exists $\beta \in (0,1)$ and $\beta_{h,\tau} \in (0,\beta)$ such that 
    $$\left|J\left(u\right)
      -J\left(\mathrm{E}_{\tau}^{\left(\boldsymbol k+\boldsymbol 1\right)}u_{\tau h}^{\boldsymbol p+\boldsymbol 1}\right)\right|+\left|\eta_{h,\mathbb{E}}^{\boldsymbol{p}+\boldsymbol{1}} +\eta_{h,\approx}^{\boldsymbol{p}+\boldsymbol{1}}
      +\eta_{\tau,\approx}^{\boldsymbol{p}+\boldsymbol{1}}\right| \leq \beta_{h,\tau}\left|J\left(u\right)-J\left(\mathrm{R}^{\boldsymbol p}_{h}u_{\tau h}^{\boldsymbol p+\boldsymbol 1}\right)\right|\,.$$
      Then, it holds that 
      \begin{equation}
        \frac{1}{1+\beta_{h,\tau}}\left|\eta_{\tau}^{\boldsymbol p+\boldsymbol 1}+\sum_{i=1}^{d}\eta_{h,i}^{\boldsymbol p + \boldsymbol 1}\right| 
       \leq
       \left|J\left(u\right)-J\left(\mathrm{R}^{\boldsymbol p}_{h}u_{\tau h}^{\boldsymbol p+\boldsymbol 1}\right)\right|
       \leq
       \frac{1}{1-\beta_{h,\tau}}\left|\eta_{\tau}^{\boldsymbol p+\boldsymbol 1}+\sum_{i=1}^{d}\eta_{h,i}^{\boldsymbol p + \boldsymbol 1}\right| \,.
      \end{equation}
       \begin{proof}
           The proof is analogous to the proof of \cite[Theorem 3.9]{endtmayer2020two} and \cite[Theorem 3.12]{endtmayer2021reliability}.
       \end{proof}
\end{theorem}

\begin{remark}
    From Theorem~\ref{thm: isotropic part}, we know that $\eta_{h,\mathbb{E}}^{\boldsymbol p+\boldsymbol 1}$ is of higher order and is therefore neglected in further correspondence.
\end{remark}
\subsection{\label{sec:localization-aniso}Localization of the Anisotropic Error Estimators}

For localization, we use the partition of unity (PU) technique presented in~\cite{RiWi15_dwr}. We split our error in a temporal  and spatial directional contributions as
\begin{subequations}\label{eq: implemented error estimator}
  \begin{alignat}{2}\label{eq: implemented error estimator-tau}
    \eta_{\tau}^{\boldsymbol p}&\coloneq{}&&\frac{1}{2}\rho_{\tau}\left(\mathrm{R}_{h}^{\boldsymbol p}u_{\tau h}^{\boldsymbol p+\boldsymbol 1}\right)\left(\mathrm{E}_{\tau}^{\left(\boldsymbol k+\boldsymbol 1\right)}\mathrm{R}_{h}^{\boldsymbol p}z_{\tau h}^{\boldsymbol p+\boldsymbol 1}-\mathrm{R}_{h}^{\boldsymbol p}z_{\tau h}^{\boldsymbol p+\boldsymbol 1}\right)
    \\\nonumber&&&
    + \frac{1}{2}\rho_{\tau}^{\ast}\left(\mathrm{R}_{h}^{\boldsymbol p}u_{\tau h}^{\boldsymbol p+\boldsymbol 1},\mathrm{R}_{h}^{\boldsymbol p}z_{\tau
      h}^{\boldsymbol p+\boldsymbol 1}\right)\left(\mathrm{E}_{\tau}^{\left(\boldsymbol k+\boldsymbol 1\right)}\mathrm{R}_{h}^{\boldsymbol p}u_{\tau h}^{\boldsymbol p+\boldsymbol 1}-\mathrm{R}_{h}^{\boldsymbol p}u_{\tau h}^{\boldsymbol p+\boldsymbol 1}\right)\,,\\
      \label{eq:aniso-error-est-h-aniso-loc}
    \eta_{h,i}^{\boldsymbol p}&\coloneq{}&&\frac{1}{2}\rho_{\tau}\left(\mathrm{R}^{\boldsymbol p - \boldsymbol 1,2}_{h}u_{\tau h}^{\boldsymbol p+\boldsymbol 1}\right)\left(\mathrm{R}^{\boldsymbol p}_{h}z_{\tau h}^{\boldsymbol p+\boldsymbol 1}- \mathrm{R}^{\boldsymbol p}_{h,i}z_{\tau h}^{\boldsymbol p+\boldsymbol 1}\right)
    \\
    &&&
    + \frac{1}{2}
    \rho_{\tau}^{\ast}\left(\mathrm{R}^{\boldsymbol p - \boldsymbol 1,2}_{h}u_{\tau h}^{\boldsymbol p+\boldsymbol 1},\mathrm{R}_{h}^{\boldsymbol p - \boldsymbol 1,2}z_{\tau h}^{\boldsymbol p+\boldsymbol 1}\right)
    \left(\mathrm{R}^{\boldsymbol p}_{h}u_{\tau h}^{\boldsymbol p+\boldsymbol 1}-\mathrm{R}^{\boldsymbol p}_{h,i}u_{\tau h}^{\boldsymbol p+\boldsymbol 1}\right)
    \,, \\
      \label{eq: implemented error estimator-tau+1-loc}
    \eta_{\tau}^{\boldsymbol p+\boldsymbol 1}&\coloneq{}&&\frac{1}{2}\rho_{\tau}\left(u_{\tau h}^{\boldsymbol p+\boldsymbol 1}\right)\left(\mathrm{E}_{\tau}^{\left(\boldsymbol k+\boldsymbol 1\right)}z_{\tau h}^{\boldsymbol p+\boldsymbol 1}-z_{\tau h}^{\boldsymbol p+\boldsymbol 1}\right)
    + \frac{1}{2}\rho_{\tau}^{\ast}\left(u_{\tau h}^{\boldsymbol p+\boldsymbol 1},z_{\tau
      h}^{\boldsymbol p+\boldsymbol 1}\right)\left(\mathrm{E}_{\tau}^{\left(\boldsymbol k+\boldsymbol 1\right)}u_{\tau h}^{\boldsymbol p+\boldsymbol 1}-u_{\tau h}^{\boldsymbol p+\boldsymbol 1}\right)\,,\\
      \label{eq: implemented error estimator-h+1-loc}
    \eta_{h,i}^{\boldsymbol p +\boldsymbol 1}&\coloneq{}&&\frac{1}{2}\rho_{\tau}\left(\mathrm{R}^{\boldsymbol p}_{h}u_{\tau h}^{\boldsymbol p+\boldsymbol 1}\right)\left(z_{\tau h}^{\boldsymbol p+\boldsymbol 1}- \mathrm{R}^{\boldsymbol p + \boldsymbol {1}}_{h,i}z_{\tau h}^{\boldsymbol p+\boldsymbol 1}\right)
      + \frac{1}{2}
      \rho_{\tau}^{\ast}\left(\mathrm{R}^{\boldsymbol p}_{h}u_{\tau h}^{\boldsymbol p+\boldsymbol 1},\mathrm{R}_{h}^{\boldsymbol p}z_{\tau h}^{\boldsymbol p+\boldsymbol 1}\right)
      \left(u_{\tau h}^{\boldsymbol p+\boldsymbol 1}-\mathrm{R}^{\boldsymbol p + \boldsymbol {1}}_{h,i}u_{\tau h}^{\boldsymbol p+\boldsymbol 1}\right)
    \,.
  \end{alignat}
\end{subequations}
To localize the contributions in~\eqref{eq: implemented error estimator}, we use the PU-technique. In particular, we use the functions $\Psi_{\tau,n}\coloneq\mathcal{X}_{I_n\times \Omega}$ as the temporal partition of unity and   $\Psi_{h,K}\coloneq\mathcal{X}_{K\times [0,T]}$ as the spatial one.
Here, $\mathcal{X}_{S}$ denotes the characteristic function of the set $S$.
Of course $\sum_{n=1}^{N}\Psi_{\tau,n} \equiv 1$ and $\sum_{K\in \mathcal T_{h}}\Psi_{h,K} \equiv 1$.
This allows us to  recast the
estimators in~\eqref{eq: implemented error estimator} with the localized indicators
\begin{alignat}{2}
  \label{eq:localized-errors-aniso-t}
  \eta_{\tau}^{\boldsymbol p} & = \sum_{n=1}^{N}\eta_{\tau,n}^{\boldsymbol p}\,, \qquad\qquad &\eta_{\tau}^{\boldsymbol p+\boldsymbol 1} & = \sum_{n=1}^{N}\eta_{\tau,n}^{\boldsymbol p+\boldsymbol 1}\,,
  \intertext{and}
  \label{eq:localized-errors-aniso-s}
  \eta_{h,i}^{\boldsymbol p} & = \sum_{K\in\mathcal{T}_h} \eta_{h,i}^{p_K}
  \,, \qquad\qquad &\eta_{h,i}^{\boldsymbol p+\boldsymbol 1} & = \sum_{K\in\mathcal{T}_h} \eta_{h,i}^{p_K+1}.
\end{alignat}
where
\begin{subequations}\label{eq: implemented error estimatorloc}
  \begin{alignat}{2}\label{eq: implemented error estimator-tau_loc}
      \eta_{\tau,n}^{\boldsymbol p}&\coloneq{}&&\frac{1}{2}\rho_{\tau}\left(\mathrm{R}_{h}^{\boldsymbol p}u_{\tau h}^{\boldsymbol p+\boldsymbol 1}\right)\left(\left(\mathrm{E}_{\tau}^{\left(\boldsymbol k+\boldsymbol 1\right)}\mathrm{R}_{h}^{\boldsymbol p }z_{\tau h}^{\boldsymbol p+\boldsymbol 1}-\mathrm{R}_{h}^{\boldsymbol p}z_{\tau h}^{\boldsymbol p+\boldsymbol 1}\right)\Psi_{\tau,n}\right)
      \\ \nonumber&&&+ \frac{1}{2}\rho_{\tau}^{\ast}\left(\mathrm{R}_{h}^{\boldsymbol p }u_{\tau h}^{\boldsymbol p+\boldsymbol 1},\mathrm{R}_{h}^{\boldsymbol p}z_{\tau
                      h}^{\boldsymbol p+\boldsymbol 1}\right)\left(\left(\mathrm{E}_{\tau}^{\left(\boldsymbol k+\boldsymbol 1\right)}\mathrm{R}_{h}^{\boldsymbol p}u_{\tau h}^{\boldsymbol p+\boldsymbol 1}-\mathrm{R}_{h}^{\boldsymbol p}u_{\tau h}^{\boldsymbol p+\boldsymbol 1}\right)\Psi_{\tau,n}\right)\,,\\
    \label{eq: implemented error estimator-h aniso loc}
    \eta_{h,i}^{ p_K}&\coloneq{}&&\frac{1}{2}\rho_{\tau}\left(\mathrm{R}_{h}^{\boldsymbol p-\boldsymbol 1,2}u_{\tau h}^{\boldsymbol p+\boldsymbol 1}\right)
    \left(\left(\mathrm{R}_{h}^{\boldsymbol p}z_{\tau h}^{\boldsymbol p+\boldsymbol 1}- \mathrm{R}^{\boldsymbol p}_{h,i}z_{\tau h}^{\boldsymbol p+\boldsymbol 1}\right)\Psi_{h,K}\right)
    \\ \nonumber&&&
      + \frac{1}{2}
      \rho_{\tau}^{\ast}\left(\mathrm{R}_{h}^{\boldsymbol p-\boldsymbol 1,2}u_{\tau h}^{\boldsymbol p+\boldsymbol 1},\mathrm{R}_{h}^{\boldsymbol p -\boldsymbol 1,2}z_{\tau h}^{\boldsymbol p+\boldsymbol 1}\right)
      \left(\left(\mathrm{R}_{h}^{\boldsymbol p}u_{\tau h}^{\boldsymbol p+\boldsymbol 1}-\mathrm{R}^{\boldsymbol p}_{h,i}u_{\tau h}^{\boldsymbol p+\boldsymbol 1}\right)\Psi_{h,K}\right)
      \,,\\
    \label{eq: implemented error estimator-tau+1}
    \eta_{\tau,n}^{\boldsymbol p+\boldsymbol 1}&\coloneq{}&&\frac{1}{2}\rho_{\tau}\left(u_{\tau h}^{\boldsymbol p+\boldsymbol 1}\right)\left(\left(\mathrm{E}_{\tau}^{\left(\boldsymbol k+\boldsymbol 1\right)}z_{\tau h}^{\boldsymbol p+\boldsymbol 1}-z_{\tau h}^{\boldsymbol p+\boldsymbol 1}\right)\Psi_{\tau,n}\right)
    \\ \nonumber&&&+ \frac{1}{2}\rho_{\tau}^{\ast}\left(u_{\tau h}^{\boldsymbol p+\boldsymbol 1},z_{\tau
      h}^{\boldsymbol p+\boldsymbol 1}\right)
      \left(\left(\mathrm{E}_{\tau}^{\left(\boldsymbol k+\boldsymbol 1\right)}u_{\tau h}^{\boldsymbol p+\boldsymbol 1}-u_{\tau h}^{\boldsymbol p+\boldsymbol 1}\right)\Psi_{\tau,n}\right)\,,\\
      \label{eq: implemented error estimator-h aniso loc-high}
    \eta_{h,i}^{ p_K+1}&\coloneq{}&&
    \, \frac{1}{2}\rho_{\tau}\left(\mathrm{R}^{\boldsymbol p}_{h}u_{\tau h}^{\boldsymbol p+\boldsymbol 1}\right)\left(\left(z_{\tau h}^{\boldsymbol p+\boldsymbol 1}- \mathrm{R}^{\boldsymbol p + \boldsymbol {1}}_{h,i}z_{\tau h}^{\boldsymbol p+\boldsymbol 1}\right)\Psi_{h,K}\right)
      \\ \nonumber&&&+ \frac{1}{2}
      \rho_{\tau}^{\ast}\left(\mathrm{R}^{\boldsymbol p}_{h}u_{\tau h}^{\boldsymbol p+\boldsymbol 1},\mathrm{R}_{h}^{\boldsymbol p}z_{\tau h}^{\boldsymbol p+\boldsymbol 1}\right)
      \left(\left(u_{\tau h}^{\boldsymbol p+\boldsymbol 1}-\mathrm{R}^{\boldsymbol p + \boldsymbol {1}}_{h,i}u_{\tau h}^{\boldsymbol p+\boldsymbol 1}\right)\Psi_{h,K}\right)
      \,.
  \end{alignat}
\end{subequations}

\section{Algorithm for Goal-Oriented Anisotropic $hp$-Adaptivity}\label{sec:algorithm}
With the goal-oriented error representation and anisotropic indicators in place,
we now present the algorithm for goal-oriented anisotropic $hp$-adaptivity. The
elementwise contributions of the anisotropic error indicators in space are
given in~\eqref{eq:localized-errors-aniso-s}. Following Thm.~\ref{thm:anisotropic-error-split}, we
discard the higher-order spatial remainder of $\eta_h^{\boldsymbol p}$ and
$\eta_h^{\boldsymbol p+\boldsymbol 1}$. The temporal slabwise contributions are
given in~\eqref{eq:localized-errors-aniso-t}. The spatial and temporal
indicators drive the anisotropic mesh-adaptation strategy in
Alg.~\ref{alg:AnisotropicAMR}. By evaluating the directional contributions separately,
the marking procedure naturally yields anisotropic refinement when the error is dominated
by specific directions, and isotropic refinement when the contributions are balanced.
\begin{algorithm}[htb]
  \caption{\label{alg:AnisotropicAMR}Anisotropic mesh adaptation: estimate, mark \& refine/coarsen}
  \begin{algorithmic}[1]
    \REQUIRE{$\theta_h\in [0,d],\,\theta_\tau\in [0,1]$,
             $\theta_h^{co}\in [0,d-\theta_h]$,
             $\theta_\tau^{co}\in [0,1-\theta_\tau]$, threshold $\gamma>0$.}
    \STATE \textbf{SOLVE:}\label{line: solve}\\
    \STATE Solve \textbf{primal} problem~\eqref{eq:discrete_variational_abstract}: Find \(u^{\boldsymbol{p}}_{\tau h}\in\mathcal{X}_{\tau h}^{\boldsymbol k,\boldsymbol{p}}:  A_{\tau h}(u^{\boldsymbol{p}}_{\tau h}, w_{\tau h}) = F_{\tau h}(w_{\tau h}) \quad\forall w_{\tau h}\in\mathcal{X}_{\tau h}^{\boldsymbol k,\boldsymbol{p}}\).
    \STATE Solve \textbf{primal ho} problem~\eqref{eq:discrete_variational_abstract}: Find \(u^{\boldsymbol{p+1}}_{\tau h}\in\mathcal{X}_{\tau h}^{\boldsymbol k,\boldsymbol{p+1}}:  A_{\tau h}(u^{\boldsymbol{p+1}}_{\tau h}, w_{\tau h}) = F_{\tau h}(w_{\tau h}) \;\forall w_{\tau h}\in\mathcal{X}_{\tau h}^{\boldsymbol k,\boldsymbol{p+1}}\).
    \STATE Solve \textbf{adjoint} problem~\eqref{eq:discrete_adjoint_variational_abstract}: Find $z_{\tau h}\in \mathcal{X}_{\tau h}^{\boldsymbol k,\boldsymbol{p+1}}: A_{\tau h}(v_{\tau h},\,z_{\tau h}) =
    J^{\prime}(u_{\tau h})(v_{\tau h}) \quad \forall v_{\tau h}
    \in \mathcal{X}_{\tau h}^{\boldsymbol k,\boldsymbol{p+1}}$.
    \STATE \textbf{ESTIMATE:}
    \STATE Compute spatial error indicators
      $\eta^{p_K+1}_{h,i}$ and $\eta^{p_K}_{h,i}$ for all $K\in\mathcal T_h$, $i=1,\dots,d$
      (cf.~\eqref{eq:localized-errors-aniso-s}).
    \STATE Compute temporal error indicators
      $\eta^{\boldsymbol p+\boldsymbol 1}_{\tau,n}$ and $\eta^{\boldsymbol p}_{\tau,n}$ for all $I_n\in\mathcal T_\tau$
      (cf.~\eqref{eq:localized-errors-aniso-t}).

    \STATE \textbf{MARK (space):}
    \STATE Let $\mathcal M_h^{ref} \subset \mathcal T_h$ be the set of
    $\theta_h\,|\mathcal T_h|$ cells $K$ with largest $\eta_{h,i}^{p_K+1}$,
    over all directions $i$.
    \STATE Let $\mathcal M_h^{co}  \subset \mathcal T_h$ be the set of
    $\theta_h^{co}\,|\mathcal T_h|$ cells $K$ with smallest $\eta_{h,i}^{p_K+1}$,
    over all directions $i$.

    \STATE $\gamma^{p_K}_{h,i}\coloneq\eta^{p_K+1}_{h,i} / \eta^{p_K}_{h,i}$\label{alg:def:gammah}
    \STATE Mark $K\in\mathcal M_h^{ref}$ in direction $i$ for anisotropic
    $\begin{cases}
    h\text{-refinement} & \text{if }\gamma^{p_K}_{h,i}>\gamma,\\
    p\text{-refinement} & \text{otherwise.}
    \end{cases}$ \label{alg:mark-aniso-hp}
    \STATE Mark $K\in\mathcal M_h^{co\hphantom{f}}$ in direction $i$ for anisotropic
    $\begin{cases}
    h\text{-coarsening} & \text{if }\gamma^{p_K}_{h,i}<\gamma,\\
    p\text{-coarsening} & \text{otherwise.}
    \end{cases}$\label{alg:coar-aniso-hp}

    \STATE \textbf{MARK (time):}
    \STATE Let $\mathcal M_\tau^{ref} \subset \mathcal T_\tau$ be the set of
           $\theta_\tau\,|\mathcal T_\tau|$ subintervals $I_n$ with largest $\eta^{\boldsymbol p}_{\tau,n}$.
    \STATE Let $\mathcal M_\tau^{co}  \subset \mathcal T_\tau$ be the set of
           $\theta_\tau^{co}\,|\mathcal T_\tau|$ subintervals $I_n$ with smallest $\eta^{\boldsymbol p}_{\tau,n}$.

    \STATE $\gamma^n_\tau\coloneq \eta^{\boldsymbol p+\boldsymbol 1}_{\tau,n} / \eta^{\boldsymbol p }_{\tau,n}$\label{alg:def:gammatau}
    \STATE Mark $I_n\in \mathcal M_\tau^{ref}$ for $\tau$-refinement  if $\gamma^n_\tau > \gamma$ and for $k$-refinement if $\gamma^n_\tau \leq \gamma$.\label{alg:mark-aniso-tau}
    \STATE Mark $I_n\in \mathcal M_\tau^{co\phantom{f}}$ for $\tau$-coarsening if $\gamma^n_\tau < \gamma$ and for $k$-coarsening if $\gamma^n_\tau \geq \gamma$.
    \STATE \textbf{REFINE/COARSEN:} Adapt $(\mathcal T_h,\mathcal T_\tau)$ according to the spatial and temporal marks.\label{alg:coar-aniso-tau}
  \end{algorithmic}
\end{algorithm}
\begin{remark}[Anisotropic-$hp$ Adaptation]
~\\[-2\topsep]
\begin{itemize}\itemsep1pt \parskip0pt \parsep0pt
\item With the spatial mesh fixed across all time slabs, the directional spatial
  estimators \(\eta_{h,i}^{p_K+1}\) and \(\eta_{h,i}^{p_K}\) for each
  \(K\in\mathcal T_h\) and \(i\in\{1,\dots,d\}\) are defined on
  \(K\times[0,T]\), while the time–slab estimators
  \(\eta_{\tau}^{\boldsymbol p}\) and
  \(\eta_{\tau}^{\boldsymbol p+\boldsymbol 1}\) on \(I_n\) are defined on
  \(\Omega\times I_n\). We use the PU localization from
  Sec.~\ref{sec:localization-aniso}, which makes the definition of the localized
  error estimators straightforward using indicator functions. In the
  implementation they are calculated as sums of space–time cell contributions.
  The spatial estimators \(\eta_{h,i}^{p_K+1}\) and \(\eta_{h,i}^{p_K}\) for a
  cell \(K\in \mathcal{T}_h\) and spatial direction \(i\) are obtained by
  summing over all \(I_n\in\mathcal T_{\tau}\), whereas the temporal estimators
  \(\eta_{\tau,n}^{\boldsymbol p+\boldsymbol 1}\) and
  \(\eta_{\tau,n}^{\boldsymbol p}\) on \(I_n\) result from summing over all
  \(K\in\mathcal T_h\).
\item Marking for $hp$-refinement and coarsening is performed separately in each spatial
  direction (Alg.~\ref{alg:AnisotropicAMR}, l.~\ref{alg:mark-aniso-hp},~\ref{alg:coar-aniso-hp}). A cell \(K\) is marked for refinement in direction \(i\) if \(\eta^K_{h,i}\) belongs to the largest \(\theta_h|\mathcal T_h|\) values among all \(d|\mathcal T_h|\) directional indicators. We use the local saturation
  indicators \(\gamma^{p_K}_{h,i}\) and a prescribed
  threshold $\gamma\in (0,\,1)$ to decide between $h$- and $p$-refinement  (Alg.~\ref{alg:AnisotropicAMR}, l.~\ref{alg:def:gammah}).
  We apply local $p$-enrichment in direction $i$ if $\gamma^{p_K}_{h,i}<\gamma$, indicating significant reduction under enrichment, and local $h$-refinement otherwise. The same rule is used in time, see (Alg.~\ref{alg:AnisotropicAMR}, l.~\ref{alg:def:gammatau}
  to~\ref{alg:coar-aniso-tau})
\item Adaptive $h$-refinement on quadrilateral or hexahedral meshes leads to
  interdependent hanging nodes. To resolve these mutual dependencies, we refine the
  coarser element at the end of the hanging-node chain in the appropriate
  direction (i.\,e.\ anisotropically).
\item Adaptive \( p \)-refinement within discontinuous Galerkin (DG) methods
  is straightforward, as no continuity is enforced across element interfaces.
  However, care must be taken to correctly evaluate face integrals. In
  contrast, conforming methods require constraint equations when using
  distinct finite elements on two sides of a face to ensure continuity across
  faces with varying polynomial degrees~\cite{BK07}.
\end{itemize}
\end{remark}
To compute the error indicators required in Alg.~\ref{alg:AnisotropicAMR},
we first solve the primal problem forward in time, followed by the adjoint
problem backward in time. Using both primal and adjoint solutions, we calculate
the local error estimators~\eqref{eq:localized-errors-aniso-s} and~\eqref{eq:localized-errors-aniso-t}, perform marking
and refinement according to Alg.~\ref{alg:AnisotropicAMR}, and repeat the
space-time solution procedure iteratively until the desired accuracy is reached.
For algorithmic details and implementation aspects of
tensor-product space-time finite elements in the \texttt{deal.II} library, we refer
to~\cite{KoeBruBau2019,RoThiKoeWi23,ThiWi24,endtmayer_goal-oriented_2024}.
For implementation aspects of $hp$ methods in \texttt{deal.II} we refer
to~\cite{fehling_algorithms_2023,BK07}.

\section{Solver for the Algebraic Systems}\label{sec:solvers}
\begin{figure}
\begin{tikzpicture}[thick,scale=1.0, every node/.style={scale=0.8}]
  \draw(-0.5,0.8) |- (1.0,0.0);
  \draw(0.5,0) |- (1.0,-0.8);
  \draw(1.5,-0.6) |- (2.0,-1.7);
  \node[align=left,anchor=west, fill=myblue!10!white] at (-1.0,0.8) {Space-time
    system on $I_n$ with $\lvert I_n\rvert=\tau_n$ and order $k_n$};
  \node[align=left,anchor=west, fill=myblue!10!white] at (0,0) {GMRES with $
    {(\symbf M_{\tau}^{k_n})}^{-1}\symbf A_\tau^{k_n} \otimes \symbf M_h + \tau \mathbb{1}_{k_n+1} \otimes \symbf A_h$ };
  \node[align=left,anchor=west, fill=myblue!10!white] at (1.0,-0.8) {Block preconditioner $\symbf P=(\symbf S^{k_n} \otimes \mathbb{1}_{N_{\symbf x}}) (\symbf \Lambda^{k_n} \otimes \symbf M_h + \tau \mathbb{1} \otimes \symbf A_h)({(\symbf S^{k_n})}^{-1} \otimes \mathbb{1}_{N_{\symbf x}})$};
  \node[align=left,anchor=west, fill=mygreen!10!white] at (2.0,-1.8) {GMRES
    for each $l$ block $\lambda_l \symbf M_h + \tau \symbf A_h$ preconditioned by $\operatorname{ILU}(\tau \symbf A_h)$};
\end{tikzpicture}
\caption{\label{fig:linear-solve}Sketch of the linear solution process: The
  space-time system on $I_n$ is preconditioned by a GMRES method to which we
  apply a block preconditioner based on the diagonalization of
  $\operatorname{tril}({(\symbf M_{\tau}^{k_n})}^{-1}\symbf A_\tau)$. The diagonal blocks
  are then solved with few iterations of a preconditioned GMRES method.}
\end{figure}
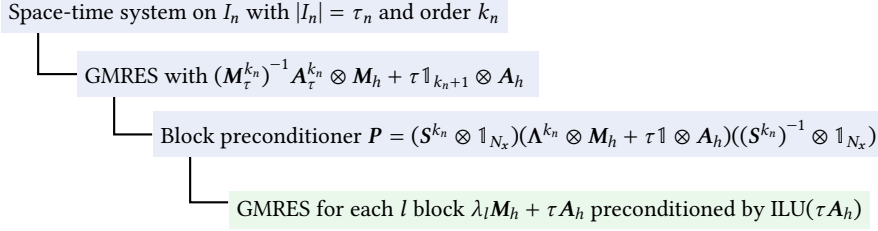

We use the decoupling of the system at interval endpoints and solve the
space–time system in Prob.~\ref{prob:primal-global} via a time–marching
approach. For the local problems on $I_n$ we use GMRES with a preconditioner
that further decouples the temporal degrees of freedom by an LU-then-diagonalize
construction of the temporal operator. This follows the ideas of Axelsson and
Neytcheva~\cite{algoritmy} and can later enable parallelism in
time~\cite{munch_stage-parallel_2023}. The general idea of decoupling large
block systems in space–time discretizations has been explored
earlier~\cite{werder_hp-discontinuous_2001,richter_efficient_2013}. In general, for $k\in \{0\}\cup\N$, let
\begin{equation}\label{eq:LU-unitU}
  \symbf G^{k} \coloneq {(\symbf M_\tau^{k})}^{-1}\symbf A_\tau^{k}\in\mathbb{R}^{(k+1)\times(k+1)},\qquad
  \symbf G^{k}=\symbf L^{k}\,\symbf U^{k}, \quad\operatorname{diag}(\symbf U^{k})=\mathbb{1}_{k+1},
\end{equation}
be an LU factorization with unit diagonal in the upper factor.
For a given slab \(I_n\), we multiply Prob.~\ref{prob:primal-local} by
${(\symbf M_\tau^{k_n})}^{-1}$. Then, the system matrix
in Prob.~\ref{prob:primal-local} can be written as
\begin{equation}
  \label{eq:primal-local-reform}
  \symbf G^{k_n} \otimes \symbf M_h + \tau\, \mathbb{1}_k \otimes
  \symbf A_h\,.
\end{equation}
We compute
an eigen-decomposition of $\symbf L^{k_n}$, the lower factor of $\symbf G^{k_n}$,
\(\symbf L^{k_n} = \symbf S^{k_n}\,\symbf \Lambda^{k_n}\,\symbf {(S^{k_n})}^{-1},\; \symbf \Lambda=\operatorname{diag}(\lambda_1,\dots,\lambda_{k_n+1})\).
This yields the block–diagonal preconditioner to the system~\eqref{eq:primal-local-reform}
\begin{equation}\label{eq:P}
  \symbf P \coloneq (\symbf S^{k_n}\otimes \mathbb{1}_{N_{\symbf x}})
  \bigl(\symbf \Lambda\otimes \symbf M_h + \mathbb{1}_k\otimes \tau \symbf A_h\bigr)
  ({(\symbf S^{k_n})}^{-1}\otimes \mathbb{1}_{N_{\symbf x}}),
\end{equation}
so that applying \(\symbf P^{-1}\) decouples the time dimension into \(k\)
independent spatial solves,
\begin{equation}\label{eq:spatial-blocks}
  (\lambda_\ell \symbf M_h + \tau \symbf A_h)\,\symbf z_\ell = \symbf r_\ell,\qquad \ell=1,\dots,k_n+1,
\end{equation}
each of which can be equipped with a suitable spatial preconditioner (e.g. an
ILU for \(\tau \symbf A_h\)). In our implementation, this preconditioning and
solver technique is applied to Prob.~\eqref{prob:adjoint-variational-local} in
an analogous manner.

\section{Numerical Examples}\label{sec:numerical_examples}
To validate the efficacy of our anisotropic \( hp \)-refinement strategy for time-dependent CDR equations, we present three benchmark problems:
\begin{enumerate}
\item \emph{Interior Layer Problem:} An exact solution is available and features
  a sharp interior layer, serving as a reference for accuracy and convergence
  studies.
\item \emph{Instationary Hemker Problem:} This classical benchmark models
  convection-dominated heat transfer around a hot cylindrical obstacle, leading
  to both boundary and interior layers.
\item \emph{Fichera corner:} In the typical domain with edge and corner
  singularities, an initial concentration is transported towards these
  singularities.
\end{enumerate}
In all cases, we apply anisotropic \( hp \)-refinement in space and time,
allowing for independent adjustment of the mesh size and polynomial degree in
each coordinate direction. No restrictions are imposed on the variation of
polynomial degrees within or across elements.
To assess the accuracy of the goal-oriented error estimator, we study the \emph{effectivity index}
\begin{equation}
  I_{\text{eff}} = \left\vert\frac{\eta_{h}^{\textup{a},\,\boldsymbol p+\boldsymbol 1} + \eta_{\tau}^{\boldsymbol p+\boldsymbol 1}}{J(u) - J(u_{\tau h})}\right\vert,
\end{equation}
where \(\eta_{h}^{\textup{a},\,\boldsymbol p+\boldsymbol 1}\coloneq\sum_{i=1}^{d}\eta_{h,i}^{\boldsymbol p+\boldsymbol 1}\). The sum of the spatial and temporal error estimators is
denoted by \(\eta_{\tau h}^{\textup{a},\,\boldsymbol p+\boldsymbol 1} =
\eta_{h}^{\textup{a},\,\boldsymbol p+\boldsymbol 1} + \eta_{\tau}^{\boldsymbol
  p+\boldsymbol 1}\). For the sake of brevity, we subsequently drop the $\boldsymbol
p+\boldsymbol 1$ superscript.
As an indicator of mesh anisotropy, we track the maximum aspect ratio,
defined by
\begin{equation}
  \rho_{\max} \coloneq \max_{K \in \mathcal T_h} \rho_K,
\end{equation}
where \(\rho_K\) is given by~\eqref{eq:loc-aniso-ratio}.
We denote the total number of DoFs by \(N_{\text{tot}}\), and
split them into spatial and temporal DoFs \(N_{\boldsymbol x}\) and \(N_t\), respectively.
All simulations are implemented using the \texttt{deal.II} finite element
library~\cite{africa_dealii_2024} and executed on one node of the HSUper
HPC cluster at Helmut Schmidt University, equipped with two Intel Xeon Platinum
8360Y CPUs and
\SI[scientific-notation=false,round-precision=0]{1024}{\giga\byte} of RAM. 

\subsection{Interior Layer}\label{sec:interior-layer}
This well-known benchmark has an exact solution
\begin{equation}
  \label{eq:step-layer}
  u(\symbf{x}, t) = \frac{\e^{3(t - 1.0)}}{2}\left( 1 - \tanh \frac{2x-y-\frac{1}{2}}{\sqrt{5\varepsilon}}\right)
\end{equation}
with a sharp interior layer, of thickness
$\mathcal{O}(\sqrt{\varepsilon}|\log \varepsilon|)$. The problem is defined on
$\Omega\times I=(0,\,1)^2\times (0,\,1]$ with inhomogeneous boundary
conditions given by~\eqref{eq:step-layer} on the whole boundary
$\Gamma_D=\partial \Omega$. The convection is set to
$\symbf{b}=\frac{1}{\sqrt{5}}(1,\,2)^{\top}$, the diffusion to
\(\varepsilon = 10^{-6}\) or \(\varepsilon = 10^{-8}\) and the reaction coefficient is given by $\alpha=1$.
As sketched in Fig.~\ref{fig:interior-layer-geo}, we employ unstructured meshes to solve the problem. We consider two different goal functionals: the $L^2$
error in space-time
\begin{equation}
  \label{eq:JL2L2}
  J(u)= \frac{1}{\|e\|_{\Omega\times I}}\displaystyle\int_I(u,e)\drv t\,,
  \quad \mathrm{with} \;\; \|\cdot\|_{\Omega\times I}
  = \left(\int_I(\cdot,\cdot)\;\drv t\right)^{\frac{1}{2}},\; e\coloneq u-u_{\tau h}\,.
\end{equation}
and a point value control inside the interior layer at final time
\begin{equation}
  \label{eq:JTu}
  J(u)=u(\boldsymbol x_{c},\,T)\,,\quad\text{where}\quad\boldsymbol x_{c}=\left(\frac{1}{2},\,\frac{1}{2}\right)\,.
\end{equation}
We regularize the point functional~\eqref{eq:JTu} by a regularized Dirac delta function
$\delta_{s,\boldsymbol x_e}(r)=\alpha\operatorname{e}^{(1-1/(1-r^2/s^2))}$, where
$r=\lVert \boldsymbol x-\boldsymbol x_c\rVert$, $s>0$ is the cutoff radius and
$\alpha$ the scaling factor, such that $\delta_{s,\boldsymbol x_c}$ integrates to
$1$. 
In all configurations, we
refine and coarsen a fixed fraction of cells in both space and time. In space
and time, the refinement fraction is set to
\(\theta_{\text{space}}^{\text{ref}} = \theta_{\text{time}}^{\text{ref}} =
\frac{1}{8}\), and the coarsening fraction is
\(\theta_{\text{space}}^{\text{co}} = \theta_{\text{time}}^{\text{co}} =
\frac{1}{25}\). We use polynomial degrees \(1 \leq p, k \leq 9\) in space and
time, without imposing restrictions on their variation within a cell or across
neighboring cells. The initial space-time triangulation consists of a single
time cell and a once uniformly refined spatial coarse mesh (cf. Fig.~\ref{fig:interior-layer-geo}).
\begin{figure}[!htb]
  \centering
  \begin{minipage}{0.48\textwidth}\centering
    \begin{tikzpicture}[scale=4, very thick]
      \draw[gray] (0.25, 0) -- ++ (0,1);%
      \draw[gray] (0, 0.5) -- ++ (0.25,0);%
      \draw[gray] (0.75, 0) -- ++ (0,1);%
      \draw[gray] (0.75, 0.5) -- ++ (.25,0);%
      \coordinate (A) at (0.4166666,0.6666666);%
      \coordinate (B) at (0.5833333,0.3333333);%
      \draw[gray] (A) -- (B);%
      \draw[gray] (0.25, 0) -- (0.75, 1);%
      \draw[ultra thick,myred,-latex] (0.5, 0.6) --node[left,myred]{$\mathbf{b}$} ++
      (0.125,0.25);%
      \draw[ultra thick,myred,-latex] (0.585, 0.561) -- ++ (0.125,0.25);%
      \draw[gray] (A) -- (0.25, 0.5);%
      \draw[gray] (A) -- (0.5, 1);%
      \draw[gray] (B) -- (0.75, 0.5);%
      \draw[gray] (B) -- (0.5, 0);%
      \draw[ultra thick,mygreen] (0, 0) rectangle (1,1);%
    \end{tikzpicture}
  \end{minipage}
  \begin{minipage}{0.48\textwidth}\centering
    \includegraphics[width=0.66\linewidth]{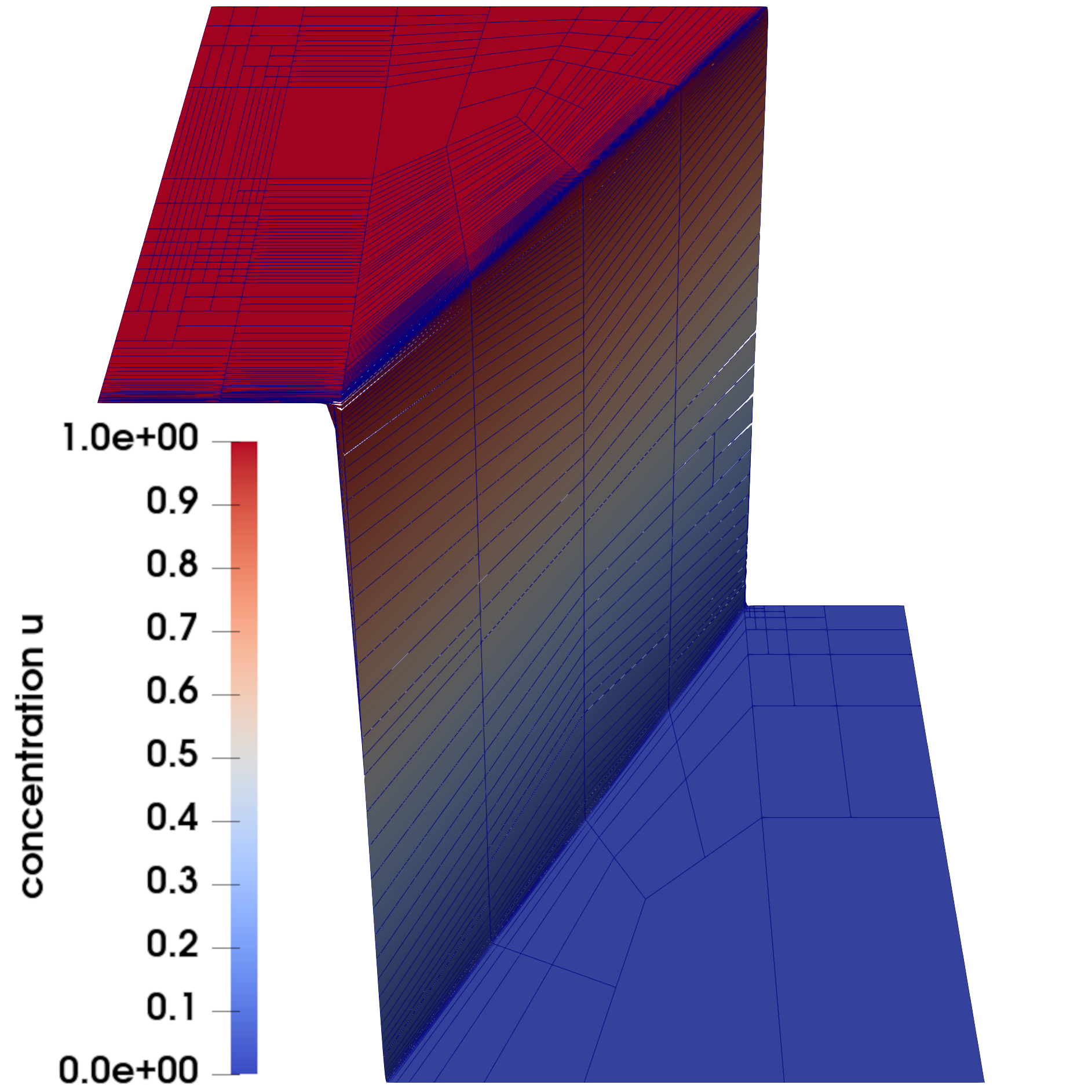}
  \end{minipage}
  \caption{Geometry and coarse, unstructured mesh of the domain (left) and the
    best adaptive solution obtained for the $L^2$ space-time error control. On
    the left, green coloring
    corresponds to inhomogeneous Dirichlet BCs.}\label{fig:interior-layer-geo}
\end{figure}

\begin{table}
\caption{\label{tab:step-layer-l2l2}Number of DoFs, error estimators, error,
  and effectivity index for the step-layer benchmark with $L^2$ space-time
  error control~\eqref{eq:JL2L2} and diffusion coefficient $\varepsilon = 10^{-6}$.}
\begin{center}\scriptsize
  \begin{tabular}{S[scientific-notation=false,round-precision=0,table-format=2]
    S[scientific-notation=false,round-precision=0,table-format=3]
    S[scientific-notation=false,round-precision=0,table-format=5]
    S[scientific-notation=false,round-precision=0,table-format=5]
    S[scientific-notation=false,round-precision=0,table-format=1]
    S[table-format=1.1e2]%
    S[table-format=1.1e2]%
    S[table-format=1.1e2]
    S[table-format=1.1e2]
    S[table-format=1.1e2]
    S[table-format=1.1e2]
    S[table-format=1.1e2]
    S[table-format=1.1e2]
    S[scientific-notation=false]}
    \toprule
    \mc{$\ell$} & \mc{$|\mathcal T_h|$} & \mc{$N_{\boldsymbol x}^{\mathrm{lo}}$} & \mc{$N_{\boldsymbol x}^{\mathrm{ho}}$} & \mc{$N_t$} & \mc{\textup{work}} & \mc{$\rho_{\max}$} & \mc{$\eta_{h,1}$} & \mc{$\eta_{h,2}$} & \mc{$\eta_h^{\textup{a}}$} & \mc{$\eta_\tau$} & \mc{$\eta_{\tau h}^{\textup{a}}$} & \mc{$e$} & \mc{$I_{\text{eff}}$} \\
\midrule
0  & 40 & 160 & 360 & 2 & 3520  & 4.2 &
-8.59e-03 & -4.96e-03 & -1.36e-02 & -1.59e-02 & -2.95e-02 & 1.72e-01 & 0.17 \\
1  & 40 & 182 & 392 & 3 & 7644  & 4.2 &
-3.39e-02 & -1.18e-02 & -4.57e-02 & -3.20e-03 & -4.89e-02 & 4.14e-02 & 1.18 \\
2  & 40 & 203 & 423 & 3 & 8526  & 4.2 &
-3.04e-02 & -9.99e-03 & -4.04e-02 & -2.91e-03 & -4.33e-02 & 4.14e-02 & 1.05 \\
3  & 40 & 223 & 453 & 3 &10035  & 4.2 &
-2.87e-02 & -9.47e-03 & -3.82e-02 & -2.91e-03 & -4.11e-02 & 4.14e-02 & 0.99 \\
4  & 40 & 243 & 483 & 3 &10935  & 4.2 &
-2.76e-02 & -9.11e-03 & -3.67e-02 & -2.91e-03 & -3.96e-02 & 4.14e-02 & 0.96 \\
5  & 40 & 263 & 513 & 3 &11835  & 4.2 &
-2.68e-02 & -8.88e-03 & -3.57e-02 & -2.91e-03 & -3.86e-02 & 4.14e-02 & 0.93 \\
6  & 40 & 283 & 543 & 3 &12735  & 4.2 &
-2.63e-02 & -8.71e-03 & -3.50e-02 & -2.91e-03 & -3.79e-02 & 4.14e-02 & 0.92 \\
7  & 40 & 303 & 573 & 3 &12726  & 4.2 &
-2.59e-02 & -8.58e-03 & -3.45e-02 & -2.91e-03 & -3.74e-02 & 4.14e-02 & 0.90 \\
8  & 40 & 323 & 603 & 3 &13566  & 4.2 &
-2.56e-02 & -8.47e-03 & -3.41e-02 & -2.91e-03 & -3.70e-02 & 4.14e-02 & 0.89 \\
11 & 70 & 898 &1545 & 4 &64656  &29.9 &
-1.29e-02 & -4.13e-03 & -1.70e-02 &  6.81e-04 & -1.63e-02 & 1.08e-02 & 1.51 \\
12 & 87 &1215 &2060 & 4 &87480  &59.8 &
-7.21e-03 & -2.48e-03 & -9.69e-03 &  7.87e-04 & -8.90e-03 & 7.84e-03 & 1.14 \\
23 &884 &22962&34441& 7 &4018350&6124.5&
 1.39e-05 & -1.41e-06 &  1.25e-05 & -1.79e-08 &  1.25e-05 & 2.44e-05 & 0.51 \\
\bottomrule
\end{tabular}
\end{center}
\end{table}

\paragraph{Space-time $L^2$ error control.} Tab.~\ref{tab:step-layer-l2l2}
demonstrates good control of the \(L^2\) error with effectivity indices
\(I_{\text{eff}}\) remaining close to 1 over a wide range of refinements,
indicating accurate and reliable a posteriori error estimation even with
anisotropic $hp$ refinements. Moreover, the maximum aspect ratio \(\rho_{\max}\)
increases significantly, reaching values above 6000, which underscores the
robustness of the method with respect to highly anisotropic elements. This
confirms the capability of the refinement strategy to efficiently resolve
anisotropic features, such as the interior layer, while maintaining
effectivity. Fig~\ref{fig:interior-layer-geo} shows a plot of the solution
with the highly refined mesh along the interior layer.

\begin{figure}[!htb]
  \centering
  \scalebox{0.95}{
    \begin{tikzpicture}
      \node[anchor=south west,inner sep=0, outer sep=0] (img) at (0,0)
      {\includegraphics[width=\textwidth]{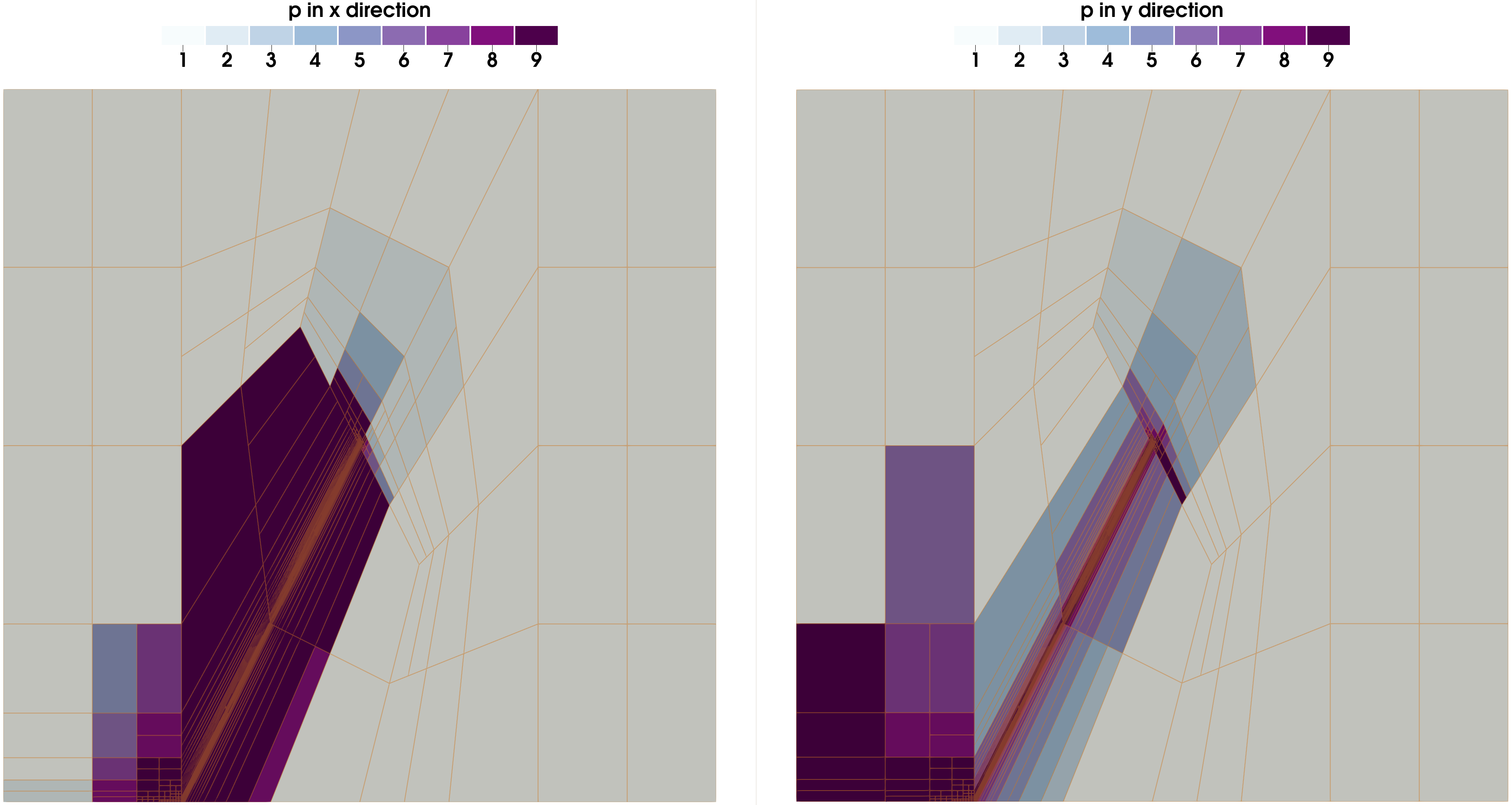}};%
      \node[draw=mygreen,anchor=south,line width=3pt,inner sep=0, outer
      sep=0,yshift=-1.7cm] at ($(img.south west) + 0.38*(img.south east)$)
      {\includegraphics[width=.275\textwidth]{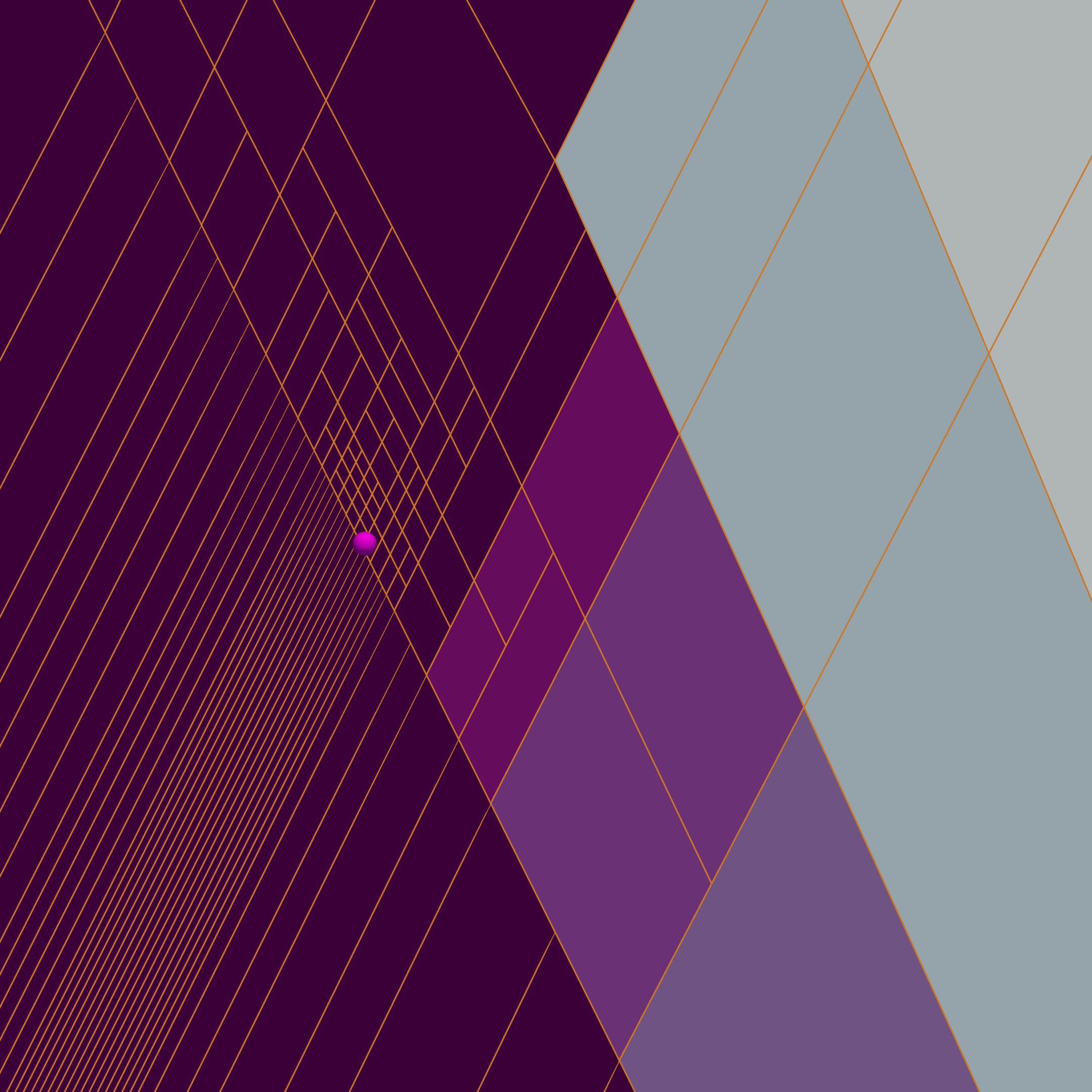}};%
      \node[draw=mygreen,anchor=south,line width=3pt,inner sep=0, outer
      sep=0,yshift=-1.7cm] at ($(img.south west) + 0.905*(img.south east)$)
      {\includegraphics[width=.275\textwidth]{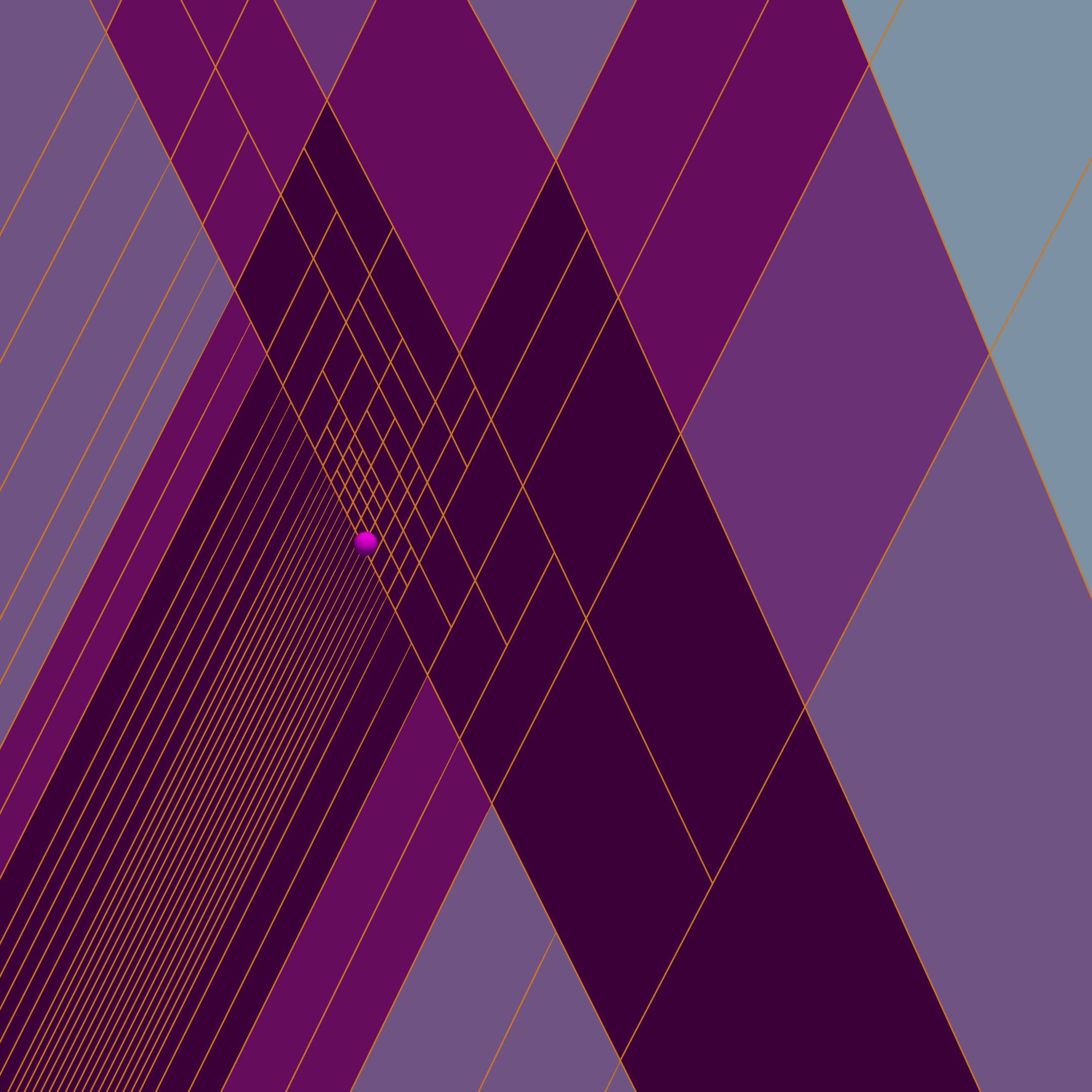}};%
      \coordinate (pos1) at
      ($(img.south west) + 0.765*(img.south east) + 0.45*(img.north west)$);
      \draw[mygreen,very thick] (pos1) +(-5pt,-5pt) rectangle +(5pt,5pt) ;
      \coordinate (pos2) at
      ($(img.south west) + 0.241*(img.south east) + 0.45*(img.north west)$);
      \draw[mygreen,very thick] (pos2) +(-5pt,-5pt) rectangle +(5pt,5pt) ;%
    \end{tikzpicture}}
\caption{\label{fig:point-mesh}Anisotropic $hp$-mesh after $18$ refinement steps (cf. Tables~\ref{tab:step-layer-point-value}) with diffusion coefficient $\varepsilon=10^{-6}$. The green squares
  are zoomed in to visualize the high anisotropic $h$ and $p$ refinement around
  the goal point. Within the mesh we observe high anisotropies in the mesh size
  and polynomial degree.}
\end{figure}
\paragraph{Point value control.}
To demonstrate that our anisotropic \( hp \)-refinement strategy generalizes
well across a range of convection-dominated regimes, we test two diffusion
coefficients \(\varepsilon = 10^{-6}\) and \(\varepsilon = 10^{-8}\).
\begin{table}
  \caption{\label{tab:step-layer-point-value}
    Number of DoFs, error estimators, error, and effectivity indices for the step-layer benchmark with point value control~\eqref{eq:JTu} and diffusion coefficient $\varepsilon = 10^{-6}$.}
\begin{center}\scriptsize
  \begin{tabular}{S[scientific-notation=false,round-precision=0,table-format=2]
    S[scientific-notation=false,round-precision=0,table-format=3]
    S[scientific-notation=false,round-precision=0,table-format=5]
    S[scientific-notation=false,round-precision=0,table-format=5]
    S[scientific-notation=false,round-precision=0,table-format=1]
    S[table-format=1.1e2]%
    S[table-format=1.1e2]%
    S[table-format=1.1e2]
    S[table-format=1.1e2]
    S[table-format=1.1e2]
    S[table-format=1.1e2]
    S[table-format=1.1e2]
    S[table-format=1.1e2]
    S[scientific-notation=false]}
\toprule
    \mc{$\ell$} & \mc{$|\mathcal T_h|$} & \mc{$N_{\boldsymbol x}^{\mathrm{lo}}$} & \mc{$N_{\boldsymbol x}^{\mathrm{ho}}$} & \mc{$N_t$} & \mc{\textup{work}} & \mc{$\rho_{\max}$} & \mc{$\eta_{h,1}$} & \mc{$\eta_{h,2}$} & \mc{$\eta_h^{\textup{a}}$} & \mc{$\eta_\tau$} & \mc{$\eta_{\tau h}^{\textup{a}}$} & \mc{$e$} & \mc{$I_{\text{eff}}$} \\
\midrule
0  & 40 & 160 & 360 & 2 & 2880   & 4.2 &
9.68e+0 & 2.61e+0 & 1.229e+1 & 3.44e-1 & 1.263e+1 & 1.10e-1 & 114.82 \\
1  & 40 & 185 & 396 & 2 & 3330   & 4.2 &
1.32e+1 & 2.11e-1 & 1.341e+1 & 1.20e+0 & 1.461e+1 & 1.15e-1 & 127.05 \\
2  & 42 & 231 & 467 & 2 & 4158   & 4.2 &
2.54e+1 & 3.10e-1 & 2.571e+1 & 2.45e+0 & 2.816e+1 & 1.18e-1 & 238.64 \\
3  & 44 & 273 & 532 & 2 & 5460   & 4.2 &
2.37e+1 & 4.32e-1 & 2.413e+1 & 4.02e+0 & 2.815e+1 & 1.18e-1 & 238.56 \\
4  & 49 & 362 & 669 & 2 & 7240   & 6.1 &
3.29e+0 & -7.19e-2 & 3.218e+0 & 1.76e+0 & 4.978e+0 & 1.18e-1 & 42.18 \\
5  & 56 & 474 & 842 & 3 & 18486  & 11.5 &
4.33e-1 & 1.06e-1 & 5.39e-1 & 6.56e-2 & 6.05e-1 & 1.41e-1 & 4.29 \\
6  & 56 & 539 & 922 & 3 & 21021  & 11.5 &
3.07e-1 & 1.08e-2 & 3.178e-1 & -4.21e-2 & 2.757e-1 & 2.06e-1 & 1.34 \\
7  & 56 & 611 & 1009 & 3 & 23829  & 11.5 &
3.57e-1 & 2.47e-2 & 3.817e-1 & -2.84e-2 & 3.533e-1 & 2.91e-1 & 1.21 \\
14 & 181 & 10373 & 13100 & 6 & 1211640 & 239.0 &
-2.79e-7 & -1.14e-7 & -3.93e-7 & 4.79e-5 & 4.75e-5 & -2.95e-5 & 1.61 \\
15 & 249 & 16862 & 20974 & 7 & 2541924 & 467.6 &
9.00e-8 & -1.96e-7 & -1.06e-7 & -2.55e-6 & -2.66e-6 & 5.56e-7 & 4.78 \\
16    & 250   & 16980     & 21115     & 9      & 4278960  & 382.8  & 6.7121e-08  & -1.4605e-07 & -7.8931e-08 & -1.8690e-08 & -9.7621e-08 & 4.2012e-07  & 0.2324\\ 
17    & 350   & 26953     & 33184     & 9      & 6549579  & 383.4  & 3.3761e-08  & -3.8931e-08 & -5.1699e-09 & -1.3831e-08 & -1.9001e-08 & -2.8858e-09 & 6.5843\\ 
18    & 490   & 40943     & 50113     & 9      & 9949149  & 382.8  & 4.8871e-08  & -2.3998e-08 & -2.4873e-08 & -2.2569e-08 &  2.3040e-09 & -1.6561e-09 & 1.3912\\ 
\bottomrule
\end{tabular}
\end{center}
\end{table}
\begin{table}
  \caption{\label{tab:step-layer-point-value-1e-8}
    Number of DoFs, error estimators, error, and effectivity indices for the step-layer benchmark with point value control~\eqref{eq:JTu} and diffusion coefficient $\varepsilon = 10^{-8}$.}
\begin{center}\scriptsize
  \begin{tabular}{S[scientific-notation=false,round-precision=0,table-format=2]
    S[scientific-notation=false,round-precision=0,table-format=3]
    S[scientific-notation=false,round-precision=0,table-format=5]
    S[scientific-notation=false,round-precision=0,table-format=5]
    S[scientific-notation=false,round-precision=0,table-format=1]
    S[table-format=1.1e2]%
    S[table-format=1.1e2]%
    S[table-format=1.1e2]
    S[table-format=1.1e2]
    S[table-format=1.1e2]
    S[table-format=1.1e2]
    S[table-format=1.1e2]
    S[table-format=1.1e2]
    S[scientific-notation=false]}
    \toprule
    \mc{$\ell$} & \mc{$|\mathcal T_h|$} & \mc{$N_{\boldsymbol x}^{\mathrm{lo}}$} & \mc{$N_{\boldsymbol x}^{\mathrm{ho}}$} & \mc{$N_t$} & \mc{\textup{work}} & \mc{$\rho_{\max}$} & \mc{$\eta_{h,1}$} & \mc{$\eta_{h,2}$} & \mc{$\eta_h^{\textup{a}}$} & \mc{$\eta_\tau$} & \mc{$\eta_{\tau h}^{\textup{a}}$} & \mc{$e$} & \mc{$I_{\text{eff}}$} \\
    \midrule
0 & 40 & 160 & 360 & 2 & 2880 & 4.2 &
9.69e+0 & 2.61e+0 & 1.23e+1 & 3.43e-1 & 1.26e+1 & -7.63e-2 & 165 \\
1 & 40 & 183 & 393 & 2 & 3294 & 4.2 &
1.32e+1 & 2.02e-1 & 1.34e+1 & 1.15e+0 & 1.46e+1 & -7.16e-2 & 204 \\
2 & 42 & 223 & 456 & 2 & 4014 & 4.2 &
2.81e+1 & 6.59e-1 & 2.88e+1 & 2.65e+0 & 3.15e+1 & -6.84e-2 & 461 \\
3 & 46 & 274 & 542 & 2 & 5480 & 4.2 &
2.58e+1 & 4.04e-1 & 2.62e+1 & 2.78e+0 & 2.90e+1 & -6.84e-2 & 424 \\
4 & 52 & 339 & 654 & 2 & 6780 & 6.1 &
1.64e+1 & 1.53e-1 & 1.66e+1 & 1.81e+0 & 1.84e+1 & -6.84e-2 & 269 \\
5 & 60 & 429 & 804 & 2 & 8580 & 11.5 &
1.92e+0 & 1.68e-2 & 1.94e+0 & 8.42e-1 & 2.78e+0 & -9.28e-2 & 30 \\
6 & 60 & 489 & 879 & 3 & 19071 & 11.5 &
3.10e-1 & -7.29e-3 & 3.03e-1 & 6.66e-2 & 3.69e-1 & 2.08e-2 & 17.8 \\
7 & 60 & 557 & 962 & 3 & 21723 & 11.5 &
1.86e-1 & -1.26e-2 & 1.73e-1 & 1.03e-2 & 1.83e-1 & -7.65e-2 & 2.39 \\
8 & 60 & 634 & 1054 & 3 & 24726 & 11.5 &
7.47e-2 & 1.21e-2 & 8.68e-2 & 1.86e-2 & 1.05e-1 & -1.37e-2 & 7.67 \\
21 & 1245 & 112689 & 137154 & 7 & 17354106 & 1790.6 &
4.09e-9 & -3.89e-7 & -3.85e-7 & -2.67e-6 & -3.06e-6 & 5.05e-6 & 0.61 \\
22 & 1633 & 150792 & 183289 & 8 & 28952064 & 1901.7 &
6.15e-10 & -1.23e-6 & -1.23e-6 & 2.51e-7 & -9.80e-7 & 3.21e-7 & 3.05 \\
23 & 2153 & 201159 & 244348 & 8 & 38622528 & 1902.6 &
7.63e-10 & -4.44e-8 & -4.36e-8 & 2.52e-8 & -1.84e-8 & 1.63e-8 & 1.13 \\
\bottomrule
\end{tabular}
\end{center}
\end{table}
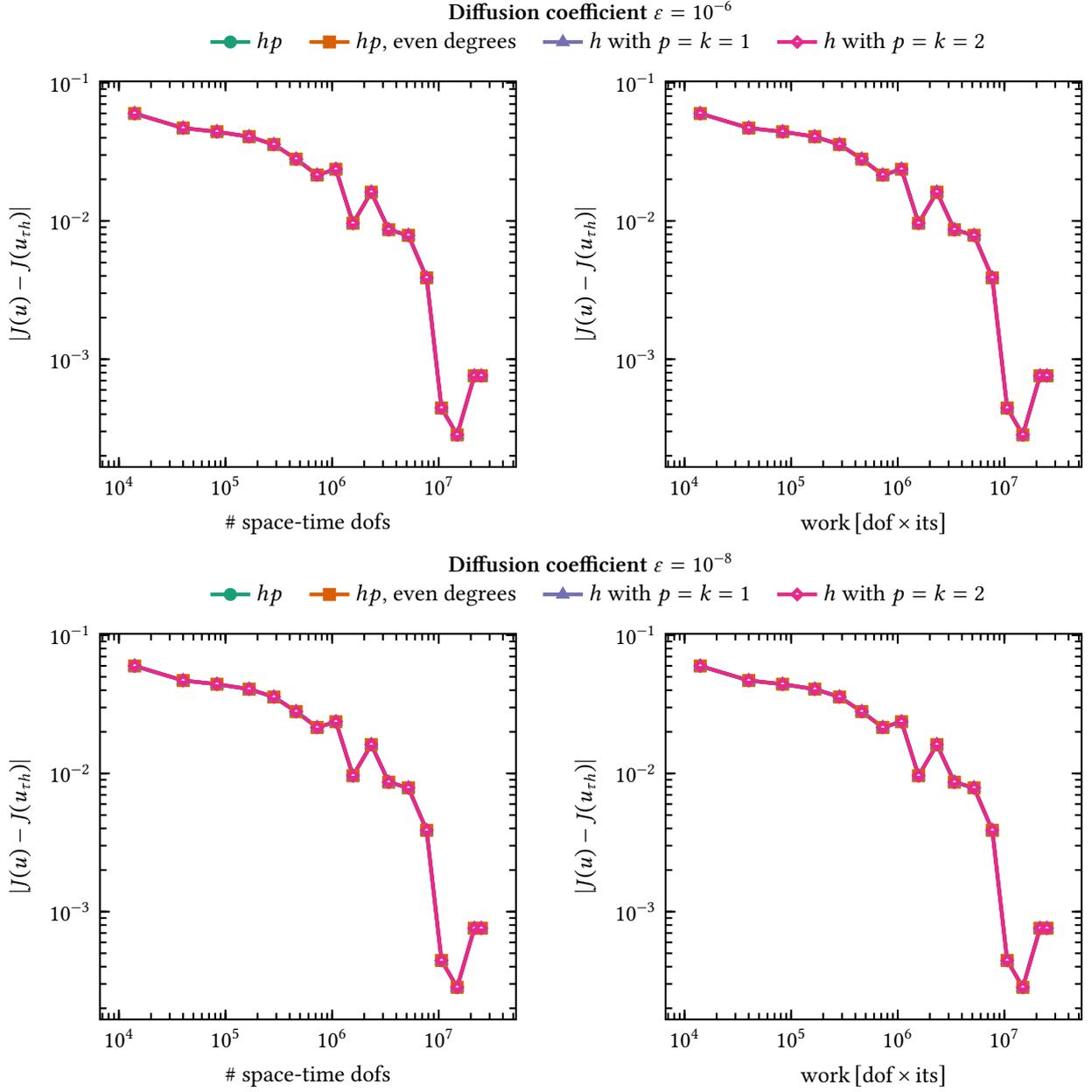
\begin{figure}[htbp]
  \begin{tikzpicture}[font=\small]
    \pgfplotsset{log origin=infty} \pgfplotsset{ colormap/Dark2-4, cycle
      multiindex* list={ mark list*\nextlist Dark2-4\nextlist }, every
      axis/.append style = {thick},}%
    \begin{groupplot}[group style={columns=2,group name=contest,horizontal
        sep=2.25cm},%
      legend columns=4,%
      legend style={/tikz/every even column/.append style={column
          sep=0.3cm},at={(1.2,1.05)},draw=none,anchor=south},%
      width=.475\textwidth,height=.45\textwidth]
      \nextgroupplot[xmode=log,ymode=log,ymode=log,log
      origin=infty,ylabel=$|J(u)-J(u_{\tau h})|$, xlabel={\# space-time dofs}]%
      \addplot +[solid,raw gnuplot,ultra thick] gnuplot {%
        plot 'convergence_default.txt' using ($1*$2):4 every :::0::0; };%
      \addlegendentry{$hp$} \addplot +[solid,raw gnuplot,ultra thick] gnuplot {%
        plot 'convergence_default.txt' using ($1*$2):4 every :::3::3; };%
      \addlegendentry{$hp$, even degrees} \addplot +[solid,raw gnuplot,ultra
      thick] gnuplot {%
        plot 'convergence_default.txt' using ($1*$2):4 every :::1::1; };%
      \addlegendentry{$h$ with $p=k=1$} \addplot +[solid,raw gnuplot,ultra
      thick] gnuplot {%
        plot 'convergence_default.txt' using ($1*$2):4 every :::2::2; };%
      \addlegendentry{$h$ with $p=k=2$}%
      \nextgroupplot[xmode=log,ymode=log,log
      origin=infty,ylabel=$|J(u)-J(u_{\tau h})|$,
      xlabel={$\text{work}\,[\text{dof}\times\text{its}]$}]%
      \addplot +[solid,raw gnuplot,ultra thick] gnuplot {%
        plot 'convergence_default.txt' using 3:4 every :::0::0; };%
      \addplot +[solid,raw gnuplot,ultra thick] gnuplot {%
        plot 'convergence_default.txt' using 3:4 every :::3::3; };%
      \addplot +[solid,raw gnuplot,ultra thick] gnuplot {%
        plot 'convergence_default.txt' using 3:4 every :::1::1; };%
      \addplot +[solid,raw gnuplot,ultra thick] gnuplot {%
        plot 'convergence_default.txt' using 3:4 every :::2::2; };%
    \end{groupplot}
    \node (title) at ($(contest c1r1.center)!0.5!(contest c2r1.center)+(0,4cm)$)
    {\bfseries Diffusion coefficient $\varepsilon = 10^{-6}$};
  \end{tikzpicture}

  \begin{tikzpicture}[font=\small]
    \pgfplotsset{log origin=infty} \pgfplotsset{ colormap/Dark2-4, cycle
      multiindex* list={ mark list*\nextlist Dark2-4\nextlist }, every
      axis/.append style = {thick},}%
    \begin{groupplot}[group style={columns=2,group name=contest,horizontal
        sep=2.25cm},%
      legend columns=4,%
      legend style={/tikz/every even column/.append style={column
          sep=0.3cm},at={(1.2,1.05)},draw=none,anchor=south},%
      width=.475\textwidth,height=.45\textwidth]
      \nextgroupplot[xmode=log,ymode=log,log
      origin=infty,ylabel=$|J(u)-J(u_{\tau h})|$, xlabel={\# space-time dofs}]%
      \addplot +[solid,raw gnuplot,ultra thick] gnuplot {%
        plot 'convergence_default1e-8.txt' using ($1*$2):4 every :::0::0; };%
      \addlegendentry{$hp$} \addplot +[solid,raw gnuplot,ultra thick] gnuplot {%
        plot 'convergence_default1e-8.txt' using ($1*$2):4 every :::3::3; };%
      \addlegendentry{$hp$, even degrees} \addplot +[solid,raw gnuplot,ultra
      thick] gnuplot {%
        plot 'convergence_default1e-8.txt' using ($1*$2):4 every :::1::1; };%
      \addlegendentry{$h$ with $p=k=1$} \addplot +[solid,raw gnuplot,ultra
      thick] gnuplot {%
        plot 'convergence_default1e-8.txt' using ($1*$2):4 every :::2::2; };%
      \addlegendentry{$h$ with $p=k=2$}%
      \nextgroupplot[xmode=log,ymode=log,log
      origin=infty,ylabel=$|J(u)-J(u_{\tau h})|$,
      xlabel={$\text{work}\,[\text{dof}\times\text{its}]$}]%
      \addplot +[solid,raw gnuplot,ultra thick] gnuplot {%
        plot 'convergence_default1e-8.txt' using 3:4 every :::0::0; };%
      \addplot +[solid,raw gnuplot,ultra thick] gnuplot {%
        plot 'convergence_default1e-8.txt' using 3:4 every :::3::3; };%
      \addplot +[solid,raw gnuplot,ultra thick] gnuplot {%
        plot 'convergence_default1e-8.txt' using 3:4 every :::1::1; };%
      \addplot +[solid,raw gnuplot,ultra thick] gnuplot {%
        plot 'convergence_default1e-8.txt' using 3:4 every :::2::2; };%
    \end{groupplot}
    \node (title) at ($(contest c1r1.center)!0.5!(contest c2r1.center)+(0,4cm)$)
    {\bfseries Diffusion coefficient $\varepsilon = 10^{-8}$};
  \end{tikzpicture}
  \caption{\label{fig:accuracy-work}Comparison between adaptive anisotropic
    refinement strategies: anisotropic $hp$ refinement considering all
    polynomials \( p, k \in \{1,\dots,9\} \) and only even polynomials \( p, k \in \{2,4,6,8,10\} \). To validate the efficiency of the anisotropic $hp$
    refinement we consider $h$ anisotropic refinement with isotropic linear and
    quadratic polynomial degree (uniform in space and time).}
\end{figure}

In Fig.~\ref{fig:point-mesh}, we show the anisotropic \( hp \)-mesh obtained in the final DWR loop for $\varepsilon=10^{-6}$. Strong local anisotropy in both mesh size and polynomial degree is
observed close to the goal point. Highly refined elements are aligned with the
interior layer. Far from $\boldsymbol x_{c}$, we observe no refinement in $h$
and $p$. Note that we observe cells with linear polynomials in one and $9$th
order polynomials in the other direction.

Tables~\ref{tab:step-layer-point-value}
and~\ref{tab:step-layer-point-value-1e-8} show that the pointwise error
converges very fast under anisotropic \( hp \)-refinement. While the initial
effectivity indices \( I_{\text{eff}} \) are high due to overestimation in the
early refinement stages, they stabilize toward moderate values as the mesh
better resolves the behavior close to the evaluation point. The maximum aspect ratio
\(\rho_{\max}\) increases substantially, highlighting the method's robustness in
resolving localized features with strongly anisotropic elements. These results
confirm that the employed refinement strategy is highly effective for accurate
pointwise error control. Fig.~\ref{fig:accuracy-work} underscores that the
anisotropic \( hp \)-refinement strategy, employing either the full range
\( p, k \in \{1,\dots,9\} \) or only even degrees \( p, k \in \{2,4,6,8,10\} \),
achieves exponential convergence of the error $|J(u)-J(u_{\tau h})|$ with
respect to the number of degrees of freedom. Thus, this strategy clearly
outperforms anisotropic \( h \)-refinement with fixed polynomial degree
(\( p = k = 1 \) or \( p = k = 2 \)). Although a pre-asymptotic regime is
visible in the convergence behavior, the results clearly demonstrate the
superior approximation efficiency of the anisotropic \( hp \)-refinement
strategy. Notably, the smaller the diffusion coefficient
(\(\varepsilon = 10^{-6}\) and \(\varepsilon = 10^{-8}\)), the more pronounced
the performance gains of anisotropic \( hp \)-refinement, which efficiently
resolves the increasingly localized features of the solution. We also observe
exponential convergence of the error \(|J(u) - J(u_{\tau h})|\) with respect to
the computational work \(w\), defined by
\begin{equation}
  \label{eq:work}
  w = \sum_{i=1}^{|\mathcal{T}_{\tau}|} (k_i + 1) \cdot N^{\symbf{x}} \cdot N^{it}_i,
\end{equation}
where \(N^{it}_i\) is the number of GMRES iterations required to solve the local
subproblem on the time interval \(I_i\). This metric accounts for the full
computational cost of solving the global space-time problem. Notably, the same
exponential convergence behavior observed with respect to the number of degrees
of freedom is retained when measured against the work. This highlights the
efficiency and scalability of the overall approach and confirms the robustness
of the linear solver introduced in Sec.~\ref{sec:solvers}. The restriction to
even polynomial degrees can lead to a more economical discretization by
accelerating error reduction in the goal functional and allowing the method to
leave the pre-asymptotic regime earlier.

Overall, we observe exponential convergence of the point value error with
respect to both degrees of freedom and work. Despite extreme anisotropies
(aspect ratios > 1900) and high polynomial degrees, the solver remains robust and
efficient. While the spatial and temporal estimator contributions initially
differ by orders of magnitude, refinement balances them. This shows the
effectiveness of goal-oriented refinement in concentrating resolution where
needed for accurate point value control.

\FloatBarrier
\subsection{Instationary Hemker Example}
\label{sec:hemker}
We investigate the classical Hemker problem~\cite{hemker_singularly_1996} in a time-dependent setting.
We consider the convection-diffusion-reaction system given by
equation~\eqref{eq:CDR_system}, and we set \(\alpha = 0\) and \(f = 0\),
$\symbf{b}=(1,\,0)^{\top}$. The diffusion coefficient is
\(\varepsilon = 10^{-4}\) or \(\varepsilon = 10^{-6}\). This guarantees that the
heat is entirely advected through the domain without any inherent
degradation. Consequently, any observed decrease in
heat arises solely from numerical artifacts, aside from a negligible
reduction induced by the diffusion term. The computational space-time domain is
defined as \(\Omega \times I\) with $I=(0,\,10]$ and
\( \Omega = \left( (-3,\,8) \times (-3,\,3) \right) \setminus \{ (x,\,y)
\suchthat x^2 + y^2 \leq 1 \}\), illustrated with the coarse mesh in
Fig.~\ref{fig:hemker-geo}. The boundary conditions are specified as follows:
The boundaries colored green in Fig.~\ref{fig:hemker-geo} correspond to
Dirichlet boundary conditions. Specifically, the boundary at \(x = -3\) is
subjected to homogeneous Dirichlet conditions $u=0$. The boundaries colored blue
indicate homogeneous Neumann boundary conditions $\partial_{\symbf{n}} u=0$. On
the circular boundary, inhomogeneous Dirichlet boundary conditions are imposed
with \(u = 1\).
The aim is to control
the error within a control point in the upper interior layer $\boldsymbol
x_{c}=(4,\,1)$. To this end, we use the
goal functional
\begin{equation}\label{eq:point-goal}
  J(u)=u(\boldsymbol x_{c},\,T).
\end{equation}
We regularize $J$ by a regularized Dirac delta function
$\delta_{s,\boldsymbol x_e}(r)$. We refine and coarsen a fixed fraction of cells in both space and time. In
space and time, the refinement fraction is set to
\(\theta_{\text{space}}^{\text{ref}} = \frac{1}{7}\) and
$\theta_{\text{time}}^{\text{ref}} = \frac{1}{8}$. The coarsening fraction is
set to
\(\theta_{\text{space}}^{\text{co}} = \theta_{\text{time}}^{\text{co}} =
\frac{1}{25}\). We use polynomial degrees \(1 \leq p \leq 9\) and
\(0 \leq k \leq 9\) in space and time, respectively. We do not put any
restrictions on their variation within a cell or across neighboring cells. The
initial space-time triangulation consists of the three times uniformly refined
time domain $I$ and a twice uniformly refined spatial coarse mesh (cf.
Fig.~\ref{fig:hemker-geo}).

\begin{figure}[htb]
  \centering
  \begin{tikzpicture}[scale=0.65,very thick]
    \draw[gray] (3, -3) -- ++ (0,6);
    \draw[gray] (-3, 3) -- ++ (6,-6);
    \draw[gray] (-3, -3) -- ++ (6,6);
    \draw[gray] (-3, 0) -- ++ (11,0);
    \draw[gray] (0, -3) -- ++ (0,6);
    \draw[ultra thick,myblue] (-3, -3) rectangle (8,3);
    \draw[ultra thick,myred,-latex] (2.5, 1.5) --node[above,myred]{$\symbf{b}$} ++ (2,0);
    \draw[ultra thick,myred,-latex] (2.5, -1.5) -- ++ (2,0);
    \draw[ultra thick,mygreen,fill=white] (0,0) circle (1);%
    \draw[ultra thick,mygreen] (-3,-3) -- node[left]{\color{mygreen}$\Gamma_D$}++(0,6);%
  \end{tikzpicture}\hfill
  \includegraphics[width=0.45\textwidth]{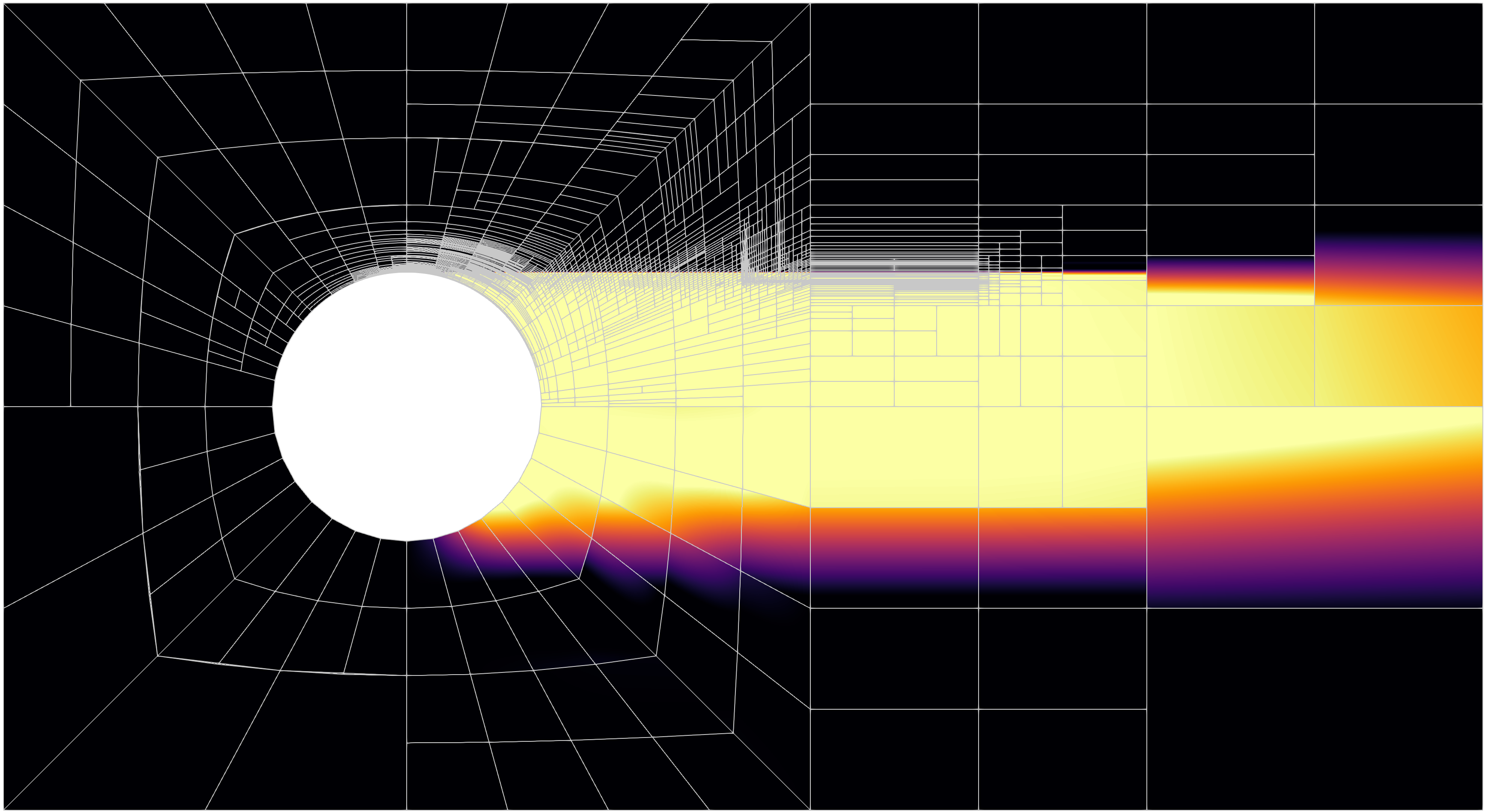}
  \caption{On the left: Geometry and coarse mesh of the domain for the Hemker
    problem. Green coloring corresponds to Dirichlet BCs, blue indicates
    homogeneous Neumann BCs. On the circle, inhomogeneous Dirichlet BCs are
    prescribed. The left boundary is associated with homogeneous Dirichlet BCs.
    On the right: The solution with mesh at the final time point after the 20th
    DWR loop with $\varepsilon=10^{-6}$. }
  \label{fig:hemker-geo}
\end{figure}
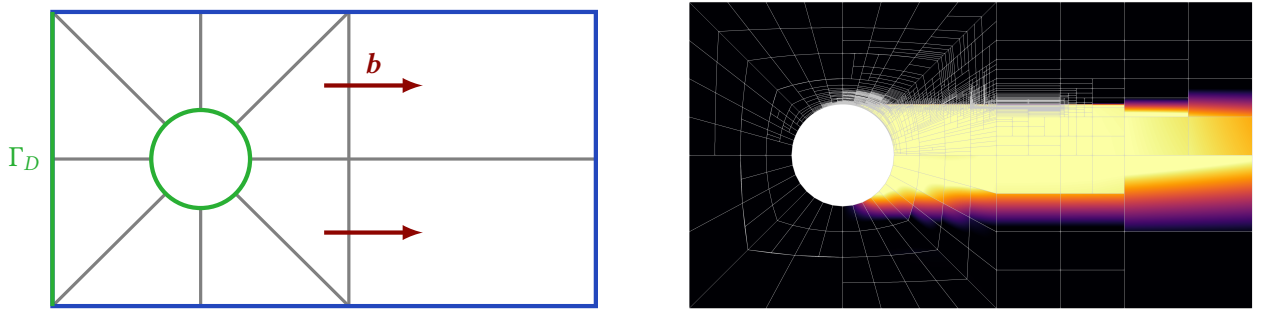

\begin{table}[htb]
  \caption{\label{tab:hemker-pv-1e-4}Number of DoFs, error estimators, error, and effectivity index for the Hemker benchmark with point value control~\eqref{eq:point-goal} and diffusion coefficient $\varepsilon = 10^{-4}$.}
\begin{center}\scriptsize
\begin{tabular}{S[scientific-notation=false,round-precision=0,table-format=2]
    S[scientific-notation=false,round-precision=0,table-format=3]
    S[scientific-notation=false,round-precision=0,table-format=5]
    S[scientific-notation=false,round-precision=0,table-format=5]
    S[scientific-notation=false,round-precision=0,table-format=1]
    S[table-format=1.1e2]%
    S[table-format=1.1e2]%
    S[table-format=1.1e2]
    S[table-format=1.1e2]
    S[table-format=1.1e2]
    S[table-format=1.1e2]
    S[table-format=1.1e2]}
\toprule
    \mc{$\ell$} & \mc{$|\mathcal T_h|$} & \mc{$N_{\boldsymbol x}^{\mathrm{lo}}$} & \mc{$N_{\boldsymbol x}^{\mathrm{ho}}$} & \mc{$N_t$} & \mc{\textup{work}} & \mc{$\rho_{\max}$} & \mc{$\eta_{h,1}$} & \mc{$\eta_{h,2}$} & \mc{$\eta_h^{\textup{a}}$} & \mc{$\eta_\tau$} & \mc{$\eta_{\tau h}^{\textup{a}}$} \\
\midrule
 0     & 160   & 640       & 1440      & 8      & 10240  & 4.1    & -6.575e-04 &  2.435e-04 & -4.140e-04  & 7.057e-05 & -3.434e-04  \\
 1     & 165   & 763       & 1632      & 8      & 12208  & 4.6    & -5.407e-03 & -5.343e-03 & -1.075e-02  & -2.025e-02 & -3.100e-02  \\
 2     & 173   & 912       & 1861      & 8      & 14592  & 4.6    & -8.211e-04 & -1.713e-03 & -2.534e-03  & -1.669e-02 & -1.922e-02  \\
 3     & 177   & 1077      & 2094      & 8      & 17232  & 6.7    & -2.752e-04 & -6.208e-03 & -6.483e-03  & -1.815e-02 & -2.463e-02  \\
 4     & 185   & 1275      & 2386      & 8      & 20400  & 6.7    & 2.184e-03  & 2.554e-04 & 2.439e-03   & -2.195e-02 & -1.951e-02  \\
 5     & 193   & 1478      & 2687      & 9      & 44340  & 6.7    & 2.113e-03  & 2.757e-03 & 4.870e-03   & -1.735e-02 & -1.248e-02  \\
 6     & 197   & 1670      & 2956      & 9      & 50100  & 6.7    & 2.581e-03  & 7.982e-04 & 3.379e-03   & -1.713e-02 & -1.375e-02  \\
 7     & 211   & 2071      & 3534      & 9      & 62130  & 6.9    & 2.380e-03  & -2.117e-04 & 2.168e-03  & -1.617e-02 & -1.400e-02  \\
 8     & 226   & 2484      & 4139      & 9      & 74520  & 13.3   & 1.713e-03  & -2.835e-03 & -1.122e-03  & -1.548e-02 & -1.660e-02  \\
 9     & 255   & 3568      & 5647      & 9      & 107040 & 13.3   & -2.527e-05 &  3.726e-04 & 3.473e-04   & -1.597e-02 & -1.563e-02  \\
 10    & 319   & 6251      & 9304      & 10     & 331303 & 13.3   & 1.831e-05  & -7.292e-05 & -5.461e-05  & -1.477e-02 & -1.482e-02  \\
 14    & 742   & 30148     & 40206     & 12     & 2954504& 106.7  & 2.114e-08  & -3.929e-03 & -3.928e-03  & -9.296e-04 & -4.857e-03 \\
 15    & 908   & 42201     & 55302     & 12     & 4135698& 106.7  & 8.676e-07  & -1.365e-03 & -1.364e-03  & -2.674e-03 & -4.038e-03  \\
 16    & 1163  & 62682     & 80664     & 12     & 6456246& 106.7  & -2.260e-07 & -5.747e-05 & -5.770e-05 & -6.480e-04 & -7.057e-04 \\
    \bottomrule
\end{tabular}
\end{center}
\end{table}
\begin{table}[htb]
  \caption{\label{tab:hemker-pv-1e-6}Number of DoFs, error estimators, error, and effectivity index for the Hemker benchmark with point value control~\eqref{eq:point-goal} and diffusion coefficient $\varepsilon = 10^{-6}$.}
\begin{center}\scriptsize
  \begin{tabular}{S[scientific-notation=false,round-precision=0,table-format=2]
    S[scientific-notation=false,round-precision=0,table-format=3]
    S[scientific-notation=false,round-precision=0,table-format=5]
    S[scientific-notation=false,round-precision=0,table-format=5]
    S[scientific-notation=false,round-precision=0,table-format=1]
    S[table-format=1.1e2]%
    S[table-format=1.1e2]%
    S[table-format=1.1e2]
    S[table-format=1.1e2]
    S[table-format=1.1e2]
    S[table-format=1.1e2]
    S[table-format=1.1e2]}
\toprule
    \mc{$\ell$} & \mc{$|\mathcal T_h|$} & \mc{$N_{\boldsymbol x}^{\mathrm{lo}}$} & \mc{$N_{\boldsymbol x}^{\mathrm{ho}}$} & \mc{$N_t$} & \mc{\textup{work}} & \mc{$\rho_{\max}$} & \mc{$\eta_{h,1}$} & \mc{$\eta_{h,2}$} & \mc{$\eta_h^{\textup{a}}$} & \mc{$\eta_\tau$} & \mc{$\eta_{\tau h}^{\textup{a}}$} \\
\midrule
0  & 160 & 640 & 1440 & 8 & 10240 & 4.1 &
-6.670e-04 &  2.462e-04 & -4.208e-04 &  7.155e-05 & -3.493e-04 \\
1  & 168 & 775 & 1659 & 8 & 12400 & 4.6 &
-5.412e-03 & -4.145e-03 & -9.557e-03 & -1.692e-02 & -2.648e-02 \\
2  & 175 & 928 & 1891 & 8 & 14848 & 4.6 &
-7.248e-03 & -3.715e-03 & -1.096e-02 & -1.700e-02 & -2.796e-02 \\
3  & 187 & 1171 & 2256 & 8 & 18736 & 6.7 &
-3.146e-03 &  5.023e-04 & -2.644e-03 & -1.924e-02 & -2.188e-02 \\
4  & 203 & 1433 & 2660 & 8 & 22928 & 12.1 &
-2.681e-03 & -4.278e-04 & -3.109e-03 & -2.019e-02 & -2.330e-02 \\
5  & 210 & 1688 & 3019 & 9 & 50640 & 13.3 &
 3.924e-04 &  6.264e-04 &  1.019e-03 & -1.380e-02 & -1.278e-02 \\
6  & 224 & 2137 & 3651 & 9 & 64110 & 13.3 &
 4.014e-04 &  7.267e-04 &  1.128e-03 & -1.382e-02 & -1.269e-02 \\
7  & 239 & 2601 & 4309 & 9 & 78030 & 26.7 &
 5.916e-04 &  1.493e-04 &  7.409e-04 & -1.359e-02 & -1.285e-02 \\
8  & 253 & 3150 & 5064 & 9 & 94500 & 26.7 &
 5.781e-04 &  4.614e-04 &  1.040e-03 & -1.349e-02 & -1.245e-02 \\
 9  & 290 & 4748 & 7206 & 9 & 142440 & 26.7 &
  1.718e-05 & -2.234e-04 & -2.062e-04 & -1.499e-02 & -1.520e-02 \\
 10 & 349 & 7241 & 10601 & 10 & 383773 & 26.7 &
  4.769e-08 &  1.511e-05 &  1.516e-05 & -1.345e-02 & -1.344e-02 \\
 14 & 843 & 41144 & 53396 & 12 & 3949824 & 53.3 &
 -3.284e-05 &  1.488e-04 &  1.159e-04 & -1.521e-03 & -1.405e-03 \\
15 & 1057 & 58657 & 75020 & 12 & 5865700 & 106.7 &
 -4.473e-05 &  1.343e-04 &  8.957e-05 & -1.539e-03 & -1.449e-03 \\
16 & 1317 & 81577 & 103094 & 12 & 8647162 & 213.3 &
 1.391e-06 & -2.807e-05 & -2.668e-05 & -1.549e-03 & -1.576e-03 \\
\bottomrule
\end{tabular}
\end{center}
\end{table}

\begin{figure}
  \scalebox{0.97}{
    \begin{tikzpicture}
      \node[anchor=south west,inner sep=0, outer sep=0] (img) at (0,0)
      {\includegraphics[width=.5\textwidth]{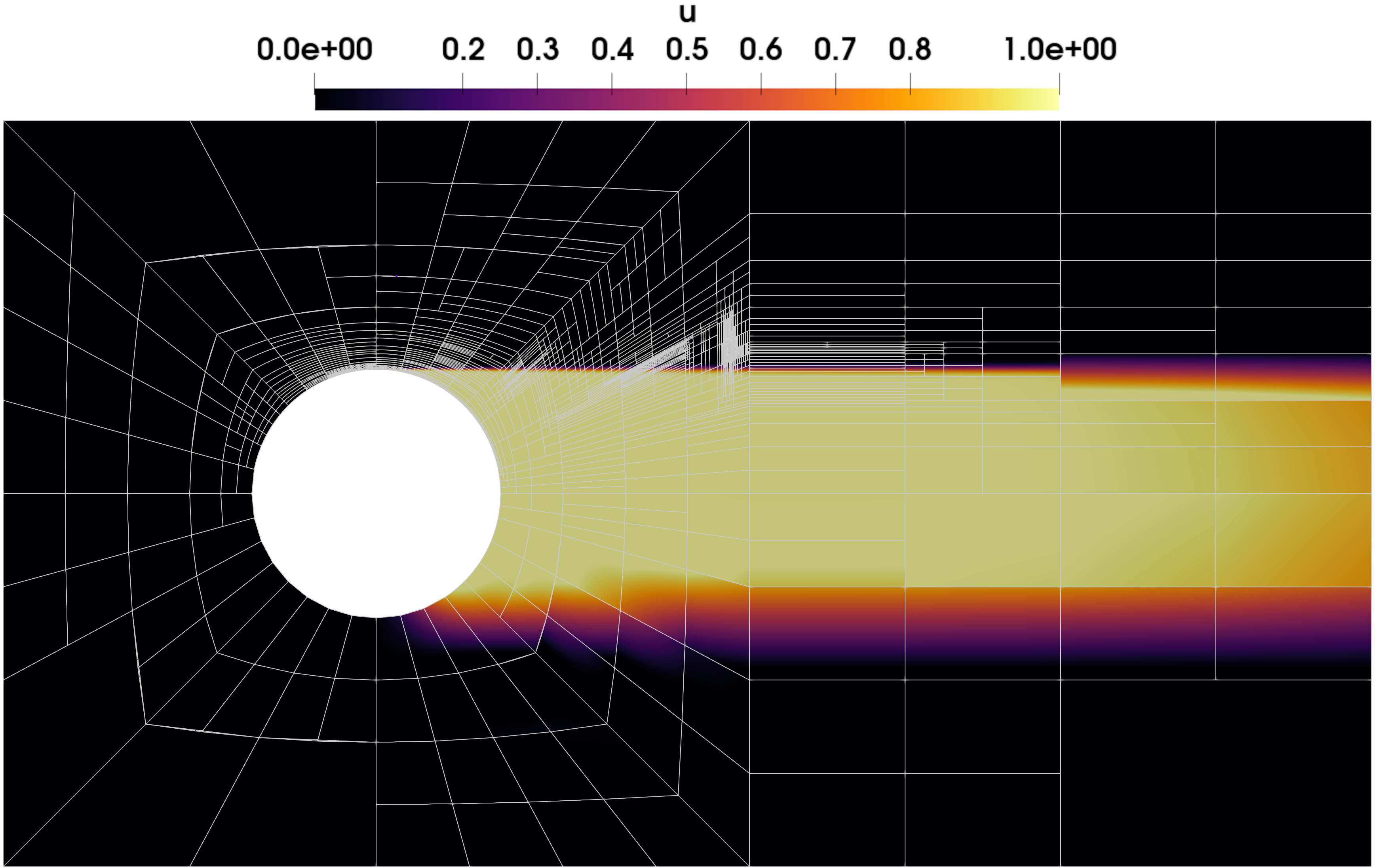}};%
      \node[shift={(0.25,0.0)},draw=mygreen,anchor=south west,line width=3pt,inner sep=0, outer sep=0] at (img.east)
      {\includegraphics[width=.5\textwidth]{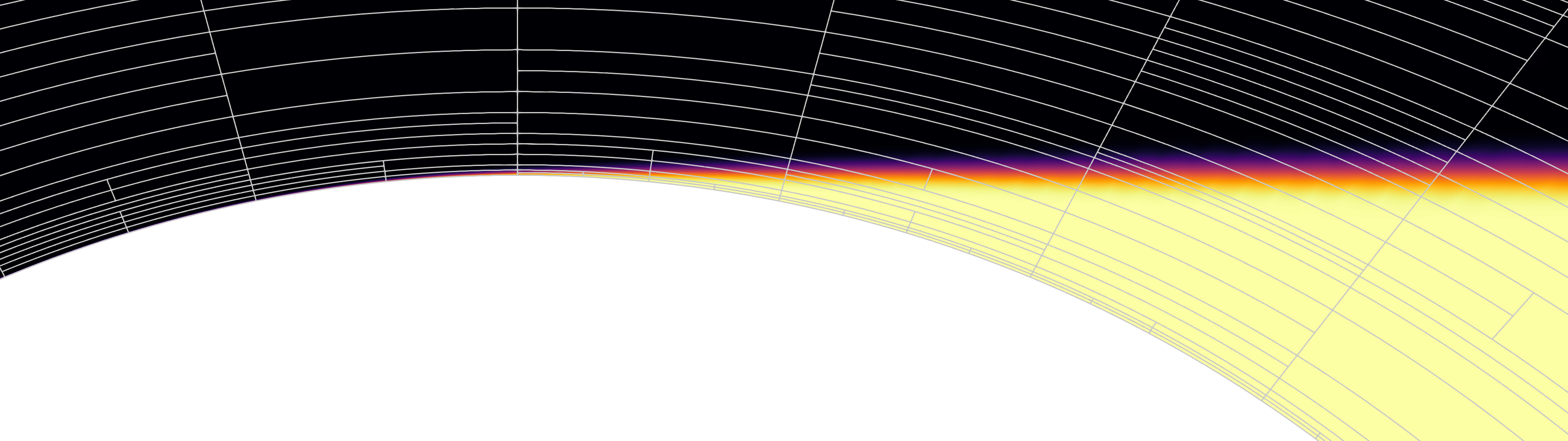}};%
      \node[shift={(0.25,-0.25)},draw=myred,anchor=north west,line width=3pt,inner sep=0, outer sep=0] at (img.east)
      {\includegraphics[width=.5\textwidth]{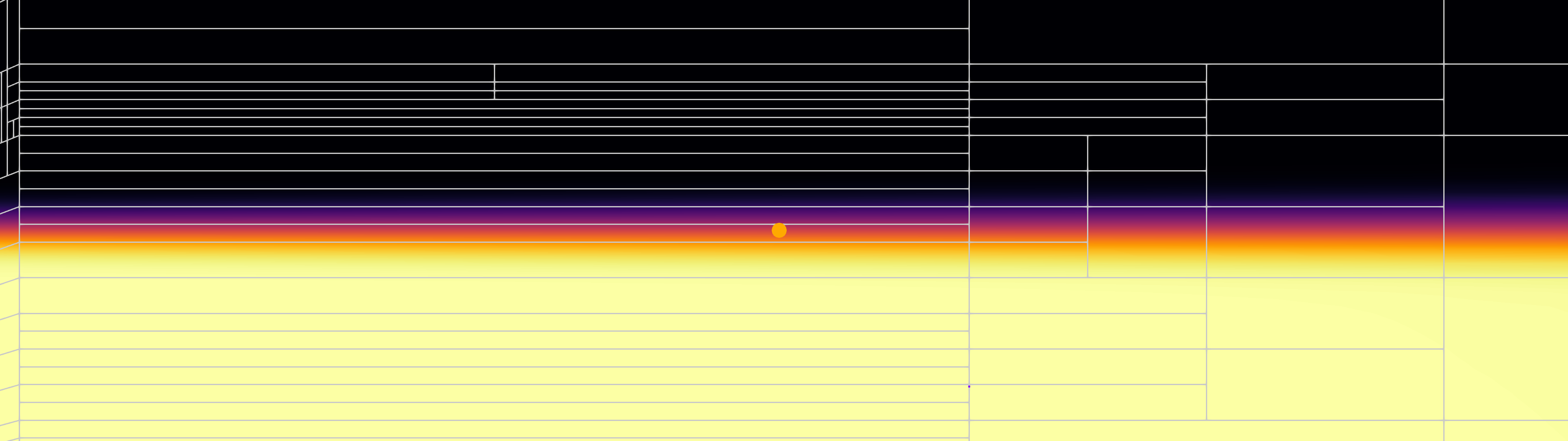}};%
      \coordinate (pos1) at
      ($(img.south west) + 0.63*(img.south east) + 0.583*(img.north west)$);
      \draw[myred,thick] (pos1) +(-20pt,-6pt) rectangle +(22pt,6pt) ;
      \coordinate (pos2) at
      ($(img.south west) + 0.29*(img.south east) + 0.5763*(img.north west)$);
      \draw[mygreen,thick] (pos2)  +(-12pt,-4pt) rectangle +(12pt,4pt);%
    \end{tikzpicture}}
  \caption{The solution to the Hemker
    problem with $\varepsilon = 10^{-4}$ in the final DWR loop. One can clearly see the adaptive mesh refinement at
    $\boldsymbol{x}_{c}$ and upstream of the point due to the goal~\eqref{eq:point-goal}.}\label{fig:hemker-pv-1e-4-solution}
  \vspace{4pt}
  \includegraphics[width=\textwidth]{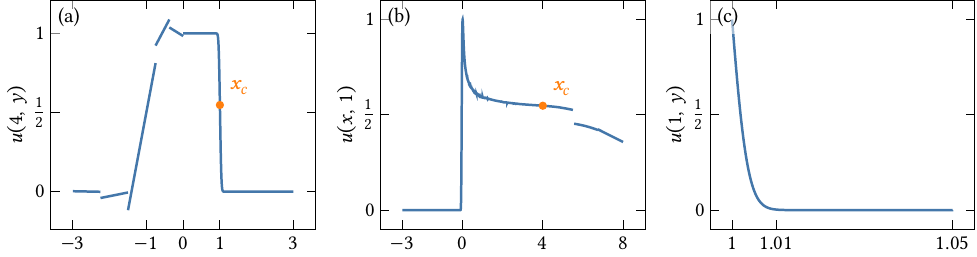}
  \caption{\label{fig:hemker-pv-1e-4}Cut lines of the solution to the Hemker
    problem with $\varepsilon = 10^{-4}$. In (a) a
    cut through the interior layers is plotted. In (b), the solution
    within the upper interior layer is plotted.
    In both cases we mark the control point which we use for the
    goal oriented error control. In (c) a cut through the boundary
    layer is plotted.}
  \vspace{4pt}
  \scalebox{0.97}{
  \begin{tikzpicture}
    \node[anchor=south west,inner sep=0, outer sep=0] (img) at (0,0)
    {\includegraphics[width=.5\textwidth]{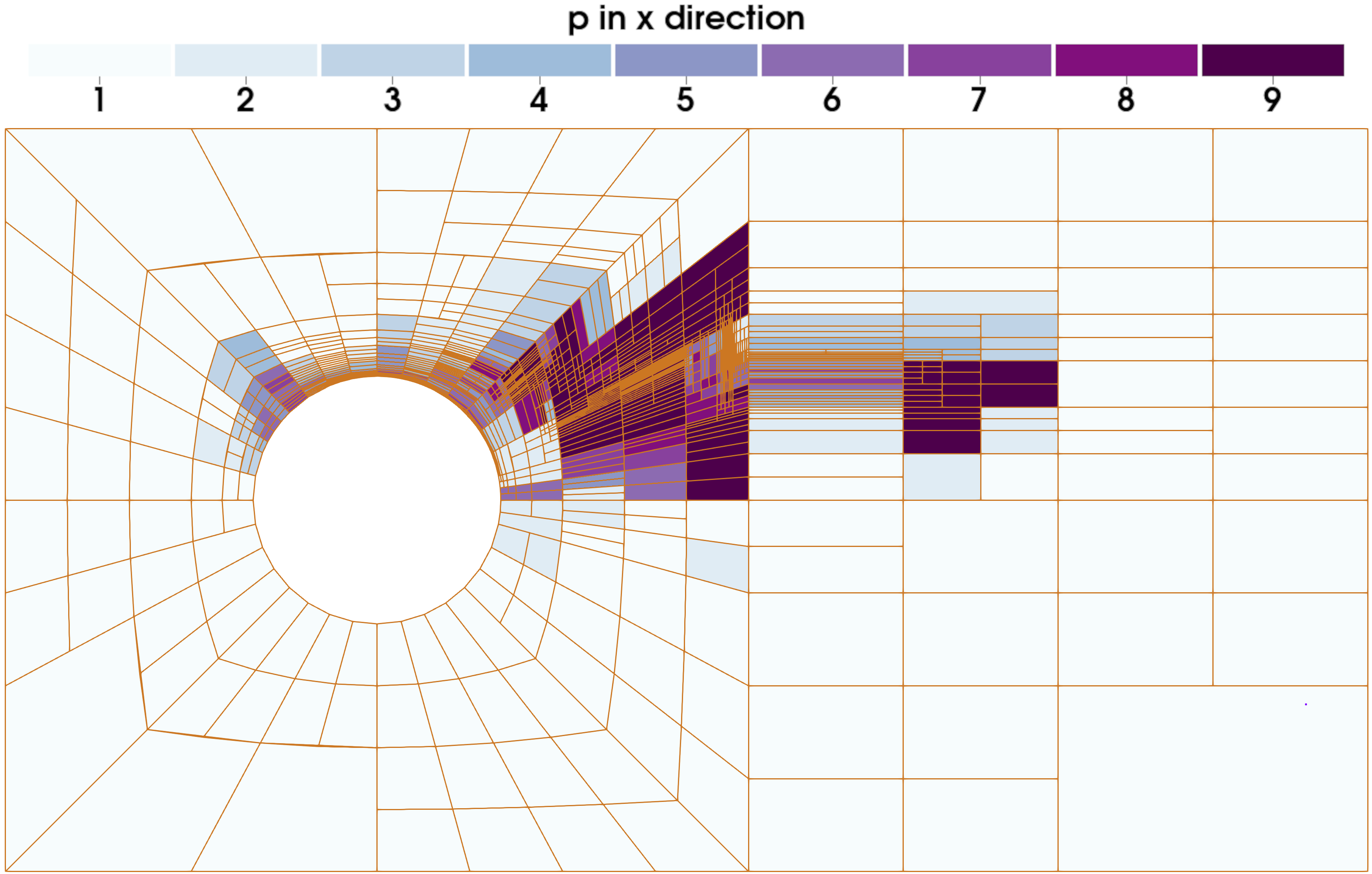}};%
    \node[shift={(0.0,-0.2)},draw=mygreen,anchor=north,line width=3pt,inner sep=0, outer sep=0] at (img.south)
    {\includegraphics[width=.49\textwidth]{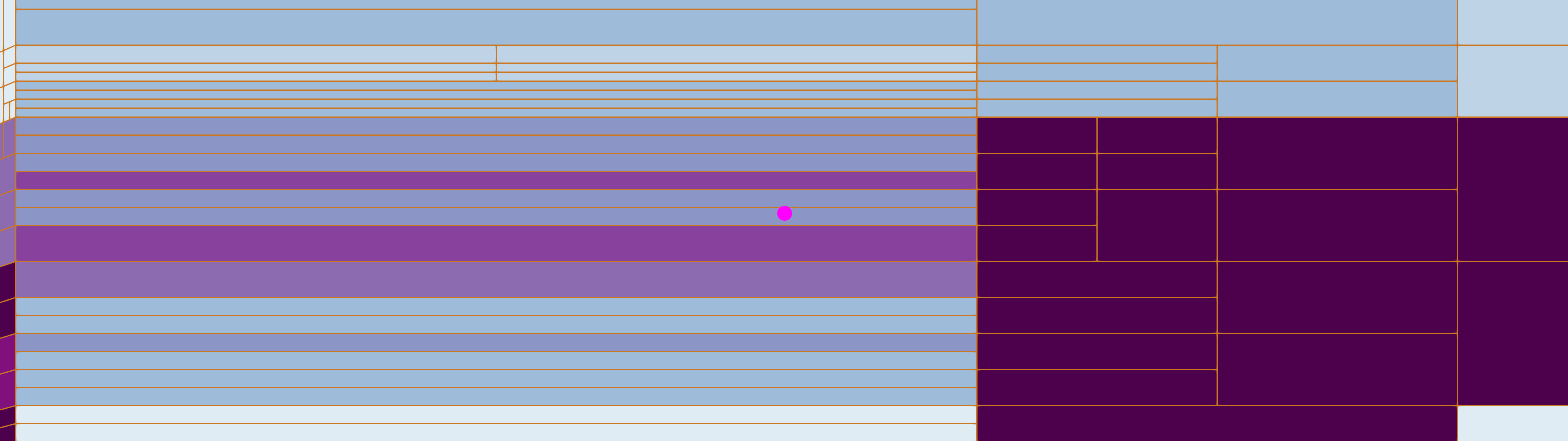}};%
    \coordinate (pos1) at
    ($(img.south west) + 0.63*(img.south east) + 0.5717*(img.north west)$);
    \draw[mygreen,thick] (pos1) +(-20pt,-6pt) rectangle +(22pt,6pt) ;
  \end{tikzpicture}}\hfill
\scalebox{0.97}{
  \begin{tikzpicture}
    \node[anchor=south west,inner sep=0, outer sep=0] (img) at (0,0)
    {\includegraphics[width=.5\textwidth]{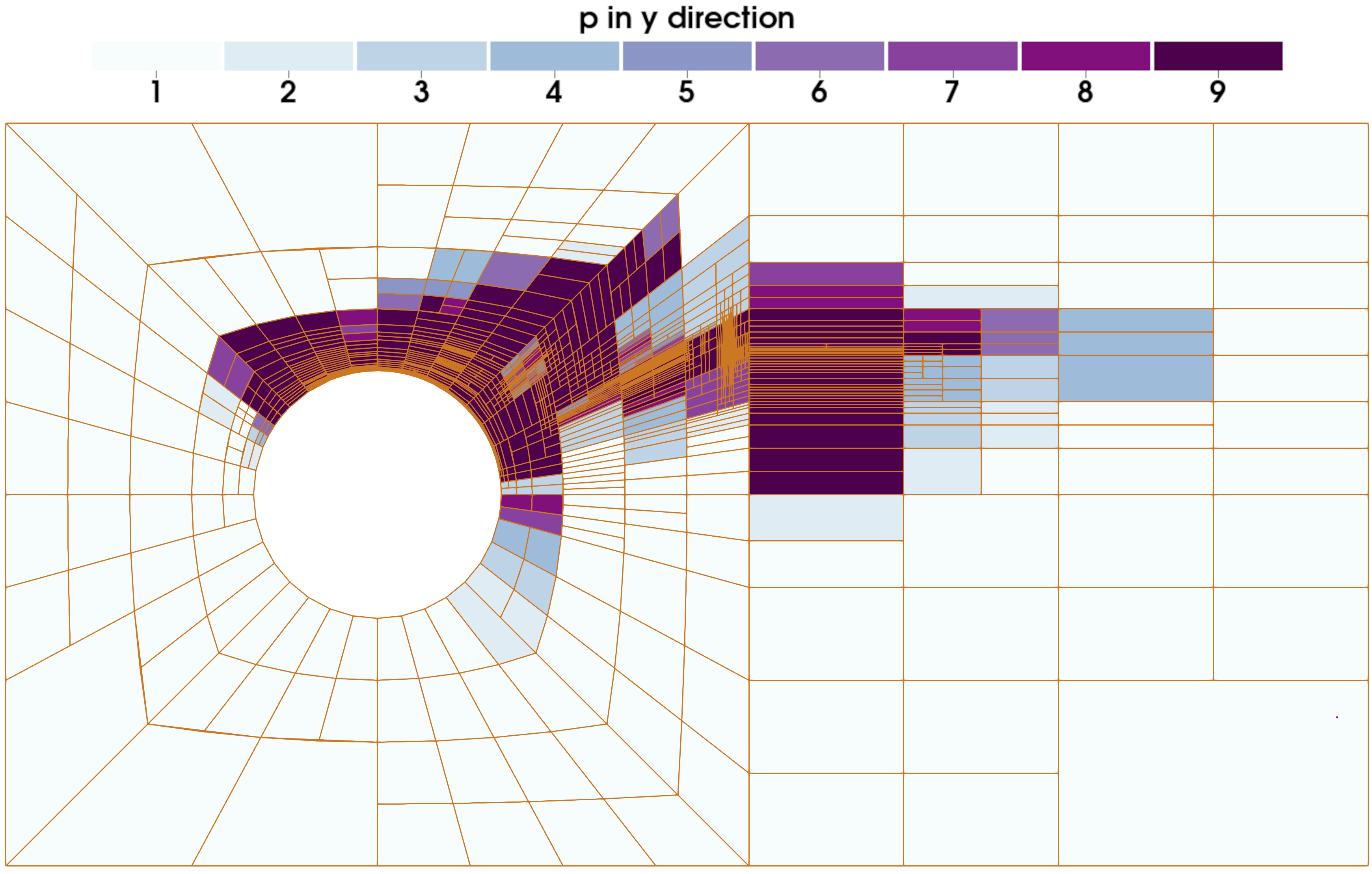}};%
    \node[shift={(0.0,-0.2)},draw=mygreen,anchor=north,line width=3pt,inner sep=0, outer sep=0] at (img.south)
    {\includegraphics[width=.49\textwidth]{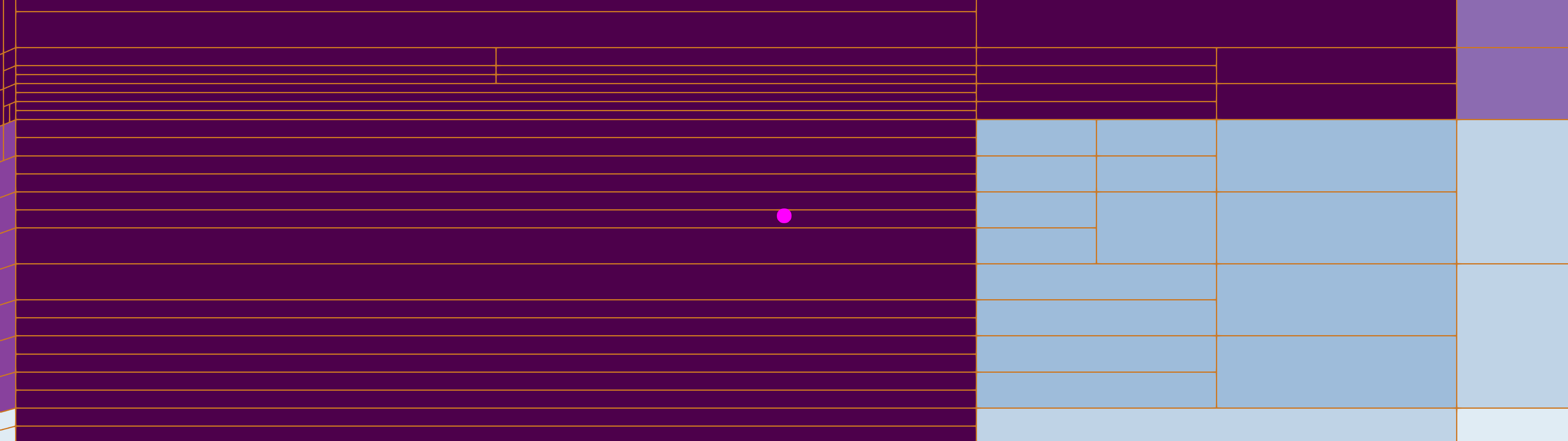}};%
    \coordinate (pos1) at
    ($(img.south west) + 0.63*(img.south east) + 0.5717*(img.north west)$);
    \draw[mygreen,thick] (pos1) +(-20pt,-6pt) rectangle +(22pt,6pt) ;
  \end{tikzpicture}}
  \caption{\label{fig:hemker-pv-1e-4-point-mesh}Anisotropic $hp$-mesh with
    polynomial degrees in $x$-direction (left) and $y$-direction (right) after
    $17$ refinement steps for the Hemker problem with $\varepsilon = 10^{-4}$.
    Within the mesh we observe high anisotropies in the mesh size and polynomial
    degree.}
\end{figure}

Tables~\ref{tab:hemker-pv-1e-4} and~\ref{tab:hemker-pv-1e-6} summarize the
adaptive refinement process for the Hemker problem with $\varepsilon=10^{-4}$
and $\varepsilon=10^{-6}$, respectively. For the less convection-dominated case
$\varepsilon=10^{-4}$ the final combined error $\eta_{\tau h}^{\textup{a}}$ is smaller for comparable computational effort. The more convection-dominated case $\varepsilon=10^{-6}$ requires a significantly larger number of cells and degrees of freedom to achieve comparable error reduction. This is due to the
need to resolve significantly thinner boundary and interior layers. Moreover, the indicators start higher and decay more slowly, underscoring the difficulty of the Hemker problem in this regime.
The maximal anisotropy
increases throughout the refinement process, reaching values up to $427$ and
$107$, indicating a strong preference for anisotropic refinement in space. The
space-time error estimate $\eta_{\tau h}^{\textup{a}}$ consistently decreases, demonstrating
the effectiveness of the goal-oriented adaptive strategy. These results
highlight the potential of the anisotropic $hp$-adaptive algorithm to
effectively address practically relevant, convection-dominated problems.

The left part of Fig.~\ref{fig:hemker-pv-1e-4-solution} shows refinement
concentrated around the obstacle and upstream of the goal point
$\boldsymbol{x}_c$. The zoomed panels highlight the isoparametric curved cells
close to the obstacle and the solution close to the goal point. The cut-line plots in
Fig.~\ref{fig:hemker-pv-1e-4} indicate that the solution remains free of
spurious oscillations despite the convection-dominated setting. The steep
gradients at the obstacle boundary and in the interior layer upstream of the
goal point are well resolved, suggesting that the adaptive procedure
successfully allocates refinement to regions with the highest influence on the
goal functional.
The width of the interior layer,
\begin{equation}\label{eq:ylayer}
  y_{\text{layer}}=y_1-y_0\,,
\end{equation}
is defined to be the length of the interval $[y_0,\,y_1]$ in which the
solution falls from $u(4, y_0) = 0.9$ to $u(4, y_1) = 0.1$. John et
al.~\cite{augustin_assessment_2011} provide a reference value of
$y_{\text{layer}}=0.0723$, which we achieve exactly in the 17th DWR loop at only
$\num[scientific-notation=false,round-precision=0]{62682}$ spatial degrees of freedom (cf.\ Tab.~\ref{tab:hemker-pv-1e-4}).

The meshes in Fig.~\ref{fig:hemker-pv-1e-4-point-mesh} show
high anisotropy: along the convection direction the polynomial degree is
predominantly increased, whereas in the transverse direction refinement is
primarily performed in $h$. This directional refinement pattern is in line with
expectations for convection-dominated problems, where $p$-refinement along the
flow direction and $h$-refinement across thin layers are expected to be most
efficient. The observed behavior indicates that the adaptive $hp$-strategy
concentrates refinement in regions most relevant for the goal functional while
selecting refinement types appropriate for the local solution structure.

\FloatBarrier
\subsection{Fichera Corner} The Fichera corner problem introduces boundary
layers within the three-dimensional domain $\Omega$ given by $\Omega=
(-1,1)^3\setminus [0,\,1]^3$,
which has a re-entrant corner at the origin \((0,\, 0,\, 0)\). The problem is defined
by an initial concentration
$u(\symbf{x},\,0)=\prod_{d=1}^{3} \left[ \frac{1}{4} \left(1 +
    \tanh\left(\frac{\symbf{x}_d + 1}{\varepsilon}\right)\right) \left(1 -
    \tanh\left(\frac{\symbf{x}_d - 1}{\varepsilon}\right)\right) \right]$, which
is advected by the convection field $\symbf{b}= (1,\, 1,\, 1)$ towards the
singularity at the origin over an interval $I=(0,1]$. This presents significant
numerical challenges due to singularities and steep gradients near the boundary.
The diffusion is set to $\varepsilon=\num[round-precision=1]{5.e-4}$ and the reaction coefficient
is given by $\alpha=0$. Despite the moderate diffusion parameter, the
elliptic nature of the corner singularity exacerbates the problem's complexity,
rendering it a challenging benchmark for the numerical methods. Homogeneous
Dirichlet boundary conditions are imposed on the three faces adjacent to the
origin, while homogeneous Neumann conditions are enforced on all other boundary
faces (cf.~Fig.~\ref{fig:geo-fichera}). The aim is to control the error within a
control point $\boldsymbol x_{c}=(-10^{-4},\,0.25,\,-10^{-4})$ in the boundary
layer near an edge. To this end, we use the goal functional
\begin{equation}\label{eq:fichera-point-goal}
  J(u)=u(\boldsymbol x_{c},\,T)\,,
\end{equation}
which is defined analogous to~\eqref{eq:point-goal}. In space and time, the
refinement fraction is set to
\(\theta_{\text{space}}^{\text{ref}} = \frac{1}{7}\) and
$\theta_{\text{time}}^{\text{ref}} = \frac{1}{8}$. The coarsening fraction is
set to
\(\theta_{\text{space}}^{\text{co}} = \theta_{\text{time}}^{\text{co}} =
\frac{1}{25}\). We use polynomial degrees \(1 \leq p \leq 5\) and
\(0 \leq k \leq 5\) in space and time, respectively. The initial space-time
triangulation consists of a single cell in time and a once uniformly refined
spatial coarse mesh (cf. Fig.~\ref{fig:geo-fichera}).

  \usetikzlibrary {perspective}
  \newcommand\simplecuboid[9]{%
    \fill[#7,draw=black] (tpp cs:x=#1,y=#2,z=#6)
    -- (tpp cs:x=#1,y=#5,z=#6)
    -- (tpp cs:x=#4,y=#5,z=#6)
    -- (tpp cs:x=#4,y=#2,z=#6) -- cycle;
    \fill[#8,draw=black]  (tpp cs:x=#1,y=#2,z=#3)
    -- (tpp cs:x=#1,y=#2,z=#6)
    -- (tpp cs:x=#1,y=#5,z=#6)
    -- (tpp cs:x=#1,y=#5,z=#3) -- cycle;
    \fill[#9,draw=black] (tpp cs:x=#1,y=#2,z=#3)
    -- (tpp cs:x=#1,y=#2,z=#6)
    -- (tpp cs:x=#4,y=#2,z=#6)
    -- (tpp cs:x=#4,y=#2,z=#3) -- cycle;}
  \begin{figure}[!htb]
    \centering
    \begin{tikzpicture}[scale=1.4,
      isometric view,very thick,
      perspective={
        p = {(16,0,0)},
        q = {(0,16,0)},
        r = {(0,0,-16)}}]
      \draw[-stealth] (0,0,0) -- (-1,0,0) node[pos=1.1]{$z$};
      \draw[-stealth] (0,0,0) -- (0,-1,0) node[pos=1.1]{$x$};
      \draw[-stealth] (0,0,0) -- (0,0,1) node[pos=1.1]{$y$};
    \end{tikzpicture}
  \begin{tikzpicture}[scale=1.4,
    isometric view,very thick,
    perspective={
      p = {(16,0,0)},
      q = {(0,16,0)},
      r = {(0,0,-16)}}]
    \simplecuboid{0}{0}{0}{1}{1}{1}{myblue}{myblue}{myblue}
    \simplecuboid{-1}{0}{-1}{0}{1}{0}{myblue}{myblue}{myblue}
    \simplecuboid{0}{-1}{-1}{1}{0}{0}{myblue}{myblue}{myblue}
    \simplecuboid{-1}{-1}{-1}{0}{0}{0}{mygreen}{myblue}{myblue}
    \simplecuboid{0}{-1}{0}{1}{0}{1}{myblue}{mygreen}{myblue}
    \simplecuboid{-1}{0}{0}{0}{1}{1}{myblue}{myblue}{mygreen}
  \end{tikzpicture}
  \hspace*{1em}
  \rule{1.5pt}{4.5cm}
  \hspace*{1em}
  \begin{tikzpicture}[scale=1.4,
    isometric view,very thick,
    perspective={
      p = {(16,0,0)},
      q = {(0,16,0)},
      r = {(0,0,-16)}}]
    \simplecuboid{0}{0}{0}{1}{1}{1}{myblue}{myblue}{myblue}
    \simplecuboid{-1}{0}{-1}{0}{1}{0}{myblue}{myblue}{myblue}
    \simplecuboid{0}{-1}{-1}{1}{0}{0}{myblue}{myblue}{myblue}
    \simplecuboid{-1}{-1}{-1}{0}{0}{0}{myblue}{myblue}{myblue}
    \simplecuboid{0}{-1}{0}{1}{0}{1}{myblue}{myblue}{myblue}
    \simplecuboid{-1}{0}{0}{0}{1}{1}{myblue}{myblue}{myblue}
    \simplecuboid{-1}{-1}{0}{0}{0}{1}{myblue}{myblue}{myblue}
  \end{tikzpicture}
  \begin{tikzpicture}[scale=1.4,
    isometric view,very thick,
    perspective={
      p = {(16,0,0)},
      q = {(0,16,0)},
      r = {(0,0,-16)}}]
    \draw[-stealth] (0,0,0) -- (1,0,0) node[pos=1.1]{$x$};
    \draw[-stealth] (0,0,0) -- (0,1,0) node[pos=1.1]{$z$};
    \draw[-stealth] (0,0,0) -- (0,0,-1) node[pos=1.1]{$y$};
  \end{tikzpicture}
  \caption{Geometry and coarse mesh of the Fichera cube. We
    depict views towards the origin \((0,\, 0,\, 0)\) in the direction
    \((-1,\, -1,\, -1)\) (left) and in the opposite direction \((1,\, 1,\, 1)\) (right). Green coloring
    corresponds to homogeneous Dirichlet BCs, while blue indicates homogeneous Neumann BCs.}\label{fig:geo-fichera}
  \includegraphics[width=.48\textwidth]{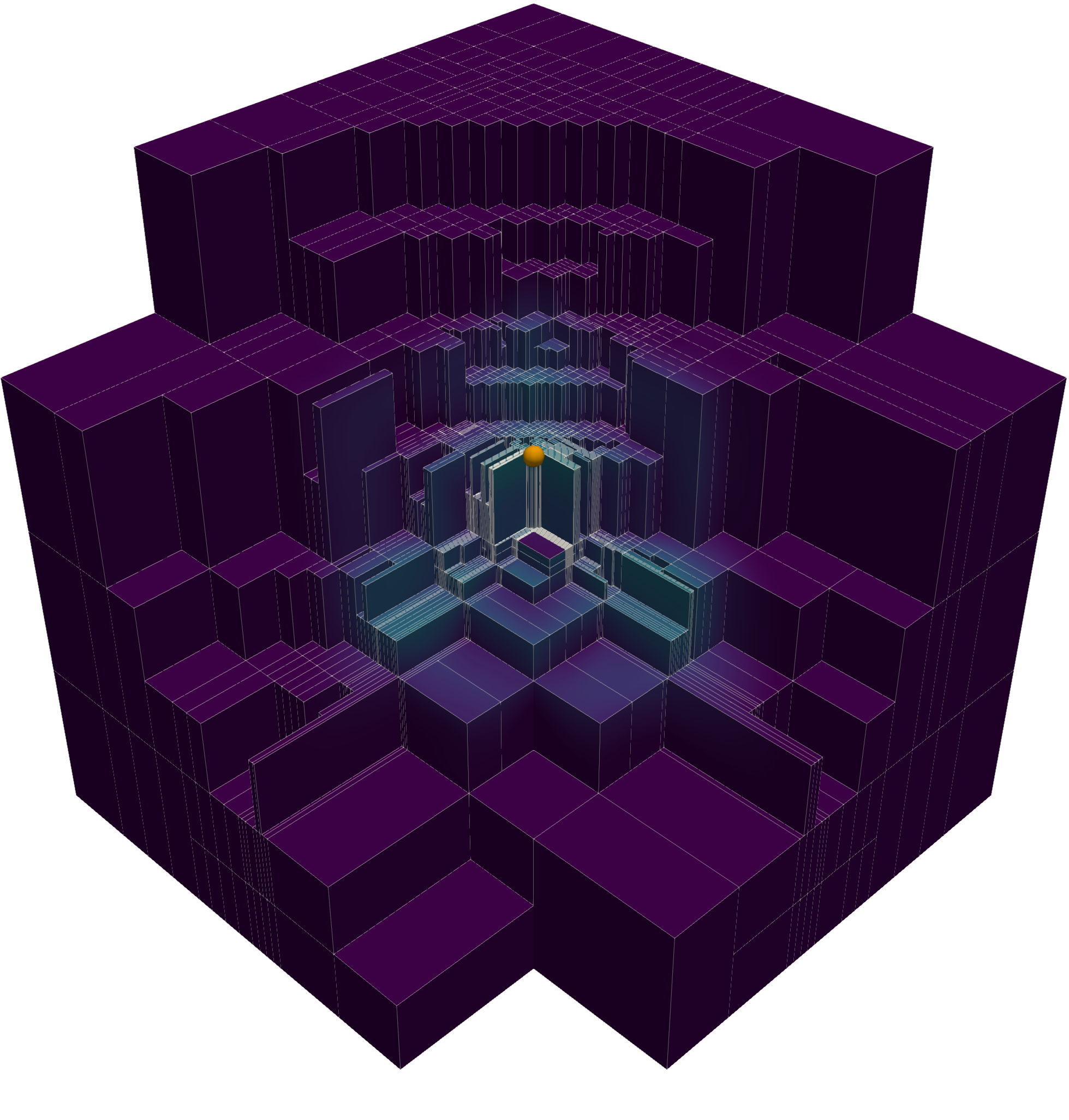}\hfill
  \includegraphics[width=.48\textwidth]{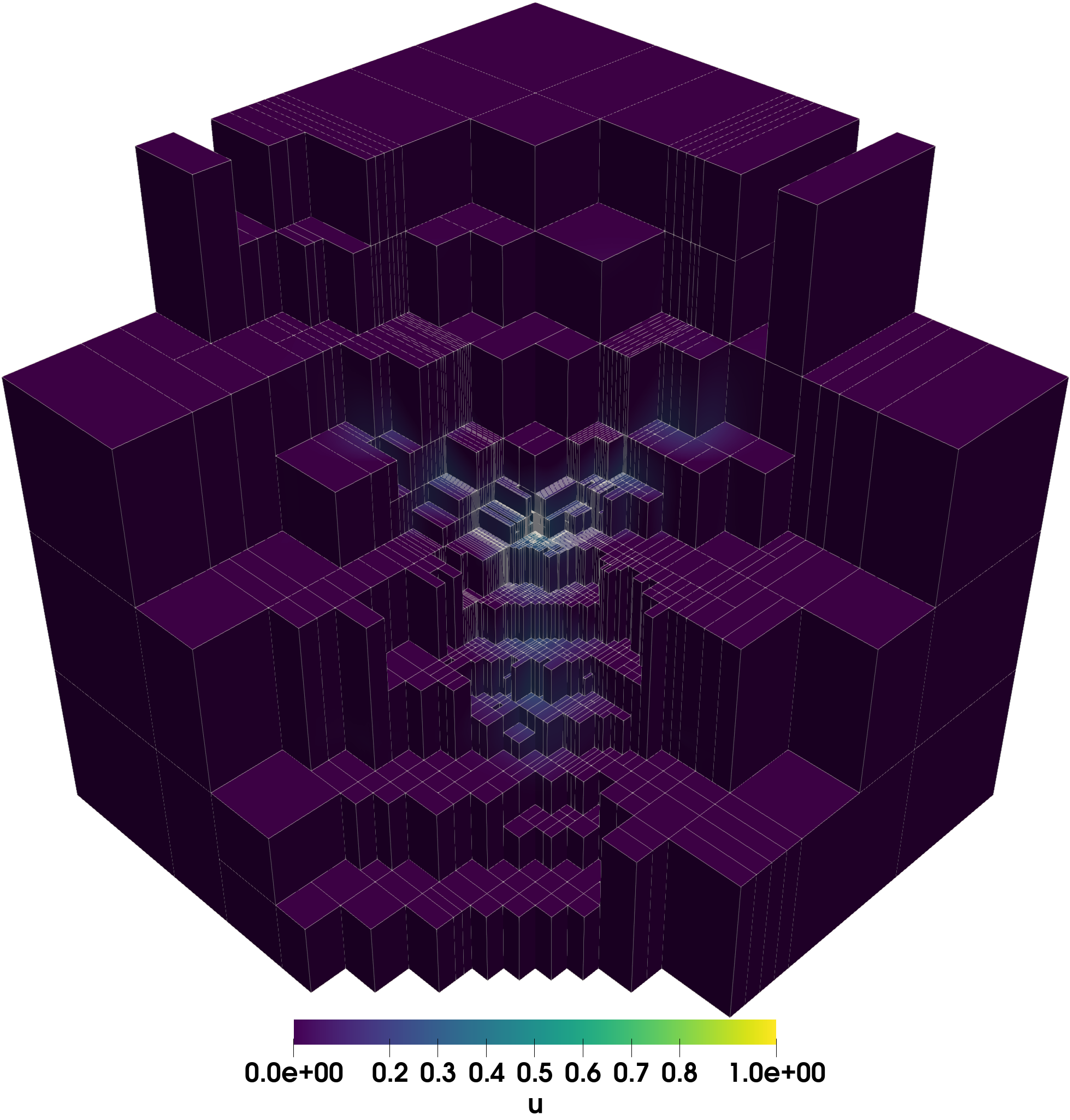}
  \caption{Solution and adapted mesh of the Fichera cube example. The depicted views are the same as above in Fig.~\ref{fig:geo-fichera}. The point $\boldsymbol x_{c}$ is marked in orange}\label{fig:soln-fichera}
  \end{figure}
\begin{figure}

\end{figure}
\begin{table}
  \caption{\label{tab:fichera}Number of DoFs, error estimators, error, and effectivity indices for the Fichera corner problem with point value control~\eqref{eq:fichera-point-goal} and diffusion coefficient $\varepsilon = \num[round-precision=1]{5.0e-4}$.}
  \begin{center}\scriptsize
    \begin{tabular}{S[scientific-notation=false,round-precision=0,table-format=2]
    S[scientific-notation=false,round-precision=0,table-format=3]
    S[scientific-notation=false,round-precision=0,table-format=5]
    S[scientific-notation=false,round-precision=0,table-format=5]
    S[scientific-notation=false,round-precision=0,table-format=1]
    S[table-format=1.1e2]%
    S[table-format=1.1e2]%
    S[table-format=1.1e2]
    S[table-format=1.1e2]
    S[table-format=1.1e2]
    S[table-format=1.1e2]
    S[table-format=1.1e2]
    S[table-format=1.1e2]
    S[table-format=1.1e2]}
      \toprule
      \mc{$\ell$} & \mc{$|\mathcal T_h|$} & \mc{$N_{\boldsymbol x}^{\mathrm{lo}}$} & \mc{$N_{\boldsymbol x}^{\mathrm{ho}}$} & \mc{$N_t$} & \mc{\textup{work}} & \mc{$\rho_{\max}$} & \mc{$\eta_{h,1}$} & \mc{$\eta_{h,2}$} & \mc{$\eta_{h,3}$} & \mc{$\eta_h^{\textup{a}}$} & \mc{$\eta_\tau$} & \mc{$\eta_{\tau h}^{\textup{a}}$} \\
\midrule
0  &  56    &  448     & 1512  & 1  & 896      & 1   & 2.579e-3 & -1.528e-3 &  2.579e-3 &  3.630e-3 & -9.782e-2 & -9.419e-2 \\
1  &  56    &  556     & 1746  & 2  & 11120    & 1   & 2.773e-3 &  2.588e-3 &  2.773e-3 &  8.134e-3 & -4.726e-2 & -3.912e-2 \\
2  &  56    &  695     & 2027  & 2  & 13900    & 1   & 3.743e-3 &  1.793e-3 &  3.743e-3 &  9.279e-3 & -5.152e-2 & -4.224e-2 \\
3  &  64    &  972     & 2670  & 3  & 37908    & 2   & 1.902e-3 &  1.652e-3 &  2.260e-3 &  5.814e-3 & -2.025e-3 &  3.789e-3 \\
4  &  66    & 1207     & 3150  & 3  & 47073    & 2   & 1.947e-3 &  5.739e-4 &  2.713e-3 &  5.234e-3 & -1.987e-3 &  3.247e-3 \\
5  &  73    & 1638     & 4058  & 3  & 63882    & 2   & 1.646e-3 &  9.078e-4 &  2.307e-3 &  4.861e-3 & -1.161e-4 &  4.745e-3 \\
6  & 134    & 3731     & 8760  & 3  & 145509   & 4   & -2.181e-4 &  2.200e-4 & -3.674e-4 & -3.655e-4 & -3.912e-3 & -4.278e-3 \\
7  & 277    & 11121    & 23981 & 3  & 433719   & 8   & 4.388e-5 & -1.935e-5 & -1.448e-4 & -1.120e-4 & -5.827e-3 & -5.947e-3 \\
8  & 503    & 21835    & 46560 & 3  & 851565   & 16  & -7.430e-6 & -1.371e-6 & -2.331e-4 & -2.420e-4 & -5.941e-3 & -6.183e-3 \\
9  & 812    & 42199    & 86481 & 3  & 1645761  & 32  & -5.367e-5 & -1.062e-4 & -1.782e-4 & -3.381e-4 & -4.668e-3 & -5.006e-3 \\
10 & 1438   & 87733    & 173442 & 3 & 3421587  & 64  & 2.614e-5 & 3.753e-5 & -9.807e-5 & -3.440e-5 & -2.229e-3 & -2.263e-3 \\
11 & 2873   & 214811   & 407097 & 3 & 8377629  & 128 & 3.652e-6 & 1.980e-4 & -1.925e-5 & 1.817e-4  & -5.215e-4 & -3.399e-4 \\
12 & 6021   & 538839   & 988246 & 3 & 21014721 & 256 & 3.828e-6 & 1.118e-4 & -3.307e-6 & 1.123e-4  & 1.940e-3  & 2.052e-3 \\
13 & 11573  & 1172210  & 2101439 & 3 & 45716190 & 512 & 2.725e-6 & 7.560e-5 & -2.407e-6 & 7.592e-5  & 3.368e-4  & 4.127e-4 \\
\bottomrule
\end{tabular}
  \end{center}
\end{table}
  \begin{figure}[htb]
    \centering
    \begin{tabular}{ccc}
      \includegraphics[width=0.3\textwidth]{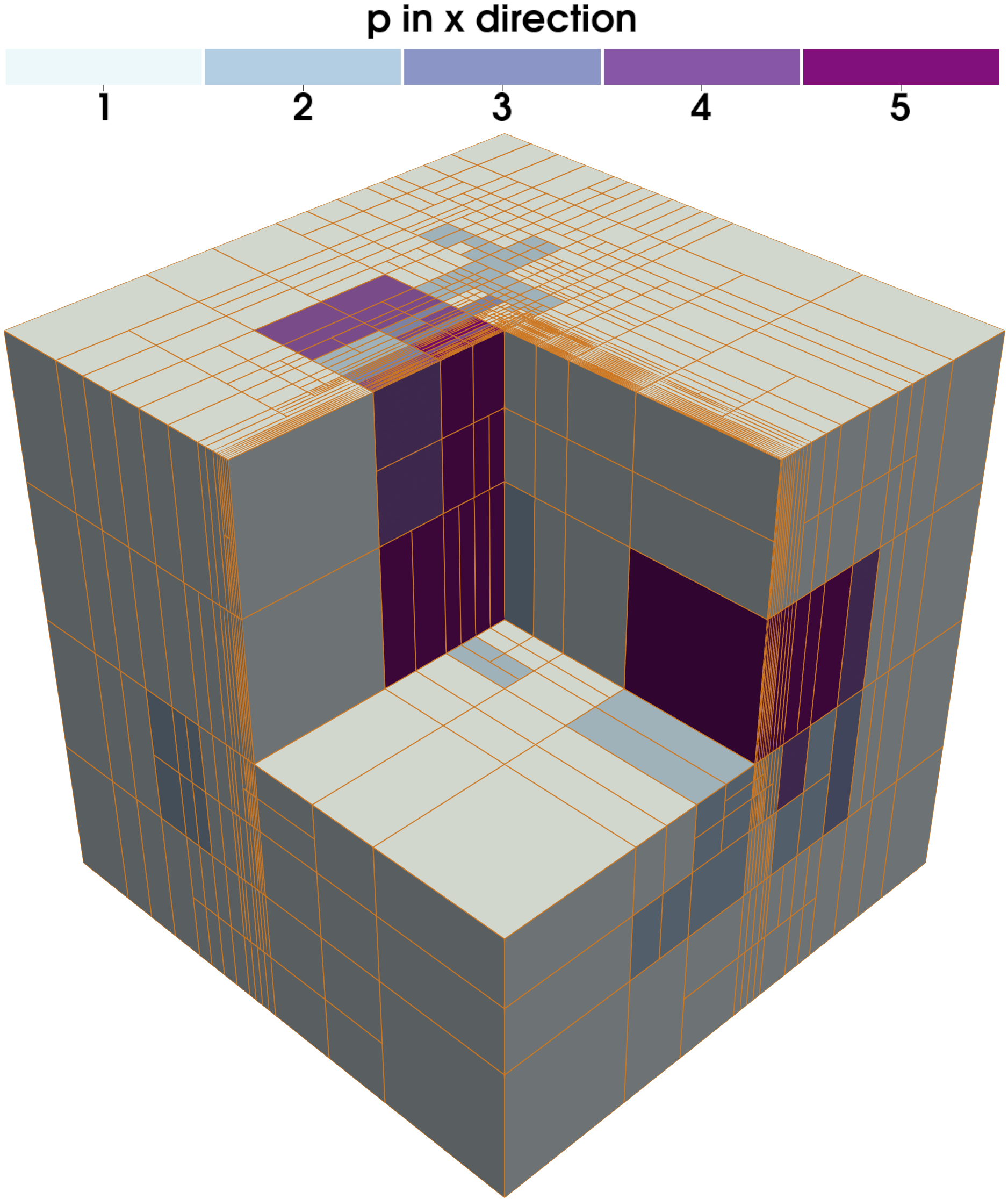} & \includegraphics[width=0.3\textwidth]{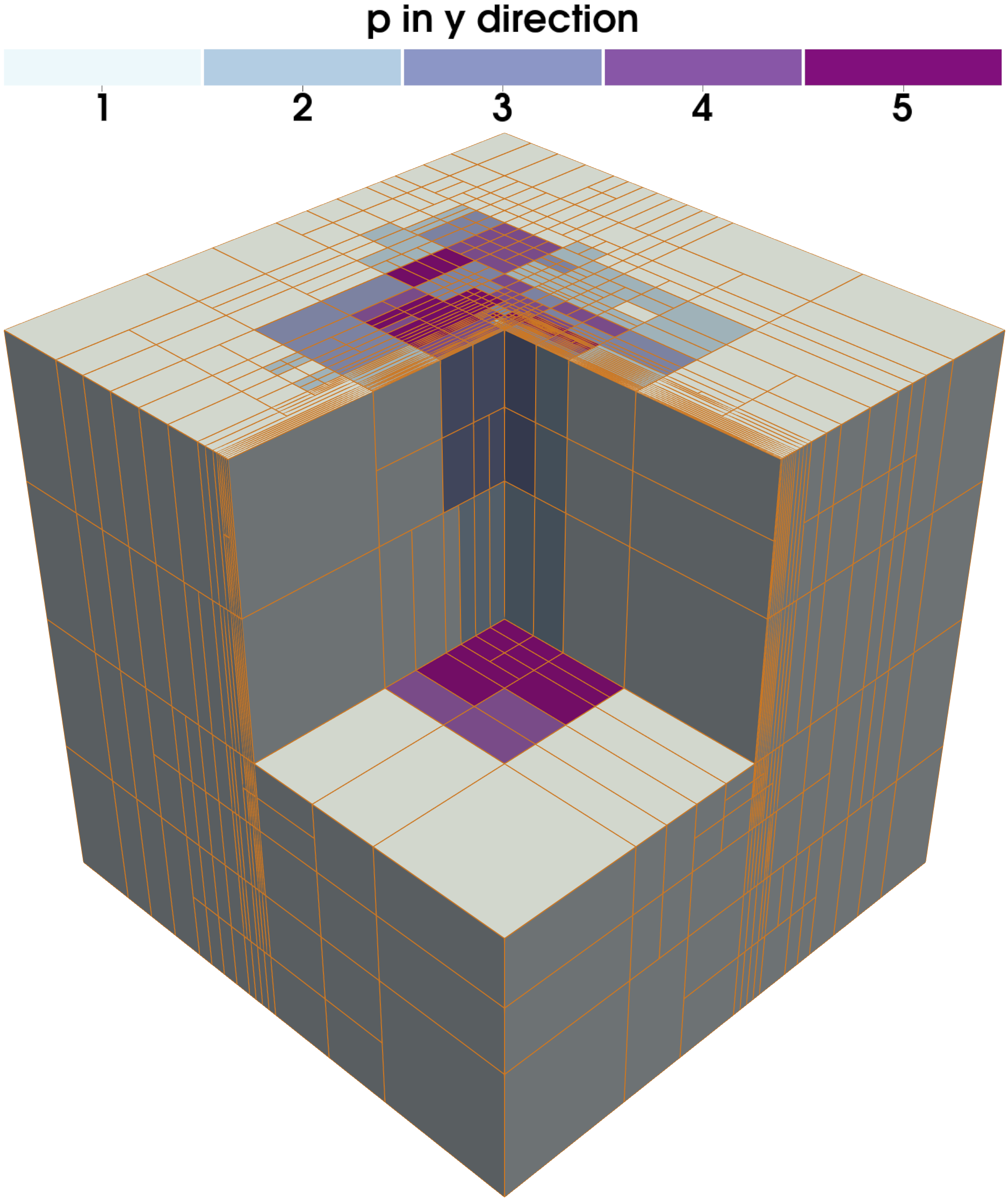} & \includegraphics[width=0.3\textwidth]{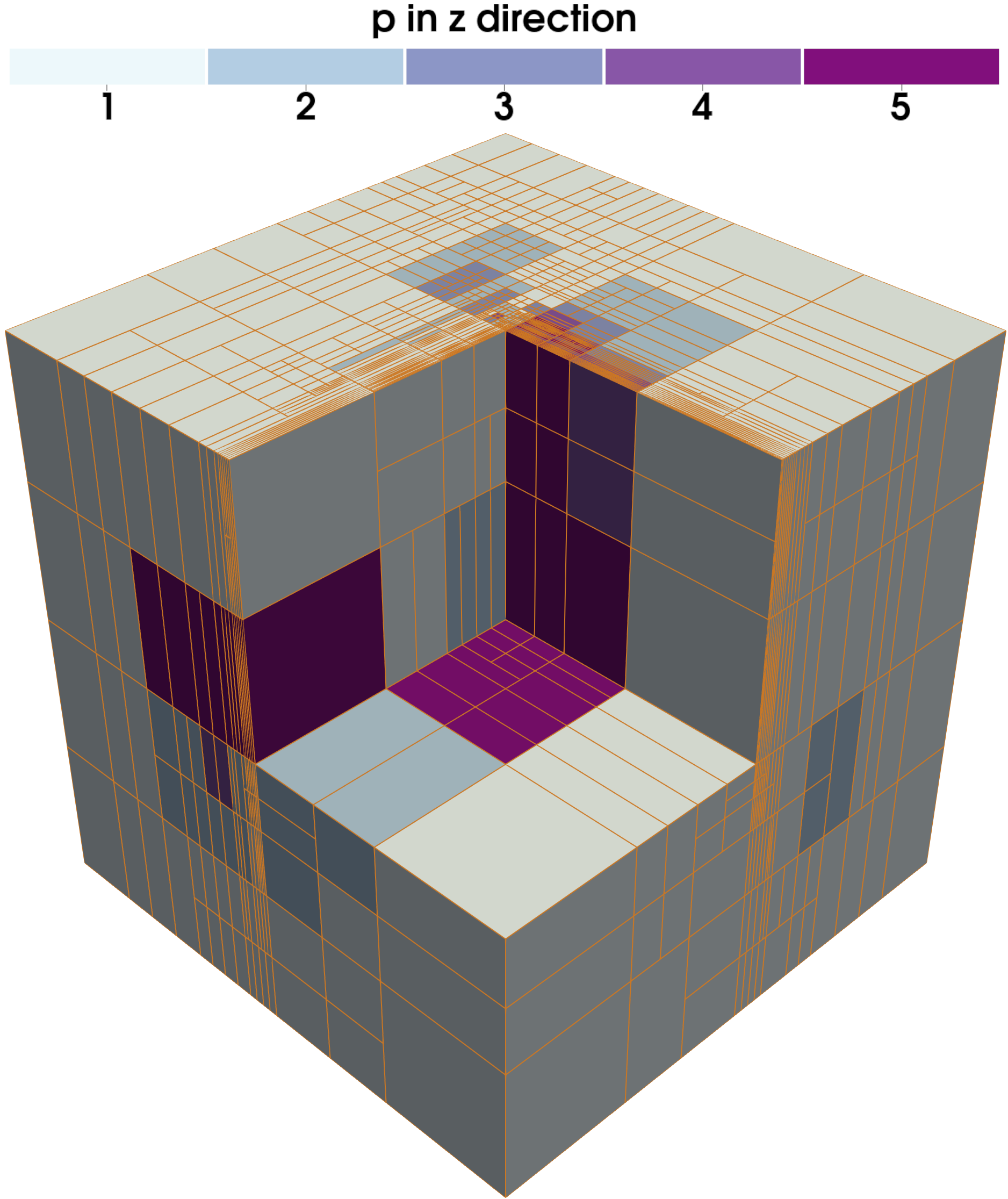} \\

      \includegraphics[width=0.3\textwidth]{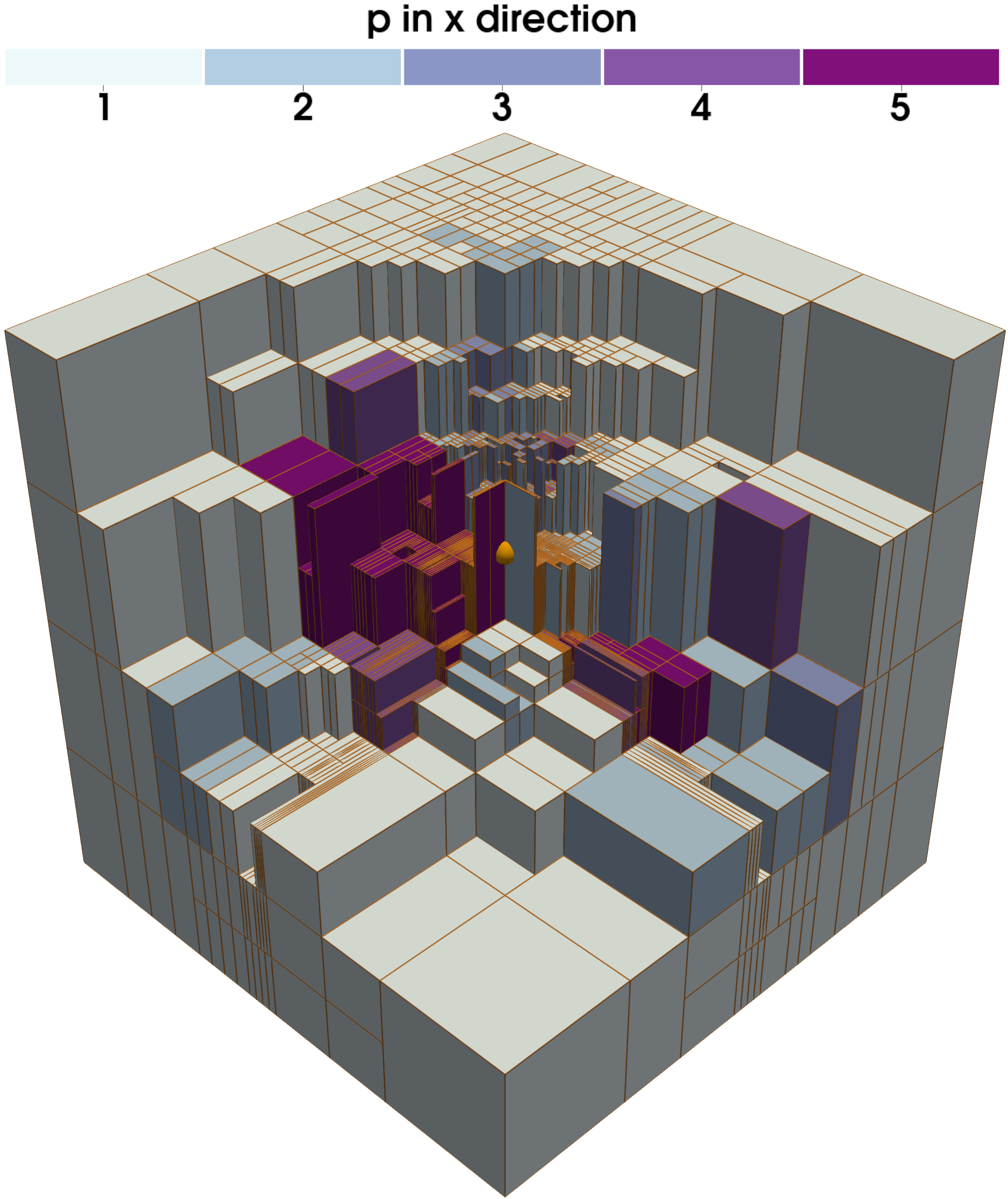} & \includegraphics[width=0.3\textwidth]{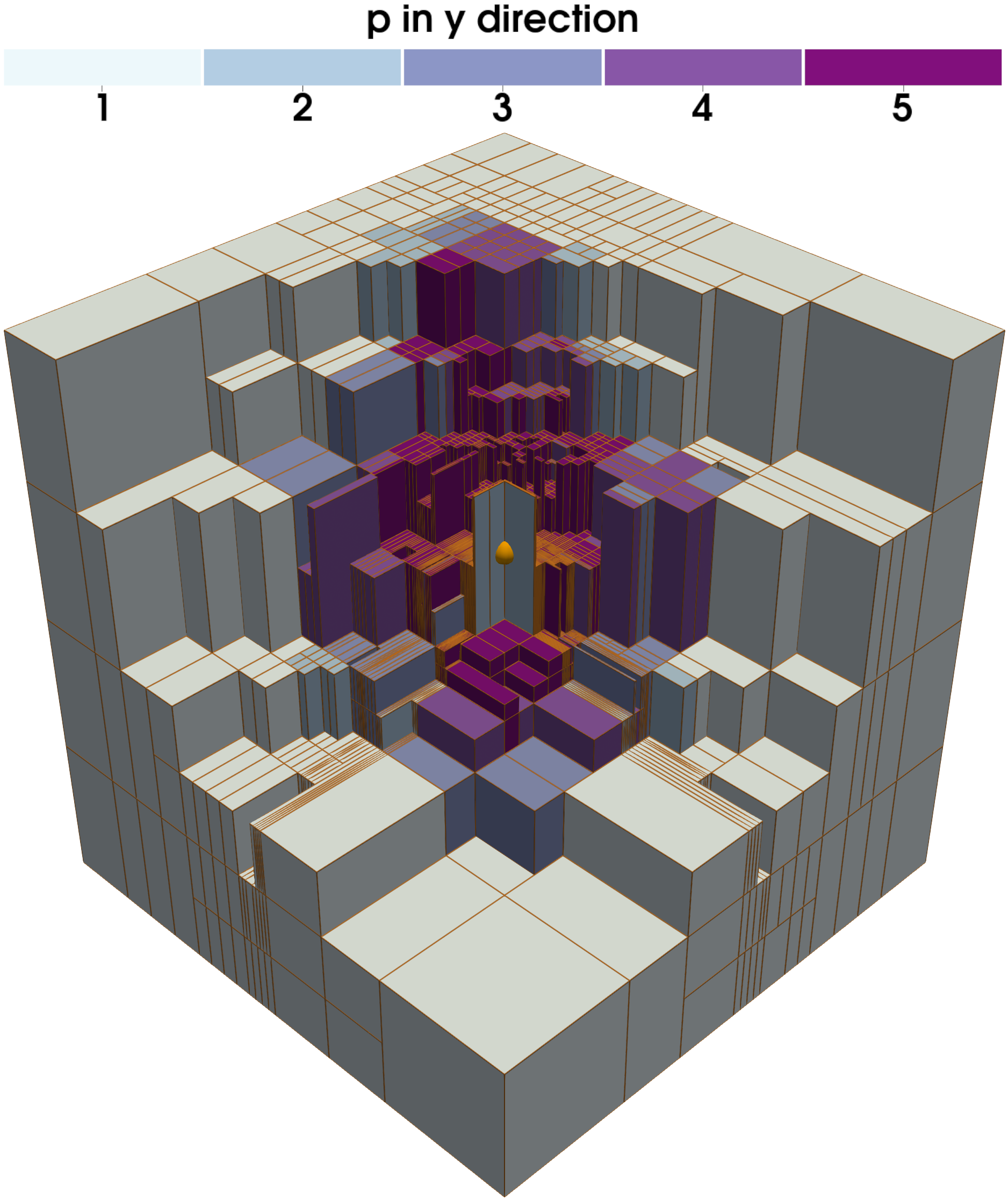} & \includegraphics[width=0.3\textwidth]{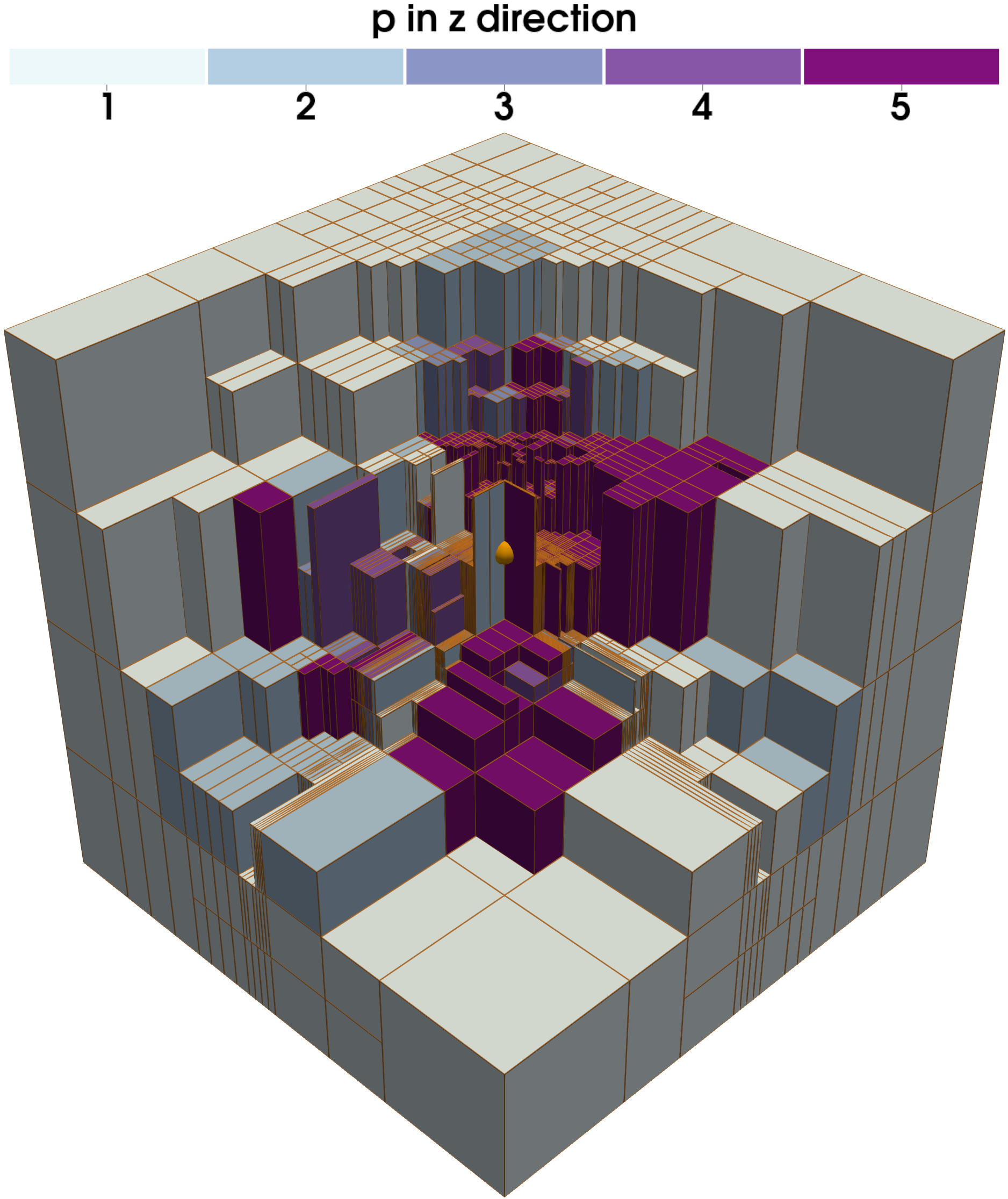} \\

    \end{tabular}
    \caption{\label{fig:poly-grid}Polynomial degree distributions in $x$-, $y$-, and $z$-direction (columns), with two zoomed-in views for each direction (rows).}
  \end{figure}
  Fig.~\ref{fig:soln-fichera} shows the solutions after 14 DWR loops at final
  time. We observe that the initial concentration has been transported towards
  the corner and edge singularities. Tab.~\ref{tab:fichera} summarizes the
  adaptive refinement for the Fichera corner problem. As the DWR algorithm
  progresses, the number of cells grows, reflecting the need to resolve the
  singularities with $h$-refinement. The maximal anisotropy $\rho_{\max}$ increases
  up to $512$, showing that the algorithm strongly exploits directional
  refinement. The combined space-time error estimate $\eta_{\tau h}^{\textup{a}}$ consistently
  decreases, indicating that the goal-oriented $hp$ adaptive strategy remains
  effective despite the reduced regularity. These results demonstrate that the
  anisotropic $hp$-adaptivity can handle geometric singularities, where purely
  isotropic refinement would be far less efficient.

  Fig.~\ref{fig:poly-grid} shows the polynomial degree distributions in $x$-,
  $y$-, and $z$-direction (columns) together with clipped views (rows). The
  plots demonstrate that the adaptive strategy is impacted by the edge and
  corner singularities, where $h$-refinement is employed. This shows
  that the anisotropic $hp$-adaptive approach is capable of efficiently
  capturing the reduced regularity close to edges and corners.

  \begin{figure}
    \includegraphics[width=0.65\textwidth]{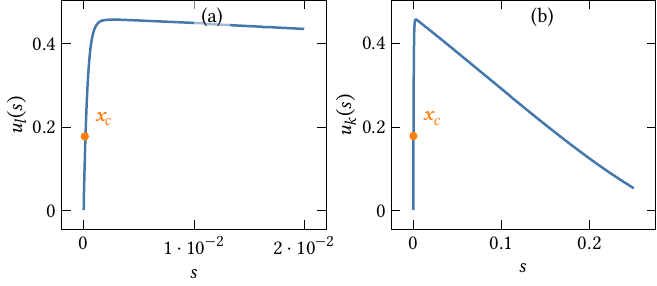}\hfill
    \includegraphics[width=0.33\textwidth]{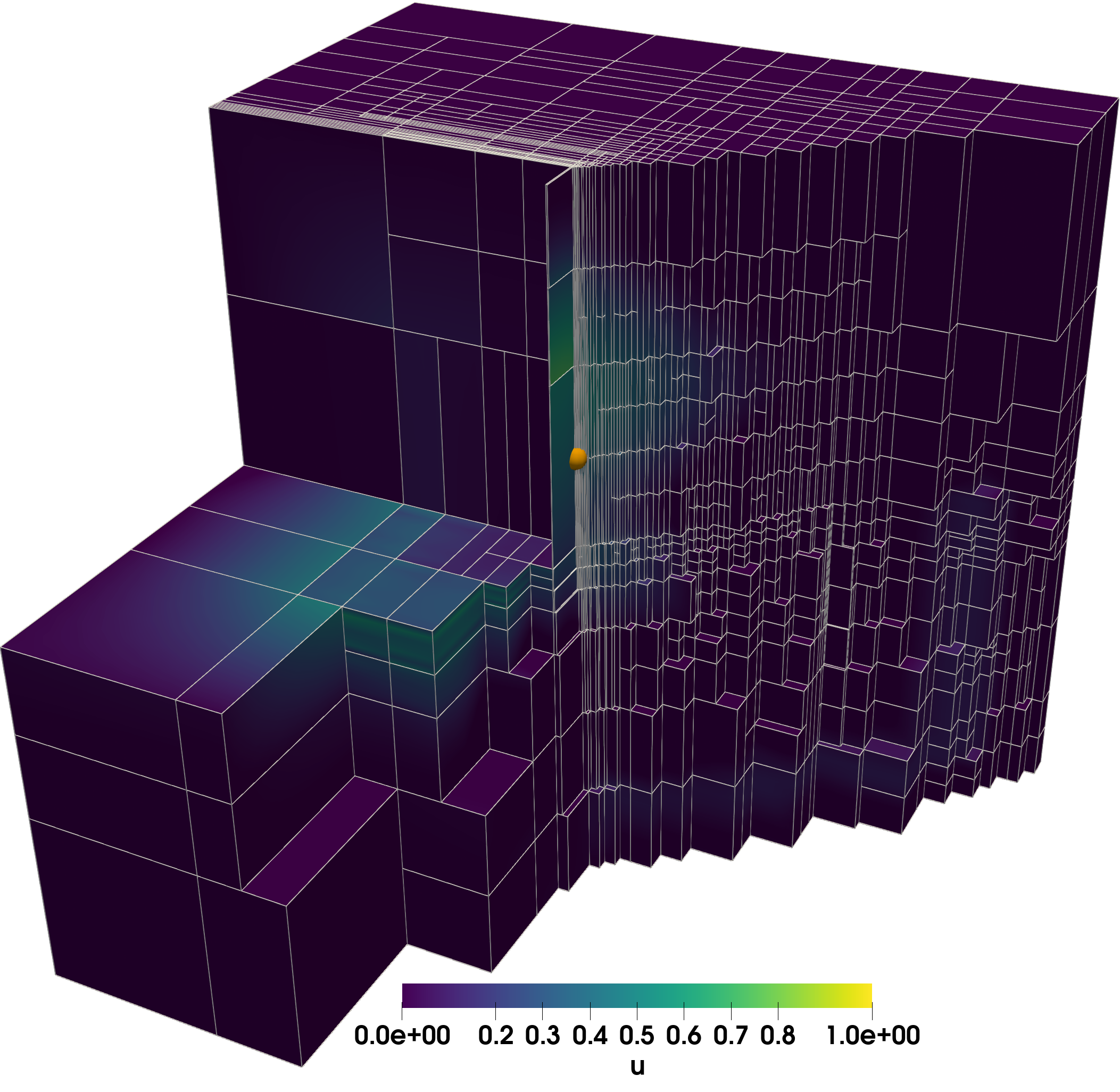}
    \caption{\label{fig:fichera-cut}Cut lines of the solution to the Fichera corner
      problem. In (a) a
      cut orthogonal to the boundary is plotted. In (b) a cut through $\boldsymbol x_{c}$ in downstream direction of the convective field $\boldsymbol b$ is plotted (cf~\eqref{eq:cutlines}).
      In both cases the control point which we use for the
      goal oriented error control is marked. On the right we plot another cut of the solution along the plane $(1,\,0,\,-1)^{\top}$.}
  \end{figure}
  Finally, let \(l\) and \(k\) be two cut lines starting at the boundary points
  \[
    \boldsymbol x_l=(0,\,\frac{1}{4},\,0)^{\top},\,\boldsymbol x_k=(0,\,\frac{1}{4}+10^{-4},\,0)^{\top}
    \;\text{with directions}\;
    \boldsymbol d_l=\frac{1}{\sqrt{2}}(-1,\,0,-1)^{\top},\,\boldsymbol d_k=\frac{1}{\sqrt{3}}(-1,\,-1,\,-1)^{\top}.
  \]
  We define their parametric forms
  \[
    l(s)=\boldsymbol x_l + s \boldsymbol d_l,\quad s\in[0,10^{-2}],
    \qquad
    k(s)=\boldsymbol x_k + s \boldsymbol d_k,\quad s\in[0,\frac{1}{4}],
  \]
  and the corresponding one-dimensional solution profiles
  \begin{equation}\label{eq:cutlines}
    u_l(s)=u(l(s)),
    \quad
    u_k(s)=u(k(s)).
  \end{equation}
  In Fig.~\ref{fig:fichera-cut} we plot \(u_l(s)\) and \(u_k(s)\) over \(s\).
  The plots demonstrate that the solution remains smooth near the goal point
  \(\boldsymbol{x}_c\), with the boundary layer being resolved by the underlying
  anisotropic \(hp\)-refined mesh. Further, the solution remains well-resolved
  downstream of \(\boldsymbol{x}_c\) due to the adaptive $hp$-refinement of
  regions that contribute to the goal functional.

\section{Conclusion}\label{ref:conclusion}
We have presented a goal-oriented adaptive framework for convection-dominated
transport problems based on a fully anisotropic space-time DG discretization
with independent \(h\)- and \(p\)-updates in each spatial direction and in time,
a directional DWR estimator yielding split indicators
\(\eta_{h,\,i}, \eta_{\tau}, i=1,\dots,d\), and an \(h\)- and
\(p\)-robust preconditioner for the resulting space-time systems. Our method
leverages the Dual Weighted Residual approach to control the error with respect
to user-defined target quantities and, by separating directional error
contributions, applies \(h\)- or \(p\)-refinement independently in each spatial
and temporal direction, thereby resolving thin boundary and interior layers and
geometric singularities, concentrating computational effort where the goal
functional is most sensitive, and maintaining solver robustness on highly
anisotropic meshes and for high polynomial degrees.

To address the resulting large and potentially ill-conditioned linear systems,
we proposed a preconditioner based
on~\cite{algoritmy,munch_stage-parallel_2023}, which shows good
robustness with respect to both \( h \)- and \( p \)-refinement and performs well
under strong anisotropies.

Numerical results for established benchmark problems in two and three dimensions
confirm the accuracy and reliability of the proposed method. The approach
consistently identifies and resolves boundary layers, interior layers, and
geometric singularities such as those arising in the Fichera corner. The
adaptive algorithm effectively concentrates resolution in regions that are most
relevant for the chosen goal functional, leading to significant error reduction
at moderate computational cost. For benchmarks with an analytical solution we
observed the expected exponential convergence of $hp$ methods. Anisotropic
refinement plays a critical role in resolving steep gradients along
characteristic directions and capturing local features without unnecessary
refinement elsewhere.

Overall, the results demonstrate that fully anisotropic \( hp \)-adaptivity
combined with goal-oriented error control yields accurate and computationally
efficient solutions for convection-dominated problems. Future work may focus on
parallel mesh refinement, automated tuning of the decision criterion between $h$
and $p$ refinement, and nonlinear problems.

\paragraph{Acknowledgments}
Bernhard Endtmayer acknowledges the support by the Cluster of Excellence PhoenixD (EXC 2122, Project ID 390833453).
Nils Margenberg acknowledges funding from the European Regional Development Fund (grant FEM Poer
II, ZS?2024/06/18815) under the European Union's Horizon Europe Research and
Innovation Program.
Computational resources (HPC cluster HSUper) have been provided by the project
hpc.bw, funded by dtec.bw - Digitalization and Technology Research Center of the
Bundeswehr. dtec.bw is funded by the European Union - NextGenerationEU.

  \FloatBarrier

\printbibliography
\end{document}